\documentclass{cedram-alco-altered}
\pdfoutput=1

\usepackage[final]{microtype}
\usepackage{mathptmx}
\DeclareMathAlphabet{\mathcal}{OMS}{cmsy}{m}{n}
\usepackage{graphicx}
\usepackage{amsmath}
\usepackage{mathrsfs}
\usepackage{amssymb}
\usepackage{tikz}
\usepackage{float}
\usepackage{ytableau}

\newenvironment{Case}[1]{\smallskip\noindent\textit{Case #1}\endgraf
   \nobreak\noindent\ignorespaces}{\endgraf\smallskip}

\newcommand{\opp}{^{\mathrm{o}}}
\newcommand{\cp}{\mathrm{cp}}
\newcommand{\leftright}{_{\mathsf{LR}}}
\newcommand{\Ini}{\mathrm{Ini}}

\newcommand{\lside}{_{\mathsf{L}}}
\newcommand{\rside}{_{\mathsf{R}}}
\newcommand{\lessthan}[2]{#1\,{\downarrow}\,#2}
\newcommand{\morethan}[2]{#1\,{\uparrow}\,#2}
\newcommand{\Lessthan}[2]{#1\,{\Downarrow}\,#2}
\newcommand{\Morethan}[2]{#1\,{\Uparrow}\,#2}
\newcommand{\perm}{\mathrm{perm}}

\newcommand{\row}{\mathrm{row}}
\newcommand{\col}{\mathrm{col}}

\newcommand{\lex}{\mathrm{lex}}
\newcommand{\citex}[2]{\textup{\!\cite[#1]{#2}\,}}
\usepackage{cyoung}
\def\ttab(#1){\text{\tiny\ytableaushort{#1}}}

\newcommand{\word}{\mathrm{word}}

\DeclareMathOperator{\CSTD}{\mathrm{CStd}}

\DeclareMathOperator{\STD}{\mathrm{Std}}
\DeclareMathOperator{\shape}{\mathrm{Shape}}
\DeclareMathOperator{\neb}{\mathrm{-neb}}

\DeclareMathOperator{\ind}{\mathrm{ind}}
\DeclareMathOperator{\SD}{\mathrm{SD}}
\DeclareMathOperator{\SA}{\mathrm{SA}}
\DeclareMathOperator{\WD}{\mathrm{WD}}
\DeclareMathOperator{\WA}{\mathrm{WA}}
\DeclareMathOperator{\D}{\mathrm{D}}
\DeclareMathOperator{\A}{\mathrm{A}}

\DeclareMathOperator{\RS}{\mathrm{RS}}

\DeclareMathOperator{\Tab}{\mathrm{Tab}}

\title{Type \(A\)-admissible cells are Kazhdan--Lusztig}
\author{\firstname{Van} \middlename{Minh} \lastname{Nguyen}}
\address{School of Mathematics and Statistics\\
University of Sydney \\
NSW 2006\\
Australia}
\email{Van.Nguyen@sydney.edu.au}

\def\exemname{Example}%

\begin{document}

\begin{abstract}
Admissible \(W\!\)-graphs were defined and combinatorially characterised by
Stembridge in \cite{stem:addwgraph}. The theory of admissible \(W\!\)-graphs
was motivated by the need to construct \(W\!\)-graphs for Kazhdan--Lusztig
cells, which play an important role in the representation theory of Hecke algebras,
without computing Kazhdan--Lusztig polynomials. In this paper, we shall show
that type \(A\)-admissible \(W\!\)-cells are Kazhdan--Lusztig as conjectured
by Stembridge in his original paper.

\keywords{Coxeter groups \and Hecke algebras \and \(W\!\)-graphs \and
Kazhdan--Lusztig polynomials \and cells}
\end{abstract}

\maketitle

\section{Introduction}
\label{intro}
Let \((W,S)\) be a Coxeter system and \(\mathcal{H}(W)\) its Hecke algebra over
\(\mathbb{Z}[q,q^{-1}]\), the ring of Laurent polynomials in the indeterminate \(q\).
We are interested in representations of \(W\) and \(\mathcal{H}(W)\) that can be described
by combinatorial objects, namely~\(W\!\)-graphs. In particular, we are interested in~\(W\!\)-graphs
corresponding to Kazhdan--Lusztig left cells.

In principle, when computing left cells one encounters the problem of having to compute a large number of
Kazhdan--Lusztig polynomials before any explicit description of their~\(W\!\)-graphs can be given.
In \cite{stem:addwgraph}, Stembridge introduced \textit{admissible\/} \(W\!\)-graphs; these can
be described combinatorially and can be constructed without calculating Kazhdan--Lusztig polynomials.
Moreover, the \(W\!\)-graphs corresponding to Kazhdan--Lusztig left cells are
admissible. Stembridge showed in~\cite{stem:finitewgraph} that for any given finite
\(W\) there are only finitely many stongly connected admissible~\(W\!\)-graphs. It
was conjectured by Stembridge that in type \(A\) all strongly connected admissible
\(W\)-graphs are isomorphic to Kazhdan--Lusztig left cells. In this paper we
complete the proof of Stembridge's conjecture.

We shall work with \textit{\(S\)-coloured graphs} (as defined in
Section~3 below), of which \(W\!\)-graphs are examples. These graphs
have both edges (bi-directional) and arcs (uni-directional). A
\textit{cell\/} in such a graph \(\Gamma\) is by definition
a strongly connected component of~\(\Gamma\), and a \textit{simple part\/}
of \(\Gamma\) is a connected component of the graph obtained from \(\Gamma\)
by removing all arcs and all edges of weight greater than~1.
A \textit{simple component\/} of \(\Gamma\) is the full subgraph of \(\Gamma\)
spanned by a simple part. If \(\Gamma\) is an admissible \(W\!\)-graph,
simple components of \(\Gamma\) are also called \textit{molecules\/}.

Admissible \(W\!\)-cells and admissible simple components are by definition cells
and simple components of admissible \(W\!\)-graphs.

In~\cite{ChuMic:typeAMol}, Chmutov established the first step towards
the proof of Stembridge's conjecture, showing that the simple part
of an admissible molecule of type \(A_{n-1}\) is isomorphic to the simple
part of a Kazhdan--Lusztig left cell. The proof made use of the
axiomatisation of dual equivalence graphs on standard tableaux generated
by dual Knuth equivalence relations, given in an earlier paper by
Assaf~\cite{assaf:dualequigraphs}. Our proof makes use of Chmutov's
result.

The proof of Proposition~\ref{monomolecularadmcellsareKL} furnishes an
algorithm for computation of \(W\!\)-graphs for left cells in type~\(A_{n-1}\)
(avoiding the computation of Kazhdan--Lusztig polynomials). This has been
implemented in Magma and checked for a variety of partitions \(\lambda\)
with \(n\le 16\), and module dimension up to 171600 (for \(\lambda=(5,5,3,3)\)).
The Magma code is available on request.

We organize the paper in the following sections. Section~\ref{sec:2} and
Section~\ref{sec:2'} deal
with the background on Coxeter groups and the corresponding Hecke algebras.
In Section~\ref{sec:3} the definition and properties of \(W\!\)-graphs are
recalled.
In Section~\ref{sec:5}, we recall the definitions of admissible \(W\!\)-graphs
and molecules and how these can be characterised combinatorially.
Section~\ref{sec:6} presents combinatorics of tableaux and the relationship
between Kazhdan--Lusztig left cells, dual Knuth equivalence classes and
admissible molecules. We introduce the
paired dual Knuth equivalence relation in Section~\ref{sec:8}. In
Section~\ref{sec:8ex}, we prove the first main result, namely that
admissible \(W\!\)-graphs in type \(A_{n-1}\) are ordered. The proof that
type \(A\)-admissible cells are isomorphic to Kazhdan--Lusztig left cells
is completed in Section~\ref{sec:9}.

\section{Coxeter groups}
\label{sec:2}
Let \((W,S)\) be a Coxeter system and \(l\) the length function on \(W\).
The Coxeter group \(W\) comes equipped with the left weak
order, the right weak order and the Bruhat order, respectively denoted by
\(\leqslant\lside\), \(\leqslant\rside\) and \(\leqslant\), and defined as follows.
\begin{defi}
\begin{enumerate}[label=(\textit{\roman*})]
\item The left weak order is the partial order generated by
the relations \(x\leqslant\lside y\) for all \(x,\,y\in W\)
with \(l(x)<l(y)\) and \(yx^{-1}\in S\).
\item The right weak order is the partial order generated by
the relations \(x\leqslant\rside y\) for all \(x\,y\in W\)
with \(l(x)<l(y)\) and \(x^{-1}y\in S\).
\item The Bruhat order is the partial order generated by
the relations \(x\leqslant y\) for all \(x,\,y\in W\)
with \(l(x)<l(y)\) and \(yx^{-1}\) conjugate to an element of~\(S\).
\end{enumerate}
\end{defi}
Observe that the weak orders are characterized by the property that
\(x\leqslant\rside xy\) and \(y\leqslant\lside xy\) whenever \(l(xy)=l(x)+l(y)\).

For each \(J\subseteq S\) let \(W_{J}\) be the (standard parabolic)
subgroup of \(W\) generated by~\(J\), and let \(D_{J}\) the set of
distinguished (or minimal) representatives of the left cosets of
\(W_{J}\) in~\(W\). Thus each \(w \in W\) has a unique factorization
\(w = du\) with \(d \in D_{J}\) and \(u \in W_{J}\), and \(l(du) = l(d) +l(u)\)
holds for all \(d \in D_{J}\) and \(u \in W_{J}\). It is easily seen
that \(D_J\) is an ideal of \((W,\leqslant\lside)\), in the sense that if
\(w\in D_J\) and \(v\in W\) with \(v\leqslant\lside w\) then \(v\in D_J\).

If \(W_J\) is finite then we denote the longest element of \(W_{J}\)
by~\(w_{J}\). By \cite[Lemma 2.2.1]{gecpfei:charHecke}, if \(W\) is finite then
\(D_{J} = \{d\in W\mid d\leqslant\lside d_J\}\), where \(d_{J}\) is
the unique element in \(D_{J} \cap w_{S}W_{J}\).

\section{Hecke algebras}
\label{sec:2'}

Let \(\mathcal{A} = \mathbb{Z}[q, q^{-1}]\), the ring of Laurent
polynomials with integer coefficients in the indeterminate \(q\),
and let \(\mathcal{A}^{+} = \mathbb{Z}[q]\).
The Hecke algebra of a Coxeter system \((W, S)\), denoted by
\(\mathcal{H}(W)\) or simply by \(\mathcal{H}\),
is an associative  \(\mathcal{A}\)-algebra with \(\mathcal{A}\)-basis
\(\{\,H_w\mid w\in W\,\}\) satisfying
\begin{align*}
H^{2}_{s}        &= 1 + (q -q^{-1})H_{s} \quad \text{for all \(s \in S\)},\\
H_{xy}&= H_{x}H_{y}   \quad
\text{for all \(x,\,y \in W\) with \(l(xy)=l(x)+l(y)\).}
\end{align*}
We let \(a \mapsto \overline{a}\) be the involutory automorphism of
\(\mathcal{A}=\mathbb{Z}[q,q^{-1}]\) defined by \(\overline{q}=q^{-1}\).
It is well known that this extends to an involutory automorphism of
\(\mathcal{H}\) satisfying
\begin{equation*}
\overline{H_{s}} = H_{s}^{-1} = H_{s} - (q - q^{-1}) \quad
\text{for all \(s \in S\)}.
\end{equation*}

If \(J\subseteq S\) then \(\mathcal{H}(W_{J})\), the Hecke algebra
associated with the Coxeter system \((W_{J},J)\), is isomorphic to
the subalgebra of \(\mathcal{H}(W)\) generated by
\(\{\,H_{s} \mid s \in J\,\}\). We shall identify \(\mathcal{H}(W_{J})\)
with this subalgebra.

\section{\textit{W-}graphs}
\label{sec:3}

Given a set \(S\), we define an \textit{\(S\)-coloured graph\/} to be
a triple \(\Gamma=(V, \mu, \tau)\) consisting of a set \(V\!\), a
function \(\mu\colon V \times V\to \mathbb{Z}\) and a function
\(\tau\) from \(V\) to \(\mathcal{P}(S)\), the power set of \(S\).
The elements of \(V\) are the \textit{vertices} of \(\Gamma\), and if
\(v\in V\) then \(\tau(v)\) is the \textit{colour\/} of the vertex.
To interpret \(\Gamma\) as a (directed) graph, we adopt the convention that
if \(v,\,u\in V\) then \((v,u)\) is an arc of \(\Gamma\) if and only if
\(\mu(u,v) \neq 0\) and \(\tau(u)\nsubseteq\tau(v)\), and
\(\{v,u\}\) is an edge of \(\Gamma\) if and only if \((v,u)\) and
\((u,v)\) are both arcs. We call \(\mu(u,v)\) the \textit{weight\/}
of the arc~\((v,u)\). An edge \(\{u,v\}\) is said to be \textit{symmetric}
if \(\mu(u,v)=\mu(v,u)\), and \textit{simple} if \(\mu(u,v)=\mu(v,u)=1\).

If \((W,S)\) is a Coxeter system, then a \textit{\(W\!\)-graph\/} is
an \(S\)-coloured graph \(\Gamma=(V, \mu, \tau)\) such that the free
\(\mathcal{A}\)-module with basis~\(V\) admits an \(\mathcal{H}\)-module
structure satisfying
\begin{equation}\label{wgraphdef}
     H_{s}v = \begin{cases}
              -q^{-1}v \quad &\text{if \(s \in \tau(v)\)}\\
              qv + \sum_{\{u \in V \mid s \in \tau(u)\}}\mu(u,v)u
              \quad &\text{if \(s \notin \tau(v)\)},
     \end{cases}
\end{equation}
for all \(s \in S\) and \(v \in V\!\).

We shall write \(M_\Gamma\) for the \(\mathcal{H}\)-module afforded by
the \(W\!\)-graph~\(\Gamma\) in the manner described above. Since \(M_\Gamma\)
is \(\mathcal{A}\)-free with basis~\(V\) it admits an
\(\mathcal{A}\)-semilinear involution \(\alpha \mapsto \overline{\alpha}\),
uniquely determined by the condition
that \(\overline v=v\) for all \(v\in V\). We call this the
\textit{bar involution\/} on \(M_\Gamma\). It is a consequence of
\eqref{wgraphdef} that \(\overline{h\alpha}=\overline{h}\overline{\alpha}\)
for all \(h\in \mathcal{H}\) and \(\alpha \in M_\Gamma\).

We shall sometimes write \(\Gamma(V)\) for the \(W\!\)-graph with vertex
set~\(V\!\), if the functions \(\mu\) and~\(\tau\) are clear from the context.

Following~\cite{kazlus:coxhecke}, define a preorder
\(\leqslant_{\Gamma}\) on \(V\) as follows: \(u \leqslant_{\Gamma} v\)
if there exists a sequence of vertices \(u = x_{0},x_{1},
\ldots, x_{m} = v\) such that \(\tau(x_{i-1}) \nsubseteq
\tau(x_{i})\) and \(\mu(x_{i-1},x_{i}) \neq 0\) for all \(i \in
[1,m]\). That is, \(u \leqslant_{\Gamma} v\) if there is a directed path
from \(v\) to \(u\) in~\(\Gamma\). Let \(\sim_{\Gamma}\) be
the equivalence relation determined by this preorder.
The equivalence classes with respect to \(\sim_{\Gamma}\)
are called the \textit{cells} of \(\Gamma\). That is, the cells are
the strongly connected components of the directed graph \(\Gamma\).
Each equivalence class,
regarded as a full subgraph of \(\Gamma\), is itself a \(W\!\)-graph,
with the \(\mu\) and \(\tau\) functions being the restrictions
of those for \(\Gamma\). The preorder \(\leqslant_{\Gamma}\) induces a partial
order on the set of cells: if \(\mathcal{C}\) and \(\mathcal{C}'\) are
cells, then \(\mathcal{C} \leqslant_{\Gamma} \mathcal{C}'\) if \(u \leqslant_{\Gamma} v\)
for some \(u \in \mathcal{C}\) and \(v \in \mathcal{C}'\).

It follows readily from \eqref{wgraphdef} that a subset of
\(V\) spans a \(\mathcal{H}(W)\)-submodule of \(M_{\Gamma}\) if
and only if it is \(\Gamma\)-closed, in the sense that for every vertex
\(v\) in the subset, each \(u\in V\) satisfying \(\mu(u,v)\ne 0\) and
\(\tau(u)\nsubseteq\tau(v)\) is also in the subset.
Thus \(U\subseteq V\) is a \(\Gamma\)-closed subset of \(V\) if and
only if \(U=\bigcup_{v\in U}\{\,u\in V\mid u\leqslant_{\Gamma}v\,\}\).
Clearly, a subset of \(V\) is \(\Gamma\)-closed if and only if it is the
union of cells that form an ideal with respect to the partial order
\(\leqslant_\Gamma\) on the set of cells.

Suppose that \(U\) is a \(\Gamma\)-closed subset of \(V\), and let
\(\Gamma(U)\) and \(\Gamma(V \setminus U)\) be the full subgraphs of
\(\Gamma\) induced by \(U\) and \(V \setminus U\), with edge weights
and vertex colours inherited from \(\Gamma\). Then \(\Gamma(U)\) and
\(\Gamma(V \setminus U)\) are themselves \(W\!\)-graphs, and
\[
M_{\Gamma(V \setminus U)} \cong M_{\Gamma(V)}/M_{\Gamma(U)}
\]
as \(\mathcal{H}(W)\)-modules.

It is clear that if \(J\subseteq S\) and \(\Gamma=(V,\mu,\tau)\) is a
\(W\!\)-graph then \(\Gamma_J=(V,\mu,\tau_J)\) is a \(W_J\)-graph, where the function
\(\tau_J\colon V\to\mathcal{P}(J)\) is given by \(\tau_J(v)=\tau(v)\cap J\).

We end this section by recalling the original Kazhdan--Lusztig \(W\!\)-graph
for the regular representation of \(\mathcal{H}(W)\).
For each \(w \in W\), define the sets
\begin{align*}
\mathcal{L}(w) &= \{s \in S \mid l(sw) < l(w)\}\\
\noalign{\nobreak\vskip-6 pt\hbox{and}\nobreak\vskip-6 pt}
\mathcal{R}(w) & = \{s \in S \mid l(ws) < l(w)\},
\end{align*}
the elements of which are called the left descents of \(w\) and the right descents
of \(w\), respectively. Kazhdan and Lusztig give a recursive procedure that defines
polynomials \(P_{y,w}\) whenever \(y,\,w \in W\) and \(y < w\). These polynomials satisfy
\(\deg P_{y,w}\leqslant\frac12(l(w)-l(y)-1)\), and \(\mu_{y,w}\) is defined to be
the leading coefficient of \(P_{y,w}\) if the degree is \(\frac12(l(w)-l(y)-1)\),
or 0 otherwise.

Define \(W\opp=\{\,w\opp\mid w\in W\,\}\) to be the group opposite to \(W\!\), and
observe that \((W\times W\opp\!,\,S\sqcup S\opp)\) is a Coxeter system.
Kazhdan and Lusztig show that if \(\mu\) and \(\tau\) are defined by the formulas
\begin{align*}
\mu(w,y)=\mu(y,w)       &= \begin{cases}
                            \mu_{y,w} &\quad \text{if \(y < w\)}\\
                            \mu_{w,y} &\quad \text{if \(w < y\)}
                            \end{cases}\\
\bar\tau(w)                 &= \mathcal{L}(w)\sqcup\mathcal{R}(w)\opp
\end{align*}
then \((W,\mu,\bar\tau)\) is a \((W\times W\opp)\)-graph.
Thus \(M=\mathcal{A}W\) may be regarded as an
\((\mathcal{H},\mathcal{H})\)-bimodule. Furthermore, the construction
produces an explicit \((\mathcal{H},\mathcal{H})\)-bimodule isomorphism
\(M\cong\mathcal{H}\).

It follows easily from the definition of \(\mu_{y,w}\) that \(\mu(y,w)\neq 0\) only
if \(l(w)-l(y)\) is odd; thus \((W,\mu,\bar\tau)\) is a bipartite graph. The non-negativity
of all coefficients of the Kazhdan--Lusztig polynomials, conjectured in
\cite{kazlus:coxhecke}, has been proved by Elias
and Williamson in~\cite{bengeogedi:howgesoerge}.

Since \(W\) and \(W\opp\) are standard parabolic subgroups of~\(W\times W\opp\), it follows
that \(\Gamma=(W,\mu,\tau)\) is a \(W\!\)-graph and \(\Gamma\opp=(W,\mu,\tau\opp)\) is a
\(W\opp\!\)-graph, where \(\tau\) and \(\tau\opp\) are defined by
\(\tau(w)=\mathcal{L}(w)\) and \(\tau\opp(w)=\mathcal{R}(w)\opp\), for all \(w\in W\).

In accordance with the theory described above, there are preorders on~\(W\)
determined by the \((W\times W\opp)\)-graph structure, the \(W\!\)-graph structure
and the \(W\opp\!\)-graph structure. We call these the \textit{two-sided preorder}
(denoted by \(\preceq\leftright\)), the \textit{left preorder\/} (\(\preceq\lside\))
and the \textit{right preorder} (\(\preceq\rside\)). The
corresponding cells are the \textit{two-sided cells}, the \textit{left cells\/}
and the \textit{right cells}.

\section{Admissible \textit{W-}graphs}
\label{sec:5}

Let \((W,S)\) be a Coxeter system, not necessarily finite. For \(s,\,t\in S\),
let \(m(s,t)\) be the order of \(st\) in \(W\!\). Thus \(\{s,t\}\) is a
bond in the Coxeter diagram if and only if \(m(s,t)>2\).

\begin{defi}\citex{Definition 2.1}{stem:addwgraph}\label{admissibleWg}
An \(S\)-coloured graph \(\Gamma = (V,\mu,\tau)\) is
\textit{admissible\/} if the following three conditions are satisfied:
\begin{enumerate}[label=(\textit{\roman*}),topsep=1 pt]
\item \(\mu(V\times V) \subseteq \mathbb{N}\);
\item \(\Gamma\) is symmetric, that is, \(\mu(u,v) = \mu(v,u)\) if \(\tau(u) \nsubseteq \tau(v)\)
and \(\tau(v) \nsubseteq \tau(u)\);
\item \(\Gamma\) has a bipartition.
\end{enumerate}
\end{defi}

\begin{rema}\label{KLgraphisadmissibl}
As we have seen in Sec.~\ref{sec:3}, the Kazhdan--Lusztig graph \(\Gamma_W=\Gamma(W,\emptyset)\) is
admissible. So its cells are admissible.
\end{rema}
Let \((W,S)\) be a braid finite Coxeter system. (That is, \(m(s,t)<\infty\) for all
\(s,t\in S\).)

\begin{defi}\citex{Definition 2.1}{stem:morewgraph}\label{compatibility}
An \(S\)-coloured graph \(\Gamma = (V,\mu,\tau)\) is said to satisfy the
\textit{\(W\!\)-Compatibility Rule\/} if for all \(u,\,v\in V\) with \(\mu(u,v)\ne 0\),
each \(i \in \tau(u)\setminus \tau(v)\) and each \(j\in\tau(v)\setminus\tau(u)\)
are joined by a bond in the Coxeter diagram of \(W\).
\end{defi}
By \cite[Proposition 4.1]{stem:addwgraph}, every \(W\!\)-graph satisfies the
\(W\!\)-Compatibility Rule.

\begin{defi}\citex{Definition 2.3}{stem:morewgraph}\label{simplicity}
An admissible \(S\)-coloured graph \(\Gamma = (V,\mu,\tau)\) satisfies the
\textit{\(W\!\)-Simplicity Rule\/} if for all \(u,\,v\in V\) with \(\mu(u,v) \neq 0\),
either \(\tau(v) \subsetneqq \tau(u)\) and \(\mu(v,u) = 0\)
or \(\tau(u)\) and \(\tau(v)\) are not comparable and \(\mu(u,v) = \mu(v,u) = 1\).
\end{defi}
The Simplicity Rule implies that if \(\mu(u,v) \neq 0\) and \(\mu(v,u) \neq 0\)
then \(\mu(u,v)=\mu(v,u)=1\). That is, all edges are simple. Furthermore if
\(\{u,v\}\) is an edge then \(\tau(u)\) and \(\tau(v)\) are not comparable,
so that there exist at least one \(i\in \tau(u)\setminus\tau(v)\) and
at least one \(j\in \tau(v)\setminus\tau(u)\). If the Compatibility Rule is
also satisfied, then \(\{i,j\}\) must be a bond in the Coxeter
diagram.

If \((W,S)\) is simply-laced then every \(W\!\)-graph with non-negative
integer edge weights satisfies the Simplicity Rule, even if it fails to be
admissible: see \cite[Remark 4.3]{stem:addwgraph}.

\begin{defi}\citex{Definition 2.4}{stem:morewgraph}\label{bonding}
An admissible \(S\)-coloured graph \(\Gamma = (V,\mu,\tau)\) satisfies the
\textit{\(W\!\)-Bonding Rule} if for all \(i,j\in S\) with \(m_{i,j}>2\),
the vertices \(v\) of \(\Gamma\) with \(i\in\tau(v)\) and \(j\notin\tau(v)\)
or \(i\notin\tau(v)\) and \(j\in\tau(v)\),
together with edges of \(\Gamma\) that include the label \(\{i,j\}\),
form a disjoint union of Dynkin diagrams of types \(A\), \(D\) or \(E\)
with Coxeter numbers that divide~\(m(i,j)\).
\end{defi}

\begin{rema}\label{molecularbondingrule}
In the case \(m(i,j)=3\), the \(W\!\)-Bonding Rule becomes the
\textit{\(W\!\)-Simply-Laced Bonding Rule}: for every vertex \(u\) such that
\(i\in\tau(u)\) and \(j\notin\tau(u)\), there exists a unique
adjacent vertex \(v\) such that \(j\in\tau(v)\) and \(i\notin\tau(v)\).
\end{rema}
By \cite[Proposition 4.4]{stem:addwgraph}, admissible \(W\!\)-graphs satisfy the
\(W\!\)-Bonding Rule.

Let \(\Gamma=(V,\mu,\tau)\) be an \(S\/\)-coloured graph. Let \(i,j\in S\) with
\(m(i,j) = p \geqslant 2\). Suppose that \(u,v\in V\) with \(i,j\notin \tau(u)\) and \(i,j\in\tau(v)\).
For \(2\leqslant r\leqslant p\), a directed path \((u,v_1,\ldots,v_{r-1},v)\) in \(\Gamma\)
is said to be alternating of type \((i,j)\)
if \(i\in \tau(v_k)\) and \(j\notin\tau(v_k)\) for odd \(k\) and
\(j\in\tau(v _k)\) and \(i\notin\tau(v_k)\) for even \(k\). Define
\begin{equation}\label{altsums}
N^r_{i,j}(\Gamma;u,v)=\sum_{v_1,\ldots,v_{r-1}}\mu(v,v_{r-1})\mu(v_{r-1}v_{r-2})\cdots \mu(v_2,v_1)\mu(v_1,u),
\end{equation}
where the sum extends over all paths \((u,v_1,\ldots,v_{r-1},v)\) that are
alternating of type \((i,j)\).

Note that if \(\Gamma\) is admissible then all terms in \eqref{altsums} are positive.

\begin{defi}\citex{Definition 2.9}{stem:morewgraph}\label{polygon}
An admissible \(S\/\)-coloured graph \(\Gamma = (V,\mu,\tau)\) satisfies the
\textit{\(W\/\)-Polygon Rule} if for all \(i,j\in S\) and all \(u, v\in V\) such
that \(i,j\in\tau(v)\setminus \tau(u)\), we have
\begin{equation*}
N_{i,j}^{r}(\Gamma;u,v) = N_{j,i}^{r}(\Gamma;u,v) \quad \text{for all
\(r\) such that \(2 \leqslant r \leqslant m(i,j)\)}.
\end{equation*}
\end{defi}
By \cite[Proposition 4.7]{stem:addwgraph}, all \(W\!\)-graphs with integer edge
weights satisfy the Polygon Rule.

The following result provides a necessary and sufficient condition for an admissible \(S\)-coloured graph
to be a \(W\!\)-graph.

\begin{theo}\citex{Theorem 4.9}{stem:addwgraph}\label{combinatorialCharacterisation}
An admissible \(S\)-coloured graph \(\Gamma = (V,\mu,\tau)\) is a \(W\!\)-graph
if and only if it satisfies the \(W\!\)-Compatibility Rule, the \(W\!\)-Simplicity Rule, the \(W\!\)-Bonding Rule
and the \(W\!\)-Polygon Rule.
\end{theo}
It is convenient to introduce a weakened version of the \(W\!\)-polygon
rule.
\begin{defi}\citex{Definition 2.9}{stem:morewgraph}\label{localpolygon}
An admissible \(S\)-coloured graph \(\Gamma = (V,\mu,\tau)\) satisfies the
\textit{\(W\!\)-Local Polygon Rule} if for all \(i,j\in S\), all \(r\) such that
\(2\leqslant r\leqslant m(i,j)\), and all \(u, v\) such that
\(i,j\in\tau(v)\setminus \tau(u)\), we have
\(N_{i,j}^{r}(\Gamma;u,v) = N_{j,i}^{r}(\Gamma;u,v)\)
under any of the following conditions:
\begin{enumerate}[label=(\textit{\roman*}),topsep=1 pt]
\item \(r=2\), and \(\tau(u)\setminus\tau(v)\ne\emptyset\);
\item \(r=3\), and there exist \(k,l\in\tau(u)\setminus\tau(v)\)
(not necessarily distinct) such that
\(\{k,i\}\) and \(\{j,l\}\) are not bonds in the Dynkin diagram of \(W\);
\item \(r\geqslant 4\), and there is \(k\in\tau(u)\setminus\tau(v)\) such that
\(\{k,i\}\) and \(\{j,k\}\) are not bonds in the Dynkin diagram of \(W\).
\end{enumerate}
\end{defi}

\begin{defi}\citex{Definition 3.3}{stem:morewgraph}\label{moleculargraph}
An admissible \(S\)-coloured graph is called a \textit{\(W\!\)-molecular graph} if it satisfies
the \(W\!\)-Compatibility Rule, the \(W\)-Simplicity Rule, the \(W\!\)-Bonding Rule
and \(W\!\)-Local Polygon Rules.
\end{defi}
A \textit{simple part\/} of an \(S\)-coloured graph \(\Gamma\) is a connected
component of the graph obtained by removing all arcs and all non-simple edges,
and a \textit{simple component\/} of \(\Gamma\) is the full subgraph spanned by a
simple part.

\begin{defi}\label{molecule}
A \textit{\(W\!\)-molecule} is a \(W\!\)-molecular graph that has only
one simple part.
\end{defi}

\begin{rema}\label{WgraphisWmoleculargraph}
If \(\Gamma\) is an admissible \(W\!\)-graph then its simple components are
\(W\!\)-molecules, by~\cite[Fact 3.1.]{stem:morewgraph}.
More generally, by~\cite[Fact 3.2.]{stem:morewgraph}, the full subgraph of
\(\Gamma\) induced by any union of simple parts is a \(W\!\)-molecular graph.
\end{rema}

If \(M=(V,\mu,\tau)\) is an \(S\)-coloured graph and \(J\subseteq S\) then
the \(W_J\)-restriction of \(M\) is defined to be the \(J\)-coloured graph
\(M{\downarrow}_{J} = (V,\underline{\mu},\underline{\tau})\)
where \(\underline{\tau}(v)=\tau(v)\cap J\) for all \(v\in V\) and
\[
\underline{\mu}(u,v)=
\begin{cases}
\mu(u,v)&\text{ if \(\underline{\tau}(u)\nsubseteq\underline{\tau}(v)\),}\\
0&\text{ if \(\underline{\tau}(u)\subseteq\underline{\tau}(v)\).}
\end{cases}
\]
It is easy to check that if \(M=(V,\mu,\tau)\) is a \(W\!\)-molecular graph
\(M{\downarrow}_{J}\) is a \(W_J\)-molecular graph. The \(W_J\)-molecules of
\(M{\downarrow}_J\) are called \(W_J\)-submolecules of \(M\).

\begin{prop}\citex{Proposition 2.7}{ChuMic:typeAMol}\label{arctransport}
Let \((W,S)\) be a Coxeter system and \(M=(V,\mu,\tau)\)
a \(W\!\)-molecular graph, and let \(J=\{r,s,t\}\subseteq S\) with
\(m(s,t)=3\) and \(r\notin\{s,t\}\). Suppose that \(v,v',u,u'\in V\),
and that \(\{v,v'\}\) and \(\{u,u'\}\) are simple edges with
\begin{align*}
\qquad&&\tau(v)\cap J &= \{s\},&\tau(u)\cap J &= \{s,r\},&&\qquad\\[-1.5pt]
\qquad&&\tau(v')\cap J &= \{t\},&\tau(u')\cap J &= \{t,r\}.&&\qquad
\end{align*}
Then \(\mu(u,v) = \mu(u',v')\).
\end{prop}

\section{Tableaux, left cells and admissible molecules of type \textit{A}}
\label{sec:6}

For the remainder of this paper we shall focus attention on Coxeter systems
of type~\(A\).

A sequence of nonnegative integers \(\lambda =
(\lambda_{1},\lambda_{2} \ldots, \lambda_{k})\) is called a \textit{composition}
of~\(n\) if \(\sum_{i=1}^k\lambda_i=n\). The \(\lambda_i\) are called the \textit{parts\/}
of \(\lambda\). We adopt the convention that \(\lambda_i=0\) for all \(i>k\).
A composition \(\lambda = (\lambda_1,\lambda_2,\ldots,\lambda_k)\) is called a
\textit{partition} of \(n\) if \(\lambda_{1} \ge \cdots \ge \lambda_{k}>0\).
We define \(C(n)\) and \(P(n)\) to be the sets
of all compositions of~\(n\) and all partitions of~\(n\), respectively.

For each
\(\lambda=(\lambda_{1},\ldots,\lambda_{k})\in C(n)\) we define
\[
[\lambda]=\{\,(i,j)\mid1\leqslant i\leqslant\lambda_{j}\text{ and }1\leqslant
j\leqslant k\,\},
\]
and refer to this as the Young diagram of~\(\lambda\). Pictorially
\([\lambda]\) is represented by a top-justified array of boxes
with \(\lambda_{j}\) boxes in the \(j\)-{th} column; the pair
\((i,j)\in[\lambda]\) corresponds to the \(i\)-{th} box in the
\(j\)-{th} column. Thus for us the Young diagram of \(\lambda =
(3,4,2)\) looks like this:
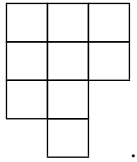
\[  \vcenter{\hbox{\begin{Young}
        &&\cr
        &&&\cr
        &\cr
      \end{Young}.}\vskip3 pt}
\]
If \(\lambda\in P(n)\) then \(\lambda^*\) denotes the \textit{conjugate} of \(\lambda\), defined
to be the partition whose diagram is the transpose of \([\lambda]\); that is,
\([\lambda^*]=\{(j,i) \mid (i,j)\in[\lambda]\}\).

Let \(\lambda\in P(n)\). If \((i,j)\in[\lambda]\) and
\([\lambda]\setminus\{(i,j)\}\) is still the Young diagram of a partition,
we say that the box \((i,j)\) is \textit{\(\lambda\)-removable}.
Similarly, if \((i,j)\notin[\lambda]\) and
\([\lambda]\cup\{(i,j)\}\) is again the Young diagram of a partition,
we say that the box \((i,j)\) is \textit{\(\lambda\)-addable}.

If \(\lambda\in C(n)\) then a \(\lambda\)\textit{-tableau} is a bijection
\(t\colon[\lambda] \rightarrow\mathcal{T}\), where \(\mathcal{T}\) is a
totally ordered set with~\(n\) elements. We call \(\mathcal{T}\) the
\textit{target\/} of~\(t\). In this paper the target will always be an
interval \([m+1,m+n]\), with \(m=0\) unless otherwise specified.
The composition \(\lambda\) is called the
\textit{shape\/} of \(t\), and we write \(\lambda=\shape(t)\). For each
\(i\in[1,n]\) we define \(\row_t(i)\) and \(\col_t(i)\) to be the row index
and column index of \(i\)~in~\(t\) (so that
\(t^{-1}(i)=(\row_t(i),\col_t(i)\))). We define \(\Tab_m(\lambda)\)
to be the set of all \(\lambda\)-tableaux with target
\(\mathcal{T}=[m+1,m+n]\), and \(\Tab(\lambda)=\Tab_0(\lambda)\).
If \(h\in\mathbb{Z}\) and \(t\in\Tab_m(\lambda)\) then we define
\(t+h\in\Tab_{m+h}(\lambda)\) to be the tableau obtained by adding
\(h\) to all entries of~\(t\).

We define \(\tau_\lambda\in \Tab(\lambda)\) to be the
specific \(\lambda\)-tableau given by
\(
\tau_\lambda(i,j)=i+\sum_{h=1}^{j-1}\lambda_h
\)
for all \((i,j)\in[\lambda]\). That is, in \(\tau_\lambda\) the numbers
\(1,\,2,\,\dots,\,\lambda_1\) fill the first column of
\([\lambda]\) in order from top to bottom, then the numbers
\(\lambda_1+1,\,\lambda_1+2,\,\dots,\,\lambda_1+\lambda_2\)
similarly fill the second column, and so on. If \(\lambda\in P(n)\) then
we also define \(\tau^{\lambda}\) to be the \(\lambda\)-tableau that
is the transpose of \(\tau_{\lambda^*}\).
Whenever \(\lambda\in P(n)\) and \(t\in\Tab_m(\lambda)\) we define
\(t^*\in\Tab_m(\lambda^*)\) to be the transpose of~\(t\).

Let \(\lambda\in C(n)\) and \(t\in\Tab(\lambda)\). We say that
\(t\) is \textit{column standard\/} if the entries increase down
each column. That is, \(t\) is column standard if \(t(i,j)<t(i+1,j)\)
whenever \((i,j)\) and \((i+1,j)\) are both in~\([\lambda]\). We
define \(\CSTD(\lambda)\) to be the set of all column standard
\(\lambda\)-tableaux. In the
case \(\lambda\in P(n)\) we say that \(t\) is \textit{row standard\/} if
its transpose is column standard (so that \(t(i,j) < t(i,j+1)\) whenever
\((i,j)\) and \((i,j+1)\) are both in~\([\lambda]\)), and we say that \(t\)
is \textit{standard\/} if it is both row standard and column standard.
For each \(\lambda\in P(n)\) we define \(\STD(\lambda)\)
to be the set of all standard \(\lambda\)-tableaux. We also define
\(\STD(n)=\bigcup_{\lambda\in P(n)}\STD(\lambda)\).

Let \(W_n\) be the symmetric group on the set \(\{1,2,\ldots,n\}\), and let
\(S_n=\{s_i\mid i \in [1,n-1]\}\), where \(s_i\) is the transposition \((i,i+1)\). Then
\((W_n,S_n)\) is a Coxeter system of type \(A_{n-1}\). If
\(1\leqslant h\leqslant k\leqslant n\) then we write \(W_{[h,k]}\) for the standard
parabolic subgroup of \(W_n\) generated by \(\{\,s_i\mid i\in[h,k-1]\,\}\).
We adopt a left operator convention for permutations, writing \(wi\) for the image
of \(i\) under the permutation~\(w\).

It is clear that for any fixed composition \(\lambda\in C(n)\) the group
\(W_{n}\) acts on \(\Tab(\lambda)\), via
\((wt)(i,j) = w(t(i,j))\) for all \((i,j)\in[\lambda]\), for all
\(\lambda\)-tableaux \(t\) and all \(w \in W_{n}\). Moreover, the
map from \(W_n\) to \(\Tab(\lambda)\) defined by \(w\mapsto
w\tau_\lambda\) for all \(w\in W_n\) is bijective. We define the map
\(\perm\colon\Tab(\lambda) \mapsto W_n\) to be the inverse of
\(w\mapsto w\tau_\lambda\), and use this to transfer the left weak order
and the Bruhat order from \(W_n\) to \(\Tab(\lambda)\). Thus if
\(t_1\) and \(t_2\) are arbitrary \(\lambda\)-tableaux, we write
\(t_1\leqslant\lside t_2\) if and only if \(\perm(t_1)\leqslant\lside\perm(t_2)\), and
\(t_1\leqslant t_2\) if and only if \(\perm(t_1)\leqslant\perm(t_2)\).
Similarly, we define the length of \(t\in\Tab(\lambda)\)
by \(l(t)=l(\perm(t))\).

\begin{rema}\label{readingword}
If \(\lambda\in C(n)\) and \(t\in\Tab(\lambda)\) then the \textit{reading word}
of \(t\) is defined to be the sequence \(b_1,\ldots,b_n\) obtained by
concatenating the columns of \(t\) in order from left to right, with the
entries of each column read from bottom to top. This produces a bijection
\(\Tab(\lambda)\mapsto W_n\) that maps each \(t\) to the permutation
\(\word(t)\) given by \(i\mapsto b_i\) for all \(i\in\{1,\ldots,n\}\). It is
obvious that \(\perm(t)=\word(t)w_\lambda^{-1}\), where
\(w_\lambda=\word(\tau_{\lambda})\).
\end{rema}
Given \(\lambda\in C(n)\) we define \(J_\lambda\) to be the subset
of \(S\) consisting of those \(s_i\) such that \(i\) and \(i+1\) lie in the same column
of~\(\tau_\lambda\), and \(W_\lambda\) to be the standard parabolic subgroup
of \(W_n\) generated by~\(J_\lambda\). Note that the longest element of
\(W_\lambda\) is the element \(w_{\lambda}=\word(\tau_{\lambda})\) defined in
Remark~\ref{readingword} above. We write
\(D_{\lambda}\) for the set of minimal length representatives of the left
cosets of \(W_{\lambda}\) in \(W_n\). Since \(l(ds_i)>l(d)\) if and
only if \(di<d(i+1)\), it follows that
\(D_{\lambda}=\{\,d\in W_n \mid di<d(i+1) \text{ whenever
\(s_i\in W_{\lambda}\)}\,\}\), and the set of column standard
\(\lambda\)-tableaux is precisely
\(\{\,d\tau_{\lambda} \mid d\in D_{\lambda}\,\}\).

We shall also need to work with tableaux defined on skew diagrams.

\begin{defi}\label{skewpartition}
A \textit{skew partition} of \(n\) is an ordered pair \((\lambda,\mu)\),
denoted by \(\lambda/\mu\), such that \(\lambda\in P(m+n)\) and \(\mu\in P(m)\)
for some \(m\geqslant 0\), and \(\lambda_i\geqslant\mu_i\) for all \(i\).
We write \(\lambda/\mu\vdash n\) to mean that \(\lambda/\mu\) is a skew
partition of \(n\). In the case \(m=0\) we identify \(\lambda/\mu\) with \(\lambda\),
and say that \(\lambda/\mu\) is a \textit{normal tableau}.
\end{defi}

\begin{defi}\label{skewdiagram}
The \textit{skew diagram} \([\lambda/\mu]\) corresponding to a skew
partition \(\lambda/\mu\) is defined to be the complement of \([\mu]\)
in \([\lambda]\):
\begin{equation*}
[\lambda/\mu]=\{(i,j)\mid (i,j)\in[\lambda] \text{ and } (i,j)\notin[\mu]\}.
\end{equation*}
\end{defi}

\begin{defi}\label{skewtableau}
A \textit{skew tableau of shape \(\lambda/\mu\)}, or \((\lambda/\mu)\)-tableau,
where \(\lambda/\mu\) is a skew partition of \(n\), is a bijective map
\(t\colon[\lambda/\mu] \rightarrow\mathcal{T}\), where \(\mathcal{T}\) is
a totally ordered set with~\(n\) elements. We write \(\Tab_m(\lambda/\mu)\)
for the set of all \((\lambda/\mu)\)-tableaux for which the target set
\(\mathcal{T}\) is the interval \([m+1,m+n]\). We shall omit the subscript
\(m\) if \(m=0\).
\end{defi}

Let \(\lambda/\mu\) be a skew partition of \(n\). We define
\(\tau_{\lambda/\mu}\in \Tab(\lambda/\mu)\)
by
\begin{equation}\label{toptab}
\tau_{\lambda/\mu}(i,j)=i-\mu_j+\sum_{h=1}^{j-1}(\lambda_h-\mu_h)
\end{equation}
for all \((i,j)\in[\lambda/\mu]\), and define
\(\tau^{\lambda/\mu}\in \Tab(\lambda/\mu)\) to be the the transpose of
\(\tau_{\smash{\lambda^*/\mu^*}}\).

If \(\lambda/\mu\vdash n\) and \(m\in\mathbb{Z}\) then \(W_{[m+1,m+n]}\) acts
naturally on \(\Tab_m(\lambda/\mu)\), and we can define
\(\perm\colon\Tab_m(\lambda/\mu)\to W_{[m+1,m+n]}\) and use it
to transfer the Bruhat order and the left weak order from \(W_{[m+1,m+n]}\) to
\(\Tab_m(\lambda/\mu)\) in exactly the same way as above.

All of our notation and terminology for partitions and Young tableaux
extends naturally to skew partitions and tableaux, and will be used
without further comment.

Let \(\lambda \in C(n)\) and \(t\) a column standard
\(\lambda\)-tableau. For each \(m\in\mathbb{Z}\) we define \(\Lessthan tm\)
to be the tableau obtained by removing from \(t\) all boxes with entries
greater than~\(m\).
Thus if \(\mu=\shape(\Lessthan tm)\) then \(\mu\in C(m)\) and
\([\mu]=\{\,b\in[\lambda]\mid t(b)\leqslant m\,\}\), and
\(\Lessthan tm\colon[\mu]\to [1,m]\) is the restriction of~\(t\).
It is clear that \(\Lessthan tm\) is column standard. Moreover, if
\(\lambda\in P(n)\) and \(t\in\STD(\lambda)\) then
\(\mu\in P(m)\) and \(\Lessthan tm\in\STD(\mu)\).

Similarly, if \(\lambda\in P(n)\) and \(t\in\STD(\lambda)\) then
for each \(m\in\mathbb{Z}\) we define \(\morethan tm\) to be the skew tableau
obtained by removing from \(t\) all boxes with entries
less than or equal to \(m\). Observe that
\(\{\,b\in [\lambda]\mid t(b)\leqslant m\,\}\) is the Young diagram of
a partition \(\nu\in P(n)\), and \(\lambda/\nu\) is
a skew partition of \(n-m\). Clearly \(\morethan tm\) is the restriction of~\(t\)
to \([\lambda/\nu]\), and \(\morethan tm\in \STD_m(\lambda/\nu)\).

We also define \(\lessthan tm=\Lessthan t(m-1)\) and \(\Morethan tm=\morethan t(m-1)\).

The \textit{dominance order} is defined on \(C(n)\) as follows.

\begin{defi}\label{Dominance Order}
Let \(\lambda,\,\mu \in C(n) \). We say that \(\lambda\)
\textit{dominates\/} \(\mu\), and write \(\lambda \geqslant \mu \) or
\(\mu \leqslant \lambda\), if \(\sum_{i=1}^{k} \lambda_i \leqslant
\sum_{i=1}^{k} \mu_i\) for each positive integer~\(k\).
\end{defi}
The \textit{lexicographic order} on compositions is defined as follows.

\begin{defi}\label{Lexicographic Order}
Let \(\lambda,\,\mu \in C(n) \). We write \(\lambda >_{\lex} \mu \) (or
\(\mu <_{\lex} \lambda\)) if there exists a positive integer~\(k\) such
that \(\lambda_k<\mu_k\) and \(\lambda_i=\mu_i\) for all \(i<k\).
We say that \(\lambda\) \textit{leads\/} \(\mu\), and write
\(\lambda \geqslant_{\lex} \mu \), if \(\lambda=\mu\) or \(\lambda >_{\lex} \mu \).
\end{defi}
It is clear that the lexicographic order is a refinement of the dominance
order.
\begin{prop}\label{domimplieslex}
If \(\lambda,\mu\in C(n)\) with \(\lambda \geqslant \mu\), then \(\lambda\geqslant_{\lex}\mu\).
\end{prop}
For a fixed \(\lambda\in C(n)\)  the \textit{dominance order
on \(\CSTD(\lambda)\)} is defined as follows.

\begin{defi}\label{Dominance Order Tableaux}
Let \(u\) and \(t\) be column standard \(\lambda\)-tableaux. We say that
\(t\) dominates \(u\) if \(\shape(\Lessthan tm) \geqslant \shape(\Lessthan um)\)
for all \(m\in [1,n]\).
\end{defi}

\begin{rema}\label{resdomrem}
Let \(\lambda\in C(n)\) and let \(u,\,t\in \CSTD(\lambda)\) with \(u\ne t\). Since
\(\Lessthan u0 = \Lessthan t0\) and \(\Lessthan un \ne \Lessthan tn\), we can
choose \(i\in[0,n-1]\) with \(\Lessthan ui = \Lessthan ti\) and
\(\Lessthan u{(i+1)}\neq\Lessthan t{(i+1)}\).
Let \(\mu=\shape(\Lessthan u{(i+1)})\) and \(\lambda=\shape(\Lessthan t{(i+1)})\),
and let \(k=\col_u(i+1)\) and \(l=\col_t(i+1)\).
Then \(k\ne l\), and \(\mu_j=\lambda_j\) for all \(j<m=\min(k,l)\). Furthermore,
\(\mu_m=\lambda_m+1\) if \(m=k\), and \(\lambda_m=\mu_m+1\) if \(m=l\). Thus
\(\lambda>_\lex\mu\) if and only if \(l>k\).

Now suppose that \(t\) dominates~\(u\). Since \(\Lessthan ui = \Lessthan ti\)
we have \(\shape(\Lessthan um)=\shape(\Lessthan tm)\) for all \(m\leq i\),
and by Definition~\ref{Dominance Order Tableaux} we must have
\(\shape(\Lessthan t{(i+1)}) \geqslant \shape(\Lessthan u{(i+1)})\). That
is, \(\lambda\geqslant\mu\). By~\ref{domimplieslex} it
follows that \(\lambda\geqslant_\lex\mu\), and so
\(\col_u(i+1)=k<l=\col_t(i+1)\).
\end{rema}

The following theorem shows that the dominance order on \(\CSTD(\lambda)\)
is the restriction of the Bruhat order on~\(\Tab(\lambda)\).
That is, if \(u,\,t\in\CSTD(\lambda)\) then \(t\) dominates \(u\)
if and only if \(t\geqslant u\).

\begin{theo}\label{equidombruhat}
Let \(\lambda\in C(n)\), and let \(u\) and \(t\) be column standard
\(\lambda\)-tableaux. Then \(t\) dominates \(u\) if and only if
\(\perm(t)\geqslant\perm(u)\).
\end{theo}

\begin{proof}
This is exactly \cite[Theorem 3.8]{math:heckeA}, except that
we use columns where \cite{math:heckeA} uses rows.
\end{proof}
Let \(\lambda=(\lambda_1,\ldots,\lambda_k)\in C(n)\). For each
\(t\in\Tab(\lambda)\), we define \(\cp(t)\) to be the composition of
the number \(\sum_{i=1}^{k}i\lambda_i\) given by \(\cp(t)_i=\col_t(n+1-i)\),
the column index of \(n+1-i\) in \(t\). We can now define
the \textit{lexicographic order} on \(\CSTD(\lambda)\), a total order
that refines the Bruhat order.

\begin{defi}\label{Lexicographic Order Tableaux}
Let \(\lambda\) be a composition of \(n\) and let \(u\) and \(t\) be
column standard \(\lambda\)-tableaux.
We say that \(t\) leads \(u\), and write \(t\geqslant_{\lex}u\), if
\(\cp(t)\geqslant_{\lex}\cp(u)\).
\end{defi}

\begin{rema}\label{LexTableauEquiv}
It is immediate from Definitions~\ref{Lexicographic Order} and \ref{Lexicographic Order Tableaux}
that if \(u,\,t\in\CSTD(\lambda)\) then \(t>_{\lex}u\) if and only if there exists \(l\in[1,n]\)
such that \(\col_t(l)<\col_u(l)\) and \(\col_t(i)=\col_u(i)\) for all \(i\in[l+1,n]\).
Since \(u\) and \(t\) are column standard and of the same shape, the latter condition is
equivalent to \(\morethan tl =\morethan ul\).
\end{rema}
\begin{lemm}\label{bruhatimplieslextableau}
Let \(\lambda\in C(n)\), and let \(u,\,t\in\Tab(\lambda)\). If
\(t\geqslant u\) then \(t\geqslant_{\lex} u\).
\end{lemm}

\begin{proof}
By Theorem~\ref{equidombruhat} and the definition of the Bruhat order, it
suffices to show that if \(u=(i,j)t\) for some \(i,j\in[1,n]\), then \(t>u\)
implies \(t>_{\lex}u\). Without loss of generality we may assume
that~\(j>i\), and then \(t>u\) means that \(\col_t(j)<\col_t(i)=\col_u(j)\).
Since \(\{\,k\mid\col_t(k)\ne\col_u(k)\,\}=\{i,j\}\), and \(j\) is the maximum
element of this set, it follows from Remark~\ref{LexTableauEquiv} that
\(t>_{\lex}u\), as required.
\end{proof}

\begin{coro}\label{domimplieslextableau}
Let \(\lambda\in C(n)\), and let \(u,\,t\in\CSTD(\lambda)\). If
\(t\) dominates \(u\) then \(t>_{\lex} u\).
\end{coro}
Let \(\lambda\in P(n)\). For each \(t\in\STD(\lambda)\) we define the following
subsets of \([1,n-1]\):
\begin{align*}
\SA(t) &= \{i\in[1,n-1] \mid \row_t(i) > \row_t(i+1)\,\},\\
\SD(t) &= \{i\in[1,n-1] \mid \col_t(i) > \col_t(i+1)\},\\
\WA(t) &= \{i\in[1,n-1] \mid \row_t(i)=\row_t(i+1)\},\\
\WD(t) &= \{i\in[1,n-1] \mid \col_t(i)=\col_t(i+1)\}.
\end{align*}

\begin{rema}\label{equalsubsetsS}
It is easily checked that
\(i\in\SA(t)\) if and only if \(s_it\in\STD(\lambda)\) and \(s_it > t\),
while \(i\in\SD(t)\) if and only if \(s_it < t\) (which implies that
\(s_it\in\STD(\lambda)\)).
Note also that if \(w=\perm(t)\) then \(i\in\D(t)\) if
and only if \(s_i\in \mathcal{L}(ww_{\lambda})\); this is proved
in~\cite[Lemma 5.2]{nguyen:wgideals2}.
\end{rema}

\begin{rema}\label{minimal-tableau}
It is clear that if \(\lambda/\mu\vdash n\) and
\(m\in\mathbb{Z}\) then \(m+\tau_{\lambda/\mu}\) is the unique minimal element of
\(\STD_m(\lambda/\mu)\) with respect to the Bruhat order and the left weak order.
Accordingly, we call \(m+\tau_{\lambda/\mu}\) the \textit{minimal element
of \(\STD_m(\lambda/\mu)\)}. It is easily shown that if \(t\in\STD_m(\lambda/\mu)\)
then \(t=m+\tau_{\lambda/\mu}\) if and only if \(\SD(t)=\emptyset\).
That is, \(t\) is minimal if and only if \(\D(t)=\WD(t)\).
\end{rema}
For technical reasons it is convenient to make the following definition.

\begin{defi}\label{m-critical}
Let \(\lambda/\mu \vdash n > 1\) and \(m\in\mathbb{Z}\). Let \(i\) be minimal such that
\(\lambda_i > \mu_i\), and assume that \(\lambda_{i+1}>\mu_{i+1}\). The
\(m\)-critical tableau of shape \(\lambda/\mu\) is the tableau
\(t\in\STD_{m-1}(\lambda/\mu)\) such that \(\col_t(m)=i\) and \(\col_t(m+1)=i+1\),
and  \(\morethan t {(m+1)}\) is the minimal tableau of its shape.
\end{defi}
If \(t\) is \(m\)-critical then, with \(i\) as in the definition,
\(\col_t(m+2)=i\) if and only if \(\lambda_i - \mu_i > 1\).

\begin{rema}\label{critical-tableau}
Let \(\lambda\in P(n)\) and \(m\in\mathbb{Z}\), and let \(t\in\STD(\lambda)\)
satisfy \(\col_t(m+1)=\col_t(m)+1\). We claim that
\(\Morethan tm\) is \(m\)-critical if and only
if the following two conditions both hold:
\begin{enumerate}[label=\arabic*),topsep=1 pt]
\item either \(\col_t(m)=\col_t(m+2)\) or \(m+1\notin\SD(t)\),
\item every \(j\in\D(t)\) with \(j>m+1\) is in \(\WD(t)\).
\end{enumerate}

Let \(\shape(\Morethan tm)=\lambda/\mu\), and put \(i=\col_t(m)\). Note
that since \(m+1\) is in column \(i+1\) of \(\Morethan tm\), it
follows that \(\lambda_{i+1}>\mu_{i+1}\).

Given that \(\col_t(m+1)=\col_t(m)+1\), the second alternative in condition (1) is
equivalent to \(\col_t(m)+1\leqslant\col_t(m+2)\). Hence condition (1) is
equivalent to \(\col_t(m)\leqslant\col_t(m+2)\). But by Remark~\ref{minimal-tableau},
condition (2) holds if and only if \(\morethan{t}{(m+1)}\) is minimal, which
in turn is equivalent to
\(\col_t(m+2)\leqslant\col_t(m+3)\leqslant \, \cdots \, \leqslant \col_t(n)\).
So (1) and (2) both hold if and only if \(\morethan{t}{(m+1)}\) is minimal
and \(\col_t(j)\geqslant \col_t(m)\) for all \(j\geqslant m\).

Since \(\col_t(m+1)=i+1\), it follows from the definition that
\(\Morethan tm\) is \(m\)-critical if and only if \(\morethan{t}{(m+1)}\)
is minimal and \(i=\col_t(m)\) is equal to \(\min\{\,j\mid \lambda_j>\mu_j\,\}\).
But this last condition holds if and only if \(m\) is in the first nonempty
column of \(\Morethan tm\), and since this holds if and only if
\(\col_t(j)\geqslant \col_t(m)\) for all \(j\geqslant m\), the claim is
established.
\end{rema}
Recall that if \(w\in W_{n}\) then applying the Robinson--Schensted algorithm
to the sequence \((w1,w2,\ldots,wn)\) produces a pair \(\RS(w) = (P(w),Q(w))\),
where \(P(w),\,Q(w)\in \STD(\lambda)\) for some \(\lambda\in P(n)\). Details
of the algorithm can be found (for example) in~\cite[Section 3.1]{sag:sym}.
The first component of \(\RS(w)\) is called the \textit{insertion\/}
tableau and the second component is called the \textit{recording\/} tableau.

The following theorem is proved, for example, in \cite[Theorem 3.1.1]{sag:sym}.

\begin{theo}\label{rs}
The Robinson--Schensted map is a bijection from \(W_{n}\) to
\(\bigcup_{\lambda\in P(n)}\STD(\lambda)^2\).
\end{theo}
The following property of the Robinson--Schensted map is also proved, for example,
in \cite[Theorem 3.6.6]{sag:sym}.

\begin{theo}\label{rswinv}
Let \(w\in W_n\). If \(RS(w) = (t,x)\) then \(RS(w^{-1}) = (x,t)\).
\end{theo}
The following lemma will be used below in the discussion of dual Knuth
equivalence classes.

\begin{lemm}\citex{Lemma 6.3}{nguyen:wgideals2}\label{rstlambda}
Let \(\lambda\in\mathcal P(n)\) and let \(w\in W_n\). Then
\(RS(w)=(t,\tau_\lambda)\) for some \(t\in\STD(\lambda)\) if and
only if \(w=vw_{\lambda}\) for some \(v\in W_n\) such that
\(v\tau_\lambda\in\STD(\lambda)\). When these conditions hold,
\(t=v\tau_\lambda\).
\end{lemm}

\begin{defi}\label{dualKnuthequivalencerelation}
The \textit{dual Knuth equivalence relation} is the equivalence relation \(\approx\)
on \(W_n\) generated by the requirements that for all \(x\in W_n\) and
\(k \in [1,n-2]\),
\begin{enumerate}[label=\arabic*),topsep=1 pt]
\item \(x\approx s_{k+1}x\) whenever \(\mathcal{L}(x)\cap\{s_{k},s_{k+1}\}=\{s_{k}\}\)
and \(\mathcal{L}(s_{k+1}x)\cap\{s_k,s_{k+1}\}=\{s_{k+1}\}\),
\item \(x\approx s_kx\) whenever \(\mathcal{L}(x)\cap\{s_{k},s_{k+1}\} = \{s_{k+1}\}\)
and \(\mathcal{L}(s_kx)\cap\{s_k,s_{k+1}\}=\{s_{k}\}\).
\end{enumerate}
\end{defi}
The relations 1) and 2) above are known as the dual Knuth relations of
the first kind and second kind, respectively.

\begin{rema}\label{altDualKnuth}
It is not hard to check that 1) and 2) above can be combined to give an
alternative formulation of Definition~\ref{dualKnuthequivalencerelation},
as follows: \(\approx\) is the equivalence relation on \(W_n\) generated by
the requirement that \(x\approx sx\) for all \(x\in W_n\) and \(s\in S_n\)
such that \(x<sx\) and \(\mathcal{L}(x)\nsubseteq\mathcal{L}(sx)\).
In \cite{kazlus:coxhecke} Kazhdan and Lusztig show that whenever this
holds then \(x\) and \(sx\) are joined by a simple
edge in the Kazhdan--Lusztig \(W\)-graph \(\Gamma=\Gamma(W_n)\). Furthermore,
they show that the dual Knuth equivalence classes coincide with the
left cells in \(\Gamma(W_n)\).
\end{rema}
The following result is well-known.

\begin{theo}\citex{Theorem 3.6.10}{sag:sym}\label{knuth}
Let \(x,y \in W_n\). Then \(x\approx y\) if and only if \(Q(x)=Q(y)\).
\end{theo}
Let \(\lambda\in P(n)\), and for each \(t\in\STD(\lambda)\)
define \(C(t)=\{\,w\in W_n \mid Q(w)=t\,\}\). Theorem~\ref{knuth} says
that these sets are the dual Knuth equivalence classes in~\(W_n\). It follows
from Lemma~\ref{rstlambda} that
\(C(\tau_\lambda)=\{\,vw_{\lambda} \mid v\tau_\lambda\in\STD(\lambda)\,\}
=\{\,\perm(t)w_{\lambda}\mid t\in\STD(\lambda)\,\}
=\{\,\word(t)\mid t\in\STD(\lambda)\,\}\).

Let \(t,\,u\in\STD(\lambda)\), and suppose that
\(t = s_ku\) for some \(k\in[2,n-1]\). By Remark~\ref{equalsubsetsS}
above, if \(x=\word(u)\) then \(\mathcal{L}(x)\cap\{s_{k-1},s_k\}=\{s_{k-1}\}\)
and \(\mathcal{L}(s_kx)\cap\{s_{k-1},s_k\}=\{s_k\}\) if and only
if \(\D(u)\cap\{k-1,k\}=\{k-1\}\) and \(\D(t)\cap\{k-1,k\}=\{k\}\).
Under these circumstances we write \(u\to^{*1} t\), and say that there
is a \textit{dual Knuth move of the first kind\/} from \(u\) to~\(t\).
Similarly, if \(t=s_ku\) for some \(k\in[1,n-2]\) such that
\(\D(u)\cap\{k,k+1\}=\{k+1\}\) and \( \D(t)\cap\{k,k+1\}=\{k\}\) then we
write \(u\to^{*2}t\), and say that there is a \textit{dual Knuth
move of the second kind\/} from \(u\)~to~\(t\).
Since \(C(\tau_\lambda)\) is a single dual Knuth equivalence class, any
standard tableau of shape~\(\lambda\) can be transformed into any other
by a sequence of dual Knuth moves or their inverses.

We call the integer \(k\) above the \textit{index\/} of the corresponding
dual Knuth move, and denote it by \(\ind(u,t)\).

\begin{rema}
Dual Knuth moves are also defined for standard skew tableaux; the definitions
are exactly the same as for tableaux of normal shape.
If \(\lambda/\mu\vdash n\) and \(u,\,t\in\STD(\lambda/\mu)\) then we
write \(u\approx t\) if and only if \(u\) and \(t\) are related by a
sequence of dual Knuth moves.
\end{rema}
\begin{defi}\label{approxm-in-W}
For each \(J\subseteq S_n\) let \(\approx_J\) be the
equivalence relation on \(W_n\) generated by the requirement that
\(x\approx_J sx\) for all \(s\in J\) and \(x\in W_n\)
such that \(x<sx\) and \(\mathcal L(x)\cap J\nsubseteq\mathcal L(sx)\).
\end{defi}

\begin{rema}\label{restriction-simple-edge}
Let \(J\subseteq S_n\), let \((W,S)=(W_n,S_n)\) and let
\(\Gamma\) be the regular Kazhdan--Lusztig \(W\)-graph. By the results of
Section~\ref{sec:3} we know that a simple
edge \(\{x,y\}\) of \(\Gamma\) remains a simple edge of
\(\Gamma{\downarrow}_J\) provided that
\(\mathcal{L}(x)\cap J\nsubseteq\mathcal{L}(y)\cap J\) and
\(\mathcal{L}(y)\cap J\nsubseteq\mathcal{L}(x)\cap J\). Recall that
the simple edges of \(\Gamma\) all have the form \(\{x,sx\}\), where
\(s\in S\) and \(x<sx\in W\). Given that \(x<sx\), the condition
\(\mathcal{L}(sx)\cap J\nsubseteq\mathcal{L}(x)\cap J\) holds if and
only if \(s\in J\), and so \(\{x,sx\}\) is a simple edge of\
\(\Gamma{\downarrow}_J\) if and only if \(s\in J\) and
\(\mathcal L(x)\cap J\nsubseteq\mathcal L(sx)\). Thus \(\approx_J\)
is the equivalence relation on \(W\) generated by the requirement
that \(x\approx_J y\) whenever \(\{x,y\}\) is a simple edge
of~\(\Gamma{\downarrow}_J\).
\end{rema}

\begin{defi}\label{approxm-dfn}
Let \(\lambda\in P(n)\) and \(1\leqslant m\leqslant n\). Let \(\approx_m\)
be the equivalence relation on \(\STD(\lambda)\) defined by the
requirement that \(u\approx_m t\) whenever there is a dual Knuth move
of index at most~\(m-1\) from \(u\) to~\(t\) and
\(\D(u)\cap[1,m-1]\nsubseteq\D(t)\). We call such a move
a \textit{\((\leqslant m)\)-dual Knuth move}. The \(\approx_m\) equivalence
classes in \(\STD(\lambda)\) will be called the
\textit{\((\leqslant m)\)-subclasses\/}
of~\(\STD(\lambda)\), and we shall say that \(u,\,t\in\STD(\lambda)\) are
\textit{\((\leqslant m)\)-dual Knuth equivalent\/} whenever \(u\approx_m t\).
\end{defi}

\begin{rema}\label{approxm-rmk}
Assume that \(\lambda\in P(n)\) and \(1\leqslant m\leqslant n\), and
let \(u,\,t\in\STD(\lambda)\). If \(u\to^{*2}t\) and \(\ind(u,t)\leqslant m-1\)
then \(\D(u)\cap[1,m-1]\nsubseteq\D(t)\) if and only if \(\ind(u,t)\in[1,m-2]\).
Clearly this holds if and only if \(\morethan um=\morethan tm\) and
\(\Lessthan um\to^{*2}\Lessthan tm\). If \(u\to^{*1}t\) and \(\ind(u,t)\leqslant m-1\)
then \(\ind(u,t)\in[2,m-1]\), and \(\D(u)\cap[1,m-1]\nsubseteq\D(t)\) is
automatically satisfied. Clearly this holds if and only if
\(\morethan um=\morethan tm\) and \(\Lessthan um\to^{*1}\Lessthan tm\).
It follows that \(u\approx_m t\) if and only if \(\morethan um=\morethan tm\),
since \(\shape(\Lessthan um)=\shape(\Lessthan tm)\) guarantees that
\(\Lessthan um\) and \(\Lessthan tm\) are related by a sequence of dual
Knuth moves. So in fact \(u\approx_m t\) if and only if \(t=wu\)
for some \(w\in W_m\).
\end{rema}
It is a consequence of Definitions~\ref{approxm-in-W} and \ref{approxm-dfn}
that if \(u,\,t\in\STD(\lambda)\) then \(u\approx_m t\) if and only if
\(\word(u)\approx_J \word(t)\), where \(J=S_m\). The set of all
\((\leqslant m)\)-subclasses of \(\STD(\lambda)\) is in bijective correspondence
with the set \(\{\,v\in\STD_m(\lambda/\mu)\mid \mu\in P(m) \text{ and }
[\mu]\subseteq[\lambda]\,\}\), and each \((\leqslant m)\)-subclass of
\(\STD(\lambda)\) is in bijective correspondence with
\(\STD(\mu)\) for some \(\mu\in P(m)\) with \([\mu]\subseteq[\lambda]\).
If \(t\in\STD(\lambda)\) then the \((\leqslant m)\)-subclass that contains
\(t\) is denoted by \(C_m(t)\) and is given by
\(C_m(t) = \{\,u\in\STD(\lambda) \mid\morethan um = \morethan tm\,\}\).

In view of Remark~\ref{altDualKnuth} and Theorem~\ref{knuth}, the following
theorem follows from the results of Kazhdan and
Lusztig~\cite{kazlus:coxhecke}.

\begin{theo}\label{klthr}
If \(t,\,t'\in\STD(n)\) then the \(W_n\)-graphs \(\Gamma(C(t))\) and \(\Gamma(C(t'))\)
are isomorphic if and only if \(\shape(t)=\shape(t')\).
In particular, if \(\lambda\in P(n)\) then
\(\Gamma(C(t))\cong \Gamma(C(\tau_{\lambda}))\) whenever \(t\in\STD(\lambda)\).
\end{theo}

\begin{coro}\label{leftcelllambda}
Let \(\Gamma\) be the \(W_n\)-graph of a Kazhdan--Lusztig left cell of~\(W_n\). Then
\(\Gamma\) is isomorphic to \(\Gamma(C(\tau_\lambda))\) for some
\(\lambda\in P(n)\).
\end{coro}
Clearly for each \(\lambda\in P(n)\) the bijection \(t\mapsto\word(t)\)
from \(\STD(\lambda)\) to \(C(\tau_\lambda)\) can be used to create a
\(W_n\)-graph isomorphic to \(\Gamma(C(\tau_\lambda)\) with \(\STD(\lambda)\)
as the vertex set.

\begin{nota}\label{std-tableau-W-graph}
For each \(\lambda\in P(n)\) we write
\(\Gamma_\lambda=\Gamma(\STD(\lambda),\mu^{(\lambda)},\tau^{(\lambda)})\) for the
\(W_n\)-graph just described.
\end{nota}

\begin{rema}\label{Gamma-lambda-submolecules}
Let \(\lambda\in P(n)\) and let \(J=S_m\subseteq S_n\). It follows from
Remark~\ref{restriction-simple-edge} and Definition~\ref{approxm-dfn}
that the \(J\)-submolecules of \(\Gamma_\lambda\) are spanned by the
\((\leqslant m)\)-subclasses of~\(\STD(\lambda)\).
\end{rema}
Now let \(\lambda\in P(n)\) and \(1\leqslant m\leqslant n\), and put
\(J=S_n\setminus S_m\). The \(J\)-submolecules of \(\Gamma_\lambda\) can
be determined by an analysis similar to that used above. We define
\(\approx^m\) to be the equivalence relation on \(\STD(\lambda)\) generated
by the requirement that \(u\approx^m t\) whenever there is a dual Knuth move
of index at least~\(m\) from \(u\) to~\(t\) and
\(\D(u)\cap[m,n-1]\nsubseteq\D(t)\). The \(\approx^m\) equivalence
classes in \(\STD(\lambda)\) will be called the
\textit{\((\geqslant m)\)-subclasses\/} of~\(\STD(\lambda)\).
If \(u,\,t\in\STD(\lambda)\) then \(u\approx^m t\) if and only if
\(\word(u)\approx_J \word(t)\), with \(J=S_n\setminus S_m\). An
equivalent condition is that \(\lessthan um=\lessthan tm\)
and \(\Morethan um\approx \Morethan tm\). It follows that if
\(t\in\STD(\lambda)\) then the \((\geqslant m)\)-subclass that contains
\(t\) is the set
\(C^m(t) = \{\,u\in\STD(\lambda) \mid\lessthan um = \lessthan tm
\text{ and } \Morethan um\approx \Morethan tm\,\}\).
\begin{rema}\label{Gamma-lambda-submolecules+}
Let \(\lambda\in P(n)\) and \(m\in[1,n]\), and put \(J=S_n\setminus S_m\).
By the discussion above, the \(J\)-submolecules of \(\Gamma_\lambda\) are
spanned by the \((\geqslant m)\)-subclasses of~\(\STD(\lambda)\).
\end{rema}
We shall need to use some properties of the well-known ``jeu-de-taquin''
operation on skew tableaux, which we now describe.

Fix a positive integer \(n\) and a target set \(\mathcal{T}=[m+1,m+n]\). It
is convenient to define a \textit{partial tableau\/} to be a bijection \(t\)
from a subset of \(\{\,(i,j)\mid i,\,j\in\mathbb{Z}^+\,\}\) to~\(\mathcal{T}\).
We shall also assume that the domain of \(t\) is always of the form
\([\kappa/\xi]\setminus\{(i,j)\}\), where \(\kappa/\xi\) is a skew partition
of~\(n+1\) and \((i,j)\in[\kappa/\xi]\). If \((i,j)\) is \(\xi\)-addable
then \(t\) is a \((\kappa/\mu)\)-tableau, with
\([\mu]=[\xi]\cup\{(i,j)\}\), and if \((i,j)\) is \(\kappa\)-removable
then \(t\) is a \((\lambda/\xi)\)-tableau, with
\([\lambda]=[\kappa]\setminus\{(i,j)\}\).

Now suppose that \(\lambda/\mu\) is a skew partition of \(n\) and
\(t\in\STD(\lambda/\mu)\), and suppose also that \(c=(i,j)\)
is a \(\mu\)-removable box.
Note that \(t\) may be regarded as a partial tableau, since
\([\lambda/\mu]=[\kappa/\xi]\setminus\{(i,j)\}\), where
\([\kappa]=[\lambda]\) and \([\xi]=[\mu]\setminus\{(i,j)\}\).
The \textit{jeu de taquin slide on \(t\) into \(c\)} is the
process \(j(c,t)\) given as follows.

Start by defining \(t_0=t\) and \(b_0 = (i,j)\). Proceeding recursively,
suppose that \(k\geqslant 0\) and that \(t_{k}\)
and \(b_{k}\) are defined, with \(t_{k}\) a partial tableau whose
domain is \([\kappa/\xi]\setminus\{b_{k}\}\). If \(b_{k}\) is
\(\lambda\)-removable then the process terminates, we define
\(t'=t_{k}\) and put \(m=k\). If \(b_{k}=(g,h)\) is not
\(\lambda\)-removable we put \(x=\min(t_k(g+1,h),t_k(g,h+1))\), define
\(b_{k+1}=t_{k}^{-1}(x)\), and define \(t_{k+1}\) to be the partial
tableau with domain \([\kappa/\xi]\setminus\{b_{k+1}\}\) given by
\[
t_{k+1}(b)=
\begin{cases}t_{k}(b)&\text{ whenever \(b\) is in the domain of \(t_{k}\)
and \(b\ne b_{k+1}\),}\\
x&\text{ if \(b=b_{k}\).}\\
\end{cases}
\]
(We say that \(x\) slides from \(b_{k+1}\) into \(b_{k}\).)
The tableau \(t'\) obtained by the above process is denoted by \(j^{(c)}(t)\).
The sequence of boxes \(b_0=c,\,b_1,\,\ldots,\,b_m\)
is called the \textit{slide path} of \(j(c,t)\), and the box \(b_m\) is
said to be \textit{vacated\/} by~\(j(c,t)\).

The following observation follows immediately from the definition
of a slide path.

\begin{lemm}\label{slidepath}
Let \(b_0=c,b_1,\ldots,b_m\) be the slide path of a jeu de taquin slide, as
described above. If\/ \(0\leqslant i<j\leqslant m\) then
\(b_i(1)\leqslant b_j(1)\) and \(b_i(2)\leqslant b_j(2)\).
\end{lemm}
We also have the following trivial result.

\begin{lemm}\label{jofminandmaxt}
Let \(\lambda=(\lambda_1^{m_1},\ldots,\lambda_l^{m_l}) \in P(n)\)
and \(\mu=(1)\in P(1)\), and put \(t=(\morethan{\tau_{\lambda}}1)-1\).
Then \((\lambda_1,m_1)\) is vacated by the slide \(j((1,1),t)\). Similarly, if
\(u=(\morethan{\tau^{\lambda}}1)-1\), where
\(\lambda\in P(n)\) and \(\lambda^*=\mu=(\mu_1^{n_1},\ldots,\mu_k^{n_k})\), then
\((n_1,\mu_1)\) is vacated by the slide \(j((1,1),u)\).
\end{lemm}
A sequence of boxes \(\beta=(b_1,\ldots,b_l)\) called a \textit{slide sequence\/}
for a standard skew tableau \(t\) if there exists a sequence of skew tableaux
\(t_0=t,\,t_1,\,\ldots,\,t_l\) such that the jeu de taquin slide \(j(b_i,t_{i-1})\)
is defined for each \(i\in[1,l]\), and \(t_i=j^{(b_i)}(t_{i-1})\). We write
\(t_l=j_\beta(t)\). Clearly the slide sequence \(\beta=(b_1,\ldots,b_l)\) can be
extended to a longer slide sequence \(b_1,\ldots,b_{l+1}\) if the skew tableau
\(t_l\) is not of normal shape. If \(t_l\) is of normal shape then we write
\(t_l=j(t)\). Theorem~\ref{jeudetaquintableau} below says that \(j(t)\) is
independent of the slide sequence and is the insertion tableau of~\(\word(t)\).

\begin{theo}\citex{Theorem 3.7.7}{sag:sym}\label{jeudetaquintableau}
Let \(\lambda/\mu\) be a skew partition of \(n\) and \(t\in \STD(\lambda/\mu)\).
If \(\beta\) is any maximal length slide sequence for~\(t\) then
\(j_\beta(t)=P(\word(t))\).
\end{theo}
Skew tableaux \(u\) and \(t\) are said to be \textit{dual equivalent\/} if
the skew tableaux \(j_\beta(u)\) and \(j_\beta(t)\) are of the same shape
whenever \(\beta\) is a slide sequence for both \(u\)~and~\(t\). Dual
equivalent skew tableaux are necessarily of the same shape, since the
slide sequence \(\beta\) is allowed to have length zero. It is easily shown
that if \(u\) and \(t\) are dual equivalent then every slide sequence
for \(u\) is also a slide sequence for~\(t\), from which it follows that
dual equivalence is indeed an equivalence relation. The following result
says that this equivalence relation coincides with dual Knuth equivalence.

\begin{theo} \citex{Theorem 3.8.8}{sag:sym}\label{dualdualKnuthequivalence}
Let \(\lambda/\mu\) be a skew partition, and let \(u\) and \(t\) be standard
\(\lambda/\mu\)-tableaux. Then \(u\) is dual equivalent to \(t\)
if and only if \(u\approx t\).
\end{theo}
Note that Theorem~\ref{dualdualKnuthequivalence} generalizes the fact that the set
of standard tableaux of a given normal shape form a single dual Knuth equivalence class.

If \(\lambda/\mu\) is a skew partition of \(n\) then the corresponding
\textit{dual equivalence graph\/} has vertex set \(\STD(\lambda/\mu)\)
and edge set \(\{\,\{u,t\}\mid u,t\in\STD(\lambda/\mu)\text{ and
\(u\to^{*1}t\) or \(u\to^{*2}t\)}\,\}\).

If \(k\in[1,n-1]\) then each
\(v\in\STD(\lambda/\mu)\) with \(\D(v)\cap\{k,k+1\}=\{k\}\) is adjacent
in the dual equivalence graph to a unique \(v'\) with
\(\D(v')\cap\{k,k+1\}=\{k+1\}\), and each \(v\) with
\(\D(v)\cap\{k,k+1\}=\{k+1\}\)
is adjacent to a unique \(v'\) with \(\D(v')\cap\{k,k+1\}=\{k\}\).
In fact, \(v'=s_kv\) if
\(\col_v(k)<\col_v(k+2)\leq\col_v(k+1)\) or
\(\col_v(k+1)<\col_v(k+2)\leq\col_v(k)\), and \(v'=s_{k+1}v\) if
\(\col_v(k+1)\leq\col_v(k)<\col_v(k+2)\) or
\(\col_v(k+2)\leq\col_v(k)<\col_v(k+1)\).

\begin{defi}\label{k-neighbourdefinition}
We call the above tableau \(v'\) the \textit{\(k\)-neighbour} of \(v\), and write
\(v'=k\neb(v)\).

\end{defi}
It follows from Remark~\ref{altDualKnuth} that if \(\mu\) is the empty
partition then the dual equivalence graph is isomorphic to
the simple part of each Kazhdan--Lusztig left cell \(\Gamma(C(t))\) for
\(t\in\STD(\lambda)\); in this case we call the dual equivalence graph the
\textit{standard dual equivalence graph\/} corresponding to \(\lambda\in P(n)\).
Extending earlier work of Assaf~\cite{assaf:dualequigraphs}, Chmutov showed
in~\cite{ChuMic:typeAMol} that the simple part of an admissible \(W_n\)-molecule
is isomorphic to a standard dual equivalence graph. The following result is
the main theorem of~\cite{ChuMic:typeAMol}.

\begin{theo}\label{molIsKL}
The simple part of an admissible molecule of type \(A_{n-1}\) is
isomorphic to the simple part of a Kazhdan--Lusztig left cell.
\end{theo}
It is worth noticing that Stembridge has shown that there are
\(A_{15}\)-molecules that cannot occur in Kazhdan--Lusztig left
cells~\cite[Remark 3.8]{stem:morewgraph}.

\begin{rema}\label{vertexsetofWngraph}
It follows that if \(M=(V,\mu,\tau)\) is a molecule then there exists \(\lambda\in P(n)\) and
a bijection \(t\mapsto c_t\) from \(\STD(\lambda)\) to \(V\) such that \(\tau(c_t)=\D(t)\) and
the simple edges of \(M\) are the pairs \(\{c_u,c_t\}\) such that \(u,t\in\STD(\lambda)\) and there is a
dual Knuth move from \(u\) to \(t\) or from \(t\) to \(u\). The molecule~\(M\) is said to be of
type~\(\lambda\).

Let \(M=(V,\mu,\tau)\) be an arbitrary \(S_n\)-coloured molecular graph, and for
each \(\lambda\in P(n)\) let \(m_{\lambda}\)
be the number of molecules of type \(\lambda\) in \(M\). For each \(\lambda\) such
that \(m_{\lambda}\neq 0\) let \(\mathcal{I}_{\lambda}\) be some indexing set of
cardinality \(m_{\lambda}\). Then we can write
\begin{equation}\label{vertexsetexp}
V = \bigsqcup_{\lambda\in\Lambda}\bigsqcup_{\alpha\in\mathcal{I}_{\lambda}}V_{\alpha,\lambda},
\end{equation}
where \(\Lambda\) consists of all \(\lambda\in P(n)\) such that \(m_{\lambda}\neq 0\), each
\(V_{\alpha,\lambda}=\{\,c_{\alpha,t}\mid t\in\STD(\lambda)\,\}\) is the vertex set
of a molecule of type~\(\lambda\), \(\tau(c_{\alpha,t})=\D(t)\), and the simple edges of \(M\) are the pairs
\(\{c_{\alpha,u},c_{\beta,t}\}\) such that \(\alpha=\beta\in\mathcal{I}_{\lambda}\)
for some \(\lambda\in P(n)\) and \(u,t\in\STD(\lambda)\) are related by a dual
Knuth move. We shall call the set \(\Lambda\) the set of \textit{molecule
types\/} for the molecular graph~\(M\).

Note that if \(\Gamma=(V,\mu,\tau)\) is an admissible \(W_n\)-graph then \(\Gamma\)
is an \(S_n\)-coloured molecular graph, by Remark~\ref{WgraphisWmoleculargraph}, and
hence Eq.~\eqref{vertexsetexp} can be used to describe the vertex set of \(\Gamma\).
\end{rema}

\begin{rema}\label{Gamma-lambdaIsAdmissible}
We know from Remark~\ref{KLgraphisadmissibl} and Corollary~\ref{leftcelllambda} that, for each
\(\lambda\in P(n)\), the \(W_n\)-graph \(\Gamma_{\lambda}=(\STD(\lambda),\mu^{(\lambda)},\tau^{(\lambda)})\)
is admissible. Since \(\{u,t\}\) is a simple edge in \(\Gamma_{\lambda}\)
when \(u,\, t\in\STD(\lambda)\) are related by a dual Knuth move, and \(\STD(\lambda)\)
is a single dual Knuth equivalence class, we see that
\(\Gamma_{\lambda}\) consists of a single molecule (of type \(\lambda\)).
\end{rema}

\begin{rema}\label{restrictionrem}
Let \(\Gamma=(V,\mu,\tau)\) be an admissible \(W_n\)-graph, and continue
with the notation and terminology of Remark~\ref{vertexsetofWngraph}
above. Let \(m\in [1,n]\), and let \(K=S_m\) and \(L=S_n\setminus S_m\).

Let \(\lambda\in\Lambda\) and \(\alpha\in\mathcal I_\lambda\), and let \(\Theta\)
be the molecule of \(\Gamma\) whose vertex set is \(V_{\alpha,\lambda}\).
Write \(\Gamma{\downarrow}_K=(V,\mu,\underline\tau)\) (where \(\underline\tau=\tau_K\)
in the notation of Section~\ref{sec:3} above). By
Remark~\ref{vertexsetofWngraph} applied to \(\Theta{\downarrow}_K\),
we may write
\[
V_{\alpha,\lambda}=\bigsqcup_{\kappa\in\Lambda_{K,\alpha,\lambda}}
\bigsqcup_{\,\,\,\,\beta\in\mathcal{I}_{K,\alpha,\lambda,\kappa}\!\!\!\!}\!\!V_{\alpha,\lambda,\beta,\kappa},
\]
where \(\Lambda_{K,\alpha,\lambda}\) is the set of all \(\kappa\in P(m)\) such that
\(\Theta\) contains a \(K\)-submolecule of type~\(\kappa\), and
\(\mathcal I_{K,\alpha,\lambda,\kappa}\) is an indexing set whose size is
the number of such \(K\)-submolecules. Each \(V_{\alpha,\lambda,\beta,\kappa}\)
is the vertex set of a \(K\)-submolecule of \(\Theta\) of type~\(\kappa\). Writing
\(V_{\alpha,\lambda,\beta,\kappa}=\{\,c'_{\beta,u}\mid u\in\STD(\kappa)\,\}\),
we see that each \(c_{\alpha,t}\in V_{\alpha,\lambda}\) coincides with some
\(c'_{\beta,v}\) with \(\beta\in\mathcal I_{K,\alpha,\lambda,\kappa}\) and
\(v\in\STD(\kappa)\).
It follows from Remark~\ref{Gamma-lambda-submolecules} above that the
\(K\)-submolecule of~\(\Theta\) containing a given vertex \(c_{\alpha,t}\) is
spanned by the
\((\leqslant m)\)-subclass~\(C_m(t)=\{\,u\in\STD(\lambda)\mid\morethan um=\morethan tm\,\}\).
Thus when we write \(c_{\alpha,t}=c'_{\beta,v}\) as above, we can identify
\(v\) with~\(\Lessthan tk\).

Similarly, applying Remark~\ref{vertexsetofWngraph} to \(\Theta{\downarrow}_L\),
we may write
\[
V_{\alpha,\lambda}=\bigsqcup_{\theta\in\Lambda_{L,\alpha,\lambda}}
\bigsqcup_{\,\,\,\,\gamma\in\mathcal{I}_{L,\alpha,\lambda,\theta}\!\!\!\!}\!\!V_{\alpha,\lambda,\gamma,\theta},
\]
where \(\Lambda_{L,\alpha,\lambda}\) is the set of all \(\theta\in P(n-m+1)\) such
that \(\Theta\) contains an \(L\)-submolecule of type~\(\theta\), and
\(\mathcal I_{L,\alpha,\lambda,\theta}\) is a set whose size is
the number of such \(L\)-submolecules. Each \(V_{\alpha,\lambda,\gamma,\theta}\)
is the vertex set of an \(L\)-submolecule of \(\Theta\) of type~\(\theta\). Writing
\(V_{\alpha,\lambda,\gamma,\theta}=\{\,c''_{\gamma,v}\mid v\in\STD_{m-1}(\theta)\,\}\)
(where \(\STD_{m-1}(\theta)\) is the set of standard \(\theta\)-tableaux with target
\([m,n]\)), we see that each \(c_{\alpha,t}\in V_{\alpha,\lambda}\) coincides
with some \(c''_{\gamma,v}\) with \(\gamma\in\mathcal I_{L,\alpha,\lambda,\theta}\)
and \(v\in\STD_{m-1}(\theta)\).
By Remark~\ref{Gamma-lambda-submolecules+} above we see that the
\(L\)-submolecule of~\(\Theta\) containing a given vertex \(c_{\alpha,t}\) is
spanned by the \((\geqslant m)\)-subclass
\(C^m(t) = \{\,u\in\STD(\lambda) \mid\lessthan um = \lessthan tm
\text{ and } \Morethan um\approx \Morethan tm\,\}\).
Since the condition \(\Morethan um\approx\Morethan tm\) is satisfied
if and only if \(\word(1-m+(\Morethan um))\approx\word(1-m+(\Morethan tm))\),
and \(j(1-m+(\Morethan tm))=P(\word(1-m+(\Morethan tm)))\) by
Theorem~\ref{jeudetaquintableau}, it follows that when we write
\(c_{\alpha,t}=c''_{\beta,v}\) as above we can identify
\(v\) with~\(j(\Morethan tm)=m-1+j(1-m+(\Morethan tm))\).
\end{rema}

\section{Extended dominance order on
\texorpdfstring{\(\STD(n)\)}{Std(n)} and paired dual Knuth equivalence relation}
\label{sec:8}

Let \(n\geqslant 1\), and let \((W_n,S_n)\) be the Coxeter group of type \(A_{n-1}\) and
\(\mathcal{H}_n\) the corresponding Hecke algebra.
We shall need the following partial order, called the
\textit{extended dominance order}, on~\(\STD(n)\).

\begin{defi}\label{ExtDominance Order Tableaux}
Let \(\lambda,\,\mu\in P(n)\), and let \(u\in\STD(\lambda)\) and \(t\in\STD(\mu)\).
Then \(t\) is said to dominate \(u\) if \(\shape(\Lessthan um) \leqslant
\shape(\Lessthan tm)\) for all \(m\in [1,n]\). When this holds we write
\(u \leqslant t\).
\end{defi}

This is obviously a partial order on
\(\STD(n)=\bigcup_{\lambda\in P(n)}\STD(\lambda)\), and it is also clear
that \(u \leqslant t\) if and only if \(\shape(u) \leqslant \shape(t)\) and
\(\Lessthan u{(n-1)} \leqslant \Lessthan t{(n-1)}\). The terminology and the
\(\leqslant\) notation is justified since it extends the
dominance order on \(\STD(\lambda)\) for each fixed \(\lambda \in P(n)\).
For example, for \(n=2\), we have \(\ttab(1,2)<\ttab(12)\,\), and for \(n=3\), we have
\(\ttab(1,2,3) < \ttab(13,2) < \ttab(12,3) < \ttab(123)\,\).

We remark that in \cite{brini:superalg} this order was used in the context of
the representation theory of symmetric groups, while in~\cite{CasNik:domPerm}
it was used in the context of combinatorics of permutations.

\begin{lemm}\label{nsbeforent}
Let \(\mu,\,\lambda\in P(n)\), let \(u\in\STD(\mu)\) and
\(t\in\STD(\lambda)\), and let \(\eta=\shape(\lessthan un)\) and
\(\theta=\shape(\lessthan tn)\). Suppose that \(\eta\leqslant\theta\)
and \(\col_u(n)\leqslant\col_t(n)\). Then \(\mu\leqslant\lambda\).
\end{lemm}

\begin{proof}
Let \(\col_u(n)=p\) and \(\col_t(n)=q\), and assume that \(p\leqslant q\).
We are given that \(\eta\leqslant\theta\), and so
\(\sum_{m=1}^l \theta_m \leqslant \sum_{m=1}^l \eta_m\) holds for all~\(l\).
Hence for all \(l\in[1,p-1]\) we have
\begin{equation*}
\sum_{m=1}^l \lambda_m =\sum_{m=1}^l \theta_m \leqslant \sum_{m=1}^l \eta_m =\sum_{m=1}^l \mu_m,
\end{equation*}
while for all \(l\in[p,q-1]\) we have
\begin{equation*}
\sum_{m=1}^l \lambda_m =\sum_{m=1}^l \theta_m <
 (\eta_p+1) + \sum_{\substack{m=1\\m\neq p}}^l \eta_m=\sum_{m=1}^l \mu_m,
\end{equation*}
and for all \(l>q\) we have
\begin{equation*}
\sum_{m=1}^l \lambda_m =(\theta_q+1) + \sum_{\substack{m=1\\m\neq q}}^l \theta_m
\leqslant (\eta_p+1) + \sum_{\substack{m=1\\m\neq p}}^l \eta_m =\sum_{m=1}^l \mu_m.
\end{equation*}
Hence \(\mu\leqslant\lambda\).
\end{proof}

\begin{lemm}\label{shapetheta}
Let \(\lambda\in P(n)\) and \(t\in\STD(\lambda)\). Suppose that \(i\in\SD(t)\),
and let \(p=\col_t(i)\) and \(j=\col_t(i+1)\). For all
\(h\in[1,n-1]\) let \(\lambda^{(h)}=\shape(\Lessthan th)\) and \(\theta^{(h)}=\shape(\Lessthan{s_it}h)\).
Then
\begin{equation}\label{sumthetatoh}
\sum_{m=1}^l \theta^{(i)}_m =
\begin{cases}
\sum_{m=1}^l \lambda^{(i)}_m     = \sum_{m=1}^l \lambda^{(i+1)}_m    = \sum_{m=1}^l \lambda^{(i-1)}_m & \text{if \(l<j\)}\\
\sum_{m=1}^l \lambda^{(i)}_m + 1 = \sum_{m=1}^l \lambda^{(i+1)}_m    = \sum_{m=1}^l \lambda^{(i-1)}_m + 1 & \text{if \(j\leqslant l<p\)}\\
\sum_{m=1}^l \lambda^{(i)}_m     = \sum_{m=1}^l \lambda^{(i+1)}_m - 1= \sum_{m=1}^l \lambda^{(i-1)}_m + 1& \text{if \(p<l\)}
\end{cases}
\end{equation}
and
\begin{equation}\label{sumnutoh}
\sum_{m=1}^l \lambda^{(i)}_m =
\begin{cases}
\sum_{m=1}^l \theta^{(i)}_m     = \sum_{m=1}^l \theta^{(i+1)}_m    = \sum_{m=1}^l \theta^{(i-1)}_m   & \text{if \(l<j\)}\\
\sum_{m=1}^l \theta^{(i)}_m - 1 = \sum_{m=1}^l \theta^{(i+1)}_m - 1= \sum_{m=1}^l \theta^{(i-1)}_m   & \text{if \(j\leqslant l<p\)}\\
\sum_{m=1}^l \theta^{(i)}_m     = \sum_{m=1}^l \theta^{(i+1)}_m - 1= \sum_{m=1}^l \theta^{(i-1)}_m + 1& \text{if \(p<l\)}.
\end{cases}
\end{equation}
\end{lemm}
\begin{proof}
The results given by Eq.~(\ref{sumthetatoh}) and Eq.~(\ref{sumnutoh}) are readily obtained from the following formulae
\begin{equation}\label{thetam}
\theta^{(i)}_m =
\begin{cases}
\lambda^{(i)}_m + 1 = \lambda^{(i+1)}_{m} = \lambda^{(i-1)}_m + 1 & \text{if \(m=j\)}\\
\lambda^{(i)}_m - 1 = \lambda^{(i+1)}_{m} - 1 = \lambda^{(i-1)}_m  & \text{if \(m=p\)}\\
\lambda^{(i)}_m = \lambda^{(i+1)}_{m} = \lambda^{(i-1)}_m & \text{if \(m\neq j,p\)}
\end{cases}
\end{equation}
and
\begin{equation}\label{num}
\lambda^{(i)}_m =
\begin{cases}
\theta^{(i)}_m - 1 = \theta^{(i+1)}_{m} - 1 = \theta^{(i-1)}_m     & \text{if \(m=j\)}\\
\theta^{(i)}_m + 1 = \theta^{(i+1)}_{m}     = \theta^{(i-1)}_m + 1 & \text{if \(m=p\)}\\
\theta^{(i)}_m     = \theta^{(i+1)}_{m}     = \theta^{(i-1)}_m     & \text{if \(m\neq j,p\)},
\end{cases}
\end{equation}
respectively.
\end{proof}

\begin{lemm}\label{characterisationofdominanceorderontab}
Let \(\mu,\,\lambda\in P(n)\), let \(u\in\STD(\mu)\) and
\(t\in\STD(\lambda)\). Suppose that \(i\in\SD(u)\cap\SD(t)\). Then
\(u \leqslant t\) if and only if \(s_{i}u \leqslant s_{i}t\).
\end{lemm}

\begin{proof}
Let \(j=\col_t(i+1)\), let \(p=\col_t(i)\), let
\(k=\col_u(i+1)\) and let \(q=\col_u(i)\).
For all \(h\in[1,n]\) let \(\lambda^{(h)} = \shape(\Lessthan th)\), let
\(\theta^{(h)} = \shape(\Lessthan{s_{i}t}h)\),
let \(\mu^{(h)} = \shape(\Lessthan uh)\) and let
 \(\eta^{(h)} = \shape(\Lessthan{s_{i}u}h)\).

Suppose that \(u\leqslant t\). Since \(s_iu\) and \(s_it\) differ from
\(u\) and \(t\) respectively only in the positions of \(i\) and \(i+1\), we have
\(\mu^{(h)}=\eta^{(h)}\) and \(\lambda^{(h)}=\theta^{(h)}\) for all \(h\neq i\).
But since \(\mu^{(h)}\leqslant\lambda^{(h)}\) for all \(h\) by our assumption,
it follows that \(\eta^{(h)}\leqslant\theta^{(h)}\) for all \(h\neq i\).
Hence to show that \(s_iu\leqslant s_it\) it suffices to show that
\(\eta^{(i)}\leqslant\theta^{(i)}\).
Let \(l\in\mathbb{Z}^+\) be arbitrary.

\begin{Case}{1.}
Suppose that \(l\geqslant k\). By Lemma~\ref{shapetheta} applied to \(u\),
we have \(\sum_{m=1}^l \eta^{(i)}_m = \sum_{m=1}^l \mu^{(i-1)}_m + 1\),
by the last two formulae of Eq.(\ref{sumthetatoh}).
Since \(\mu^{(i-1)}\leqslant\lambda^{(i-1)}\) gives
\(\sum_{m=1}^l \mu^{(i-1)}_m \geqslant \sum_{m=1}^l \lambda^{(i-1)}_m\),
it follows that
\(\sum_{m=1}^l \eta^{(i)}_m \geqslant \sum_{m=1}^l \lambda^{(i-1)}_m + 1\).
But by Lemma~\ref{shapetheta} applied to~\(t\), in each case
in Eq.(\ref{sumthetatoh}) we have
\(\sum_{m=1}^l \lambda^{(i-1)}_m + 1 \geqslant \sum_{m=1}^l \theta^{(i)}_m\).
Hence \(\sum_{m=1}^l \eta^{(i)}_m \geqslant \sum_{m=1}^l \theta^{(i)}_m\).
\end{Case}

\begin{Case}{2.}
Suppose that \(l<k\). By Lemma~\ref{shapetheta} applied to \(u\), we have
\(\sum_{m=1}^l \eta^{(i)}_m = \sum_{m=1}^l \mu^{(i)}_m
= \sum_{m=1}^l \mu^{(i+1)}_m\!\), by the first formula of Eq.~(\ref{sumthetatoh}).
Since \(\mu^{(i)}\leqslant\lambda^{(i)}\)
and \(\mu^{(i+1)}\leqslant\lambda^{(i+1)}\), for each \(h\in\{i,i+1\}\)
we obtain
\(\sum_{m=1}^l \mu^{(h)}_m \geqslant \sum_{m=1}^l \lambda^{(h)}_m\),
and hence \(\sum_{m=1}^l \eta^{(i)}_m\geqslant \sum_{m=1}^l \lambda^{(h)}_m\).
By Lemma~\ref{shapetheta} applied to \(t\),
in each case in~Eq.(\ref{sumthetatoh}) there exists \(h\in\{i,i+1\}\)
such that \(\sum_{m=1}^l \lambda^{(h)}_m = \sum_{m=1}^l \theta^{(i)}_m\).
Hence \(\sum_{m=1}^l \eta^{(i)}_m \geqslant \sum_{m=1}^l \theta^{(i)}_m\).
\end{Case}

Conversely, suppose that \(s_iu\leqslant s_it\). As above, it suffices to show that
\(\mu^{(i)}\leqslant\lambda^{(i)}\). Let \(l\in\mathbb{Z}^+\) be arbitrary.

\begin{Case}{1.}
Suppose that \(l\geqslant j\). By Lemma~\ref{shapetheta} applied to \(t\),
we have \(\sum_{m=1}^l \lambda^{(i)}_m = \sum_{m=1}^l \theta^{(i+1)}_m - 1\),
by the last two formulae of Eq.(\ref{sumnutoh}).
Since \(\eta^{(i+1)}\leqslant\theta^{(i+1)}\) gives
\(\sum_{m=1}^l \eta^{(i+1)}_m \geqslant \sum_{m=1}^l \theta^{(i+1)}_m\),
it follows that
\(\sum_{m=1}^l \eta^{(i+1)}_m - 1\geqslant \sum_{m=1}^l \lambda^{(i)}_m\).
But by Lemma~\ref{shapetheta} applied to~\(u\), in each case
in Eq.(\ref{sumnutoh}) we have
\(\sum_{m=1}^l \mu^{(i)}_m\geqslant \sum_{m=1}^l \eta^{(i+1)}-1\).
Hence \(\sum_{m=1}^l \mu^{(i)}_m \geqslant \sum_{m=1}^l \lambda^{(i)}_m\).
\end{Case}

\begin{Case}{2.}
Suppose that \(l<j\). By Lemma~\ref{shapetheta} applied to \(t\), we have
\(\sum_{m=1}^l \lambda^{(i)}_m = \sum_{m=1}^l \theta^{(i-1)}_m=\sum_{m=1}^l \theta^{(i)}_m\),
by the first formula of Eq.~(\ref{sumnutoh}). Since \(\theta^{(i-1)}\geqslant\eta^{(i-1)}\) and
\(\theta^{(i)}\geqslant\eta^{(i)}\), for each \(h\in\{i-1,i\}\) we obtain
\(\sum_{m=1}^l \theta^{(h)}_m\leqslant\sum_{m=1}^l \eta^{(h)}_m\), and hence
\(\sum_{m=1}^l \lambda^{(i)}_m\leqslant\sum_{m=1}^l \eta^{(h)}_m\).
By Lemma~\ref{shapetheta} applied to \(u\),
in each case in~Eq.(\ref{sumnutoh}) there exists \(h\in\{i-1,i\}\)
such that \(\sum_{m=1}^l \eta^{(h)}_m = \sum_{m=1}^l \mu^{(i)}_m\).
Hence \(\sum_{m=1}^l \lambda^{(i)}_m\leqslant\sum_{m=1}^l \mu^{(i)}_m \).\qedhere
\end{Case}
\end{proof}

\begin{defi}\label{paireddualKmove}
Let \(\lambda,\,\mu\in P(n)\) and let \(1\leqslant m\leqslant n\).
Let \(u,\,v \in \STD(\mu)\) and \(t,\,x\in\STD(\lambda)\), and let
\(i\in\{1,2\}\). We say there is a \textit{paired \((\leqslant m)\)-dual
Knuth move of the \(i\)-th kind} from \((u,t)\) to \((v,x)\)
if there exists \(k\leqslant m-1\) such that \(u\to^{*i}v\) and \(t\to^{*i}x\)
are \((\leqslant m)\)-dual Knuth moves of index~\(k\).
When this holds we write \((u,t)\to^{*i}(v,x)\), and call \(k\) the index
of the paired move.
\end{defi}
We have the following equivalence relation on~\(\STD(\mu)\times\STD(\lambda)\).
\begin{defi}
Let \(\lambda,\,\mu\in P(n)\). The \textit{paired \((\leqslant m)\)-dual
Knuth equivalence relation\/} is the equivalence relation \(\approx_m\) on
\(\STD(\mu)\times\STD(\lambda)\) generated by paired \((\leqslant m)\)-dual
Knuth moves. When \(m=n\) we write \(\approx\) for \(\approx_n\), and call it
the paired dual Knuth equivalence relation.
\end{defi}
We denote by \(C_m(u,t)\) the \(\approx_m\) equivalence class that contains
\((u,t)\). By Remark~\ref{approxm-rmk} we see that if \((v,x)\in C_m(u,t)\) then
\((v,x)=(wu,wt)\) for some \(w\in W_m\); furthermore,
\(\morethan vm = \morethan um\) and \(\morethan xm = \morethan tm\).

\begin{rema}\label{mimpliesn}
It is clear that \((u,t)\approx_m (v,x)\) implies \((u,t)\approx_{m'}(v,x)\) whenever \(m\leqslant m'\).
In particular, \((u,t)\approx_m (v,x)\) implies \((u,t)\approx (v,x)\). For this
reason, \(C_m(u,t)\) will be called the \textit{\((\leqslant m)\)-subclass\/} of
\(C(u,t)=C_n(u,t)\).
\end{rema}

\begin{rema}
Let \(\lambda\in P(n)\). Since \(\STD(\lambda)\) is a single dual Knuth
equivalence class, it follows that \(C(u,u)=\{(t,t)\mid t\in\STD(\lambda)\}\)
holds for all \(u\in\STD(\lambda)\).
\end{rema}
For example, consider \(\mu=(3,1)\) and \(\lambda=(2,1,1)\). Then set \(\STD(\mu)\times\STD(\lambda)\) has
\(9\) elements.
It is easily shown that there are seven paired dual Knuth equivalence classes, of
which two classes have \(2\) elements and five classes have \(1\) element only.
The two non-trivial classes are
\(\{(\ttab(14,2,3),\ttab(124,3)), (\ttab(13,2,4),\ttab(123,4))\}\) and \(\{(\ttab(13,2,4),\ttab(134,2)),
(\ttab(12,3,4),\ttab(124,3))\}\).

Let \(\mu,\,\lambda\in P(n)\) and \((u,t),\,(v,x)\in \STD(\mu)\times\STD(\lambda)\).
and suppose that \((v,x) = (s_iu,s_it)\) for some \(i\in [1,n-1]\). If
\(i\in \SD(u)\cap SD(t)\) then \(u\leqslant t\) if and only if \(v\leqslant x\),
by Lemma~\ref{characterisationofdominanceorderontab}, and it follows by
interchanging the roles of \((u,t)\) and \((v,x)\) that the same is true
if \(i\in \SA(u)\cap SA(t)\). In particular, if there is a paired
dual Knuth move from \((u,t)\) to \((v,x)\) or from \((v,x)\) to \((u,t)\) then
\(u\leqslant t\) if and only if \(v\leqslant x\). An obvious induction
now yields the following result.

\begin{prop}\label{bruhatorderpreserved}
Let \(\mu,\,\lambda\in P(n)\). Let \((u,t),\,(v,x)\in \STD(\mu)\times\STD(\lambda)\) and
suppose that \((u,t)\approx (v,x)\). Then  \(u\leqslant t\) if and only if
\(v\leqslant x\).
\end{prop}
Let \(\mu,\,\lambda\in P(n)\) and \((u,t),\,(v,x)\in \STD(\mu)\times\STD(\lambda)\).
and suppose that \((v,x) = (s_iu,s_it)\) for some \(i\in [1,n-1]\). If
\(i\in \SD(u)\cap SD(t)\) then \(l(v)-l(x)=(l(u)-1)-(l(t)-1) = l(u)-l(t)\), and
if \(i\in \SA(u)\cap\SA(t)\) then \(l(v)-l(x)=(l(u)+1)-(l(t)+1) = l(u)-l(t)\).
In particular, \(l(v)-l(x)=l(u)-l(t)\) if there is a paired dual Knuth move
from \((u,t)\) to \((v,x)\) or from \((v,x)\) to \((u,t)\). It clearly follows
that \(l(x)-l(v)\) is constant for all \((v,x)\in C(u,t)\). Hence we
obtain the following result.

\begin{prop}\label{weakorderpreserved}
Let \(\mu,\,\lambda\in P(n)\). Let \((u,t),\,(v,x)\in \STD(\mu)\times\STD(\lambda)\) and
suppose that \((u,t)\approx (v,x)\). Then  \(u\leqslant\lside v\) if and only if
\(t\leqslant\lside x\).
\end{prop}
\begin{proof}
Since \((u,t)\approx (v,x)\) there exists \(w\in W_m\) such that \(v=wu\)
and \(x=wt\). By the definition of the left weak order it follows that
\(u\leqslant\lside v\) if and only if \(l(v)-l(u)=l(w)\), and
\(t\leqslant\lside x\) if and only if \(l(x)-l(t)=l(w)\). Since
\((u,t)\approx (v,x)\) implies that \(l(v)-l(u)=l(x)-l(t)\), the result
follows.
\end{proof}

\begin{defi}\label{k-restrictable-ed}
Let \(\mu,\,\lambda\in P(n)\) and \((u,t)\in\STD(\mu)\times\STD(\lambda)\).
If \(j\in [1,n]\) and \(\Lessthan uj=\Lessthan tj\) then we say that the
pair \((u,t)\) is \(j\)-\textit{restrictable}.
\end{defi}

\begin{rema}\label{k-restrictable-rem}
It is clear that the set
\(R(u,t)=\{\,j\in[1,n]\mid (u,t)\text{ is }j\text{-restrictable}\,\}\)
is always nonempty, since \(1\in R(u,t)\). Moreover, \(R(u,t)=[1,k]\) for
some \(k\in [1,n]\).
\end{rema}

\begin{defi}\label{restrictedNumber}
Let \(\mu,\,\lambda\in P(n)\) and \((u,t)\in\STD(\mu)\times\STD(\lambda)\).
We shall call the number \(k\) satisfying \(R(u,t)=[1,k]\)
the \textit{restriction number\/} of the pair \((u,t)\). If \(k\) is the
restriction number of \((u,t)\) then we say that \((u,t)\) is
\(k\)-\textit{restricted}.
\end{defi}

\begin{rema}\label{k-restrricted-rem}
With \((u,t)\) as above, the restriction number of \((u,t)\) is at least \(1\) and at
most~\(n\). If \(k\in[1,n]\) then \((u,t)\) is \(k\)-restricted if and only
if it is \(k\)-restrictable and not \((k+1)\)-restrictable. If \((u,t)\) is
\(k\)-restricted then \(k=n\) if and only if \(u=t\), and if \(k<n\)
then \(\col_u(k+1)\ne\col_t(k+1)\) and \(\row_u(k+1)\ne\row_t(k+1)\).
\end{rema}

\begin{lemm}\label{knuthweakordercase23}
Let \(\mu,\,\lambda\in P(n)\), and let \(u\in\STD(\mu)\) and
\(t\in\STD(\lambda)\). If \(n<4\) then \(\D(u)=\D(t)\) implies \(u=t\).
\end{lemm}
\begin{proof}
This is trivially proved by listing all the standard tableaux.
\end{proof}

\begin{defi}\label{k-reduced}
Let \(\mu,\,\lambda\in P(n)\) and \((u,t)\in\STD(\mu)\times\STD(\lambda)\).
We say that the pair \((u,t)\) is \textit{favourable\/} if the restriction number
of~\((u,t)\) lies in \(\D(u)\oplus\D(t)\), the symmetric difference of the descent
sets of \(u\)~and~\(t\).
\end{defi}

\begin{rema}\label{k-reducedrem}
Let \(\mu,\,\lambda\in P(n)\), and suppose that \((u,t)\in\STD(\mu)\times\STD(\lambda)\)
is \(k\)-restricted. Then no element of \([1,k-1]\) can belong to \(\D(u)\oplus\D(t)\),
since \(\Lessthan uk=\Lessthan tk\) implies that \(\D(u)\cap[1,k-1]=\D(t)\cap[1,k-1]\).
So if \((u,t)\) is favourable then \(k =\min(\D(u)\oplus\D(t))\), and if
\((u,t)\) is not favourable then \(k<\min(\D(u)\oplus\D(t))\).
\end{rema}

Let \(\mu,\,\lambda\in P(n)\), and let \((u,t)\in\STD(\mu)\times\STD(\lambda)\).
Let \(i\) be the restriction number of~\((u,t)\), and suppose that \(i\ne n\).
Let \(w=\Lessthan ui=\Lessthan ti\in\STD(\xi)\), where
\(\xi=\shape(w)\), and let also \((g,p)=u^{-1}(i+1)\) and \((h,q)=t^{-1}(i+1)\),
the boxes of \(u\) and \(t\) that contain~\(i+1\). Thus \((g,p)\) and \((h,q)\)
are \(\xi\)-addable, and \((g,p)\ne (h,q)\) since \((u,t)\) is not
\((i+1)\)-restrictable. Clearly there is at least one \(\xi\)-removable
box \((d,m)\) that lies between \((g,p)\) and \((h,q)\) (in the sense that
either \(g>d\geqslant h\) and \(p\leqslant m<q\), or \(h>d\geqslant g\) and
\(q\leqslant m<p\)), and note that \(i\in \D(u)\oplus\D(t)\) if and only if
the \(\xi\)-removable box \(w^{-1}(i)\) is such a box.

With \((d,m)\) as above, suppose that \(w'\in\STD(\xi)\) satisfies \(w'(d,m)=i\).
Since \(\STD(\xi)\) is a single dual Knuth equivalence class there must be a
sequence of dual Knuth moves of index at most~\(i-1\) taking \(w\) to~\(w'\).
This same sequence of dual Knuth moves takes \((u,t)\) to \((v,x)\), where \(v\)
satisfies \(\Lessthan vi = w'\) and \(\morethan vi= \morethan ui\), and \(x\)
satisfies \(\Lessthan xi = w'\) and \(\morethan xi= \morethan ti\). Thus \((v,x)\)
is \(i\)-restricted and favourable, and \((v,x)\approx_i (u,t)\).

We denote by \(F(u,t)\) the set of all \((v,x)\) obtained by the above
construction, as \((d,m)\) and \(w'\) vary.
Clearly every \((v,x)\in F(u,t)\) is \(k\)-restricted and favourable, and
satisfies\((v,x)\approx_{i}(u,t)\). Note also
that \((u,t)\in F(u,t)\) if and only if \((u,t)\) is favourable.

Since \(\col_v(i+1)=\col_u(i+1)\) and \(\col_x(i+1)=\col_t(i+1)\),
we can now deduce the following result.

\begin{lemm}\label{knuthmoveweakorderrelationLem}
Let \(\mu,\,\lambda\in P(n)\) and let \((u,t)\in\STD(\mu)\times\STD(\lambda)\)
with \(u\ne t\). Let \(i\) be the restriction
number of\/~\((u,t)\), and assume that \(i\notin \D(u)\oplus\D(t)\).
Let \((v,x)\in F(u,t)\).
Then either
\(\D(x)\setminus\D(v)=\D(t)\setminus\D(u)\)
and
\(\D(v)\setminus\D(x)=\{i\}\cup(\D(u)\setminus\D(t))\),
this alternative occurring in the case that
\(\col_u(i+1)<\col_t(i+1)\), or else
\(\D(x)\setminus\D(v)=\{i\}\cup(\D(t)\setminus\D(u))\)
and
\(\D(v)\setminus\D(x)=\D(u)\setminus\D(t)\)
(in the case that
\(\col_t(i+1)<\col_u(i+1)\)).
\end{lemm}

\begin{proof}
The construction of \((v,x)\) is given in the preamble above. Since
\((v,x)\) and \((u,t)\) are both \(i\)-restricted,
\(\D(v)\cap[1,i-1]=\D(x)\cap[1,i-1]\) and \(\D(u)\cap[1,i-1]=\D(t)\cap[1,i-1]\).
That is,
\((\D(v)\oplus \D(x))\cap[1,i-1] = (\D(u)\oplus \D(t))\cap[1,i-1] =\emptyset\).
Furthermore, since \(\morethan vi= \morethan ui\) and
\(\morethan xi= \morethan ti\) it follows that
\((\D(v)\setminus\D(x))\cap[i+1,n-1]=(\D(u)\setminus\D(t))\cap[i+1,n-1]\) and
\((\D(x)\setminus\D(v))\cap[i+1,n-1]=(\D(t)\setminus\D(u))\cap[i+1,n-1]\).
It remains to observe that if
\(p=\col_v(i+1)\leqslant m=\col_v(i)=\col_x(i)<q=\col_x(i+1)\) then
\(i\in\D(v)\setminus\D(x)\), while if
\(q\leqslant m<p\) then \(i\in\D(x)\setminus\D(v)\).
\end{proof}

\begin{lemm}\label{knuthmoveweakorderrelationLemb}
Let \(\mu,\,\lambda\in P(n)\) and let \((u,t)\in\STD(\mu)\times\STD(\lambda)\).
Assume that the restriction number of\/~\((u,t)\) lies in
\(\D(u)\oplus\D(t)\), and let \((v,x)\in F(u,t)\).
Then \(\D(v)\setminus\D(x)=\D(u)\setminus\D(t)\) and
\(\D(x)\setminus\D(v)=\D(t)\setminus\D(u)\).
\end{lemm}

\begin{proof}
The proof is the same as the proof of Lemma~\ref{knuthmoveweakorderrelationLem}, except
that it can be seen now that
\(i\in\D(v)\setminus\D(x)\) if \(i\in\D(u)\setminus\D(t)\) and
\(i\in\D(x)\setminus\D(v)\) if \(i\in\D(t)\setminus\D(u)\).
\end{proof}

\begin{lemm}\label{noname2}
Let \(\mu,\,\lambda\in P(n)\) and \((u,t)\in\STD(\mu)\times\STD(\lambda)\),
and \(i\) the restriction number of~\((u,t)\).
Suppose that \(\D(t)\subsetneqq\D(u)\) and \(i<j\), where
\(j=\min (\D(u)\setminus\D(t))\). Let \((v,x)\in F(u,t)\).
If \(\col_u(i+1)<\col_t(i+1)\) then
\(\D(v)\setminus\D(x)=\{i\}\cup(\D(u)\setminus\D(t))\) and
\(\D(x)\setminus\D(v)=\emptyset\), while if \(\col_t(i+1)<\col_u(i+1)\) then
\(\D(v)\setminus\D(x)=\D(u)\setminus\D(t)\) and
\(\D(x)\setminus\D(v)=\{i\}\). In the former case \(\D(v)\cap\{i,j\}=\{i,j\}\) and
\(\D(x)\cap\{i,j\}=\emptyset\), while in the latter case \(\D(v)\cap\{i,j\}=\{j\}\) and
\(\D(x)\cap\{i,j\}=\{i\}\).
\end{lemm}

\begin{proof}
Since \(\D(t)\subsetneqq\D(u)\) we have \(\D(u)\oplus\D(t)=\D(u)\setminus\D(t)\neq\emptyset\);
so \(j=\min(\D(u)\oplus\D(t))\), and since \(j>i\), we have \(i\notin\D(u)\oplus\D(t)\), hence
\((v,x)\in F(u,t)\) satisfies
the further properties specified in Lemma~\ref{knuthmoveweakorderrelationLem}.

If \(\col_u(i+1)<\col_t(i+1)\) then Lemma~\ref{knuthmoveweakorderrelationLem} gives
\(\D(v)\setminus\D(x)=\{i\}\cup(\D(u)\setminus\D(t))\) and \(\D(x)\setminus\D(v)=\emptyset\),
since \(\D(t)\setminus\D(u)=\emptyset\) by hypothesis. In particular, since \(j\in\D(u)\setminus\D(t)\),
we see that \(\D(v)\setminus\D(x)\) contains both \(i\)~and~\(j\).

If \(\col_t(i+1)<\col_u(i+1)\) then Lemma~\ref{knuthmoveweakorderrelationLem} combined together
with \(\D(t)\setminus\D(u)=\emptyset\) gives
\(\D(v)\setminus\D(x)=\D(u)\setminus\D(t)\) and \(\D(x)\setminus\D(v)=\{i\}\).
In particular it follows that \(j\in\D(v)\setminus\D(x)\) and \(i\in\D(x)\setminus\D(v)\).
\end{proof}

Let \(\Gamma=\Gamma(C,\mu,\tau)\) be a \(W_n\)-molecular graph, and let \(\Lambda\)
be the set of molecule types for~\(\Gamma\). For each \(\lambda\in\Lambda\) let
\(m_{\lambda}\) be the number of molecules of type \(\lambda\) in \(\Gamma\), and
\(\mathcal{I}_{\lambda}\) some indexing set of cardinality \(m_{\lambda}\).
As in Remark~\ref{vertexsetofWngraph}, the
vertex set of \(\Gamma\) can be expressed in the form
\begin{equation*}
C=\bigsqcup_{\lambda\in\Lambda}\bigsqcup_{\alpha\in\mathcal{I}_{\lambda}}C_{\alpha,\lambda},
\end{equation*}
where \(C_{\alpha,\lambda}=\{c_{\alpha,t}\mid t\in\STD(\lambda)\}\) for each
\(\alpha\in\mathcal{I}_{\lambda}\),
and the simple edges of \(\Gamma\) are the pairs
\(\{c_{\beta,u},c_{\alpha,t}\}\) such that \(\alpha=\beta\in\mathcal{I}_{\lambda}\) for some \(\lambda\in\Lambda\)
and \(u,\,t\in\STD(\lambda)\) are related by a dual Knuth move.

Now let \(\lambda,\,\mu\in\Lambda\), and let \((\alpha,t)\in\mathcal{I}_\lambda\times\STD(\lambda)\) and
\((\beta,u)\in\mathcal{I}_\mu\times\STD(\mu)\), so that \(c_{\alpha,t}\) and
\(c_{\beta,u}\) are vertices of~\(\Gamma\). Suppose that
\(\D(u)\setminus\D(t)\neq\emptyset\), and let \(j\in\D(u)\setminus\D(t)\).

Suppose that there exist \(i<j\) and
\((v,x)\in\STD(\mu)\times\STD(\lambda)\) such that \((u,t)\) and \((v,x)\)
are related by a paired \((\leqslant i)\)-dual Knuth move.
Then \(j\in\D(\morethan ui)\setminus\D(\morethan ti)\), since \(j\in\D(u)\setminus\D(t)\)
and \(j>i\). Thus \(j\in\D(\morethan vi)\setminus\D(\morethan xi)\), since
\((v,x)\approx_i(u,t)\) gives \(\morethan vi=\morethan ui\) and \(\morethan xi=\morethan ti\).
Hence \(j\in\D(v)\setminus\D(x)\). Moreover, since \((u,t)\) and \((v,x)\) are related by a
paired \((\leqslant i)\)-dual Knuth move, there are \(k,l\leqslant i-1\) with \(|k-l|=1\)
such that
\begin{align*}
\qquad&&\D(x)\cap \{k,l,j\} &= \{k\},&\D(v)\cap \{k,l,j\} &= \{k,j\},&&\qquad\\
\qquad&&\D(t)\cap \{k,l,j\} &= \{l\},&\D(u)\cap \{k,l,j\} &= \{l,j\},&&\qquad
\end{align*}
and it follows from Proposition~\ref{arctransport}
that \(\mu(c_{\beta,v},c_{\alpha,x})=\mu(c_{\beta,u},c_{\alpha,t})\).

More generally, suppose that \(i<j\) and \((v,x)\in\STD(\mu)\times\STD(\lambda)\) satisfy
\((v,x)\approx_i (u,t)\), so that for some \(m\in\mathbb{N}\) there exist
\((u_0,t_0)\), \((u_1,t_1)\), \dots, \((u_m,t_m)\) in \(\STD(\mu)\times\STD(\lambda)\),
with \((u_{h-1},t_{h-1})\) and \((u_h,t_h)\) related by a
paired \((\leqslant i)\)-dual Knuth move for each \(h\in[1,m]\), and \((u_0,t_0)=(u,t)\)
and \((u_m,t_m)=(v,x)\). Applying the argument in the preceding paragraph and a trivial induction,
we deduce that \(j\in\D(u_h)\setminus\D(t_h)\) and
\(\mu(c_{\beta,u_h},c_{\alpha,t_h})=\mu(c_{\beta,u},c_{\alpha,t})\) for all \(h\in[0,m]\).
Thus we obtain the following result.

\begin{lemm}\label{noname0}
Let \(\Gamma\) be a \(W_n\)-molecular graph. Using the notation as above, let \(\lambda,\,\mu\in\Lambda\),
and let \((\alpha,t)\in\mathcal{I}_\lambda\times\STD(\lambda)\) and
\((\beta,u)\in\mathcal{I}_\mu\times\STD(\mu)\). Suppose that \(\D(u)\setminus\D(t)\neq\emptyset\), and
let \(j\in\D(u)\setminus\D(t)\). Then for all \(i<j\) and all \((v,x)\in C_i(u,t)\) we have
\(j\in\D(v)\setminus\D(x)\) and \(\mu(c_{\beta,v},c_{\alpha,x})=\mu(c_{\beta,u},c_{\alpha,t})\).
\end{lemm}

\begin{coro}\label{arweightone}
Let \(\Gamma\) be a \(W_n\)-molecular graph as above.
Let \(\lambda\in\Lambda\), and \(u,\,t\in\STD(\lambda)\), and suppose that
\(u=s_jt>t\) for some \(j\in[1,n-1]\). Then \(\mu(c_{\alpha,u},c_{\alpha,t}) = 1\),
for all \(\alpha\in\mathcal{I}_{\lambda}\).
\end{coro}

\begin{proof}
Since \(t<s_jt=u\), it follows from Remark~\ref{altDualKnuth} that if \(\D(t)\nsubseteq\D(u)\)
then there is a dual Knuth move from \(t\) to \(u\), and \(\{c_{\alpha,u},c_{\alpha,t}\}\)
is a simple edge. Thus \(\mu(c_{\alpha,u},c_{\alpha,t}) = 1\) in this case, and so we may
assume that \(\D(t)\subsetneqq\D(u)\).

Since \(u=s_jt\) it is clear that \(\lessthan tj=\lessthan uj\), and hence
\(\D(t)\cap[1,j-2]=\D(u)\cap[1,j-2]\). If \(j-1\in\D(u)\) then \(j-1\in\D(t)\), as
\(\col_t(j-1)=\col_u(j-1)\geqslant\col_u(j)\geqslant\col_u(j+1)=\col_t(j)\). Moreover, since
\(t<s_jt\) gives \(j\in\D(u)\setminus\D(t)\), it follows that \(j=\min(\D(u)\setminus\D(t))\).
Note also that \(j-1\) is the restriction number of \((u,t)\).

Writing \(i\) for \(j-1\), we see that \(u\) and \(t\) satisfy the hypotheses of Lemma~\ref{noname2},
since \(i<j=\min(\D(u)\setminus\D(t))\). Since \(\col_t(i+1)<\col_u(i+1)\), it
follows that \((u,t)\approx_i(v,x)\), where \((v,x)\in F(u,t)\) satisfies
\(\D(v)\cap\{i,j\}=\{j\}\) and \(\D(x)\cap\{i,j\}=\{i\}\). Since
\((u,t)\approx_i(v,x)\) there exists \(w\in W_i\) with \(v=wu\) and \(x=wt\), and since
\(j>i\) it follows that \(s_jw=ws_j\). Thus \(s_jx=s_jwt=ws_jt=wu=v\). Furthermore \(s_jx>x\), since
\(j\notin\D(x)\), and \(\D(x)\nsubseteq\D(v)\) since \(i\in\D(x)\setminus\D(v)\).
So there is a dual Knuth move indexed by \(j\) from \(x\) to \(v\), and so
\(\{c_{\alpha,v},c_{\alpha,x}\}\) is a simple edge. Thus
\(\mu(c_{\alpha,v},c_{\alpha,x})=1\), and so \(\mu(c_{\alpha,u},c_{\alpha,t})=1\) by
Lemma~\ref{noname0}.
\end{proof}

\begin{lemm}\label{noname1muequallambda}
Let \(\Gamma\) be a \(W_n\)-molecular graph as above. Let
\(\mu,\,\lambda\in\Lambda\), let \(u\in\STD(\mu)\) and let \(t\in\STD(\lambda)\).
Suppose that \(\D(u) = \{n-1\} \cup \D(t)\) and \(\mu(c_{\beta,u},c_{\alpha,t}) \neq 0\)
for some \(\beta\in\mathcal{I}_{\mu}\) and \(\alpha\in\mathcal{I}_{\lambda}\).
Suppose further that the restriction number of~\((u,t)\)
is \(i< n-2\). Then \(\col_u(i+1)<\col_t(i+1)\), and \((u,t)\approx_i (v,x)\) for some
\((v,x)\in\STD(\mu)\times\STD(\lambda)\) such that
\(\D(v)=\D(x)\cup\{i,n-1\}\) and \(\mu(c_{\beta,v},c_{\alpha,x})=\mu(c_{\beta,u},c_{\alpha,t}) \neq 0\).
\end{lemm}

\begin{proof}
Since clearly \(u\ne t\), the set \(F(u,t)\) is defined and nonempty.
Let \((v,x)\in F(u,t)\). Then it follows by Lemmas~\ref{noname2} and
\ref{noname0} that \((u,t)\approx_i (v,x)\)
and \(\mu(c_{\beta,v},c_{\alpha,x})=\mu(c_{\beta,u},c_{\alpha,t}) \neq 0\).
Moreover, if \(\col_u(i+1)<\col_t(i+1)\) then
Lemma~\ref{noname2} gives \(\D(v)=\D(x)\cup\{i,n-1\}\). Thus it
remains to show that \(\col_u(i+1)<\col_t(i+1)\). Suppose otherwise.
Then Lemma~\ref{noname2} shows that \(n-1\in\D(v)\setminus\D(x)\)
and \(i\in\D(x)\setminus\D(v)\), and now the \(W\)-Compatibility Rule says that
\(i\) and \(n-1\) must be joined by a bond in the Coxeter diagram of \(W_n\).
This contradicts the assumption that \(i<n-2\).
\end{proof}

\begin{lemm}\label{nminus1order0}
Suppose that \(u,\,t\in\STD(n)\) are such that the restriction number of~\((u,t)\)
is \(n-1\) and \(\D(u) = \{n-1\} \cup \D(t)\). Then
\(\col_u(n)<\col_t(n)\), \(\shape(u)<\shape(t)\), and \(u<t\).
\end{lemm}

\begin{proof}
Clearly \(n\geqslant 2\). Since \(\Lessthan u{(n-1)}=\Lessthan t{(n-1)}\)
we have \(\shape(\lessthan un)=\shape(\lessthan tn)\), and
since \(n-1\in\D(u)\setminus\D(t)\) we have
\(\col_u(n)\leqslant\col_u(n-1)=\col_t(n-1)<\col_t(n)\).
Hence \(\shape(u)<\shape(t)\) by Lemma~\ref{nsbeforent}, and \(u<t\)
by Definition~\ref{ExtDominance Order Tableaux}.
\end{proof}

\begin{lemm}\label{nminus1order}
Let \(\Gamma\) be a \(W_n\)-molecular graph as above. Let
\(\mu,\,\lambda\in\Lambda\), let \(u\in\STD(\mu)\), and let \(t\in\STD(\lambda)\).
Suppose that \(\D(u) = \{n-1\} \cup \D(t)\) and that \(\mu(c_{\beta,u},c_{\alpha,t}) \neq 0\)
for some \(\beta\in\mathcal{I}_{\mu}\) and \(\alpha\in\mathcal{I}_{\lambda}\),
and suppose that the restriction number of \((u,t)\) is \(n-2\).
Then \((u,t)\approx_{n-2} (v,x)\) for some
\((v,x)\in\STD(\mu)\times\STD(\lambda)\) with
\(\mu(c_{\beta,v},c_{\alpha,x})=\mu(c_{\beta,u},c_{\alpha,t})\), and
either \(u<t\) and \(\D(v) = \{n-2,n-1\} \cup \D(x)\)
(in the case \(\col_u(n-1)<\col_t(n-1)\)), or else \((\lambda,\alpha)=(\mu,\beta)\)
and \(u=s_{n-1}t>t\), and \(\mu(c_{\beta,u},c_{\alpha,t}) = 1\) (in the
case \(\col_t(n-1)<\col_u(n-1)\)).
\end{lemm}

\begin{proof}
Clearly \(n\geqslant 3\). Since \((u,t)\) is \((n-2)\)-restricted,
we have \(\Lessthan u{(n-2)} = \Lessthan t{(n-2)}\)
and \(\col_u(n-1) \neq \col_t(n-1)\). Observe that \((u,t)\) satisfies the
hypotheses of Lemma~\ref{noname2} with \(i=n-2\) and \(j=n-1\). Thus letting
\((v,x)\in F(u,t)\), it follows that
\((u,t)\approx_{n-2}(v,x)\),
and furthermore, \(\mu(c_{\beta,v},c_{\alpha,x})=\mu(c_{\beta,u},c_{\alpha,t}) \neq 0\)
by Lemma~\ref{noname0}.

\begin{Case}{1.}Suppose that \(\col_u(n-1)<\col_t(n-1)\). Then since
\(\shape(\Lessthan u{(n-2)})=\shape(\Lessthan t{(n-2)})\)
it follows from Lemma~\ref{nsbeforent} that
\(\shape(\Lessthan u{(n-1)})<\shape(\Lessthan t{(n-1)})\). Furthermore, since
\(n-1\in\D(u)\setminus\D(t)\), it follows that
\(\col_u(n)\leqslant\col_u(n-1)<\col_t(n-1)<\col_t(n)\). Hence \(\mu<\lambda\)
by Lemma~\ref{nsbeforent}, and \(u<t\) by Definition~\ref{ExtDominance Order Tableaux}.
Moreover, since \(\col_u(n-1)<\col_t(n-1)\) and \(\D(u)\setminus\D(t)=\{n-1\}\), it
follows from Lemma~\ref{knuthmoveweakorderrelationLem} that
\(\D(v)=\D(x)\cup\{n-2,n-1\}\).
\end{Case}
\begin{Case}{2.}Suppose that \(\col_t(n-1)<\col_u(n-1)\). Lemma~\ref{knuthmoveweakorderrelationLem}
gives \(\D(x)\cap\{n-2,n-1\}=\{n-2\}\) and \(\D(v)\cap\{n-2,n-1\}=\{n-1\}\), and since
\(\mu(c_{\beta,v},c_{\alpha,x})\ne0\) it follows from
\(W_n\)-Simplicity Rule that \(\mu(c_{\beta,v},c_{\alpha,x})=1\).
Hence \(\mu(c_{\beta,u},c_{\alpha,t})=1\). Moreover, since \(\{c_{\beta,v},c_{\alpha,x}\}\) is a
simple edge, it follows that from Theorem~\ref{molIsKL} and Remark~\ref{vertexsetofWngraph}
that \(\lambda=\mu\) and \(\alpha=\beta\). Hence \(u=s_{n-1}t\), since
\(\Lessthan u{(n-2)} = \Lessthan t{(n-2)}\), and \(u>t\) since \(\col_t(n-1)<\col_u(n-1)\).
\qedhere
\end{Case}
\end{proof}

\begin{rema}\label{nminus1iplusonest}
Let \(\Gamma\) be a \(W_n\)-molecular graph as above. Let
\(\mu,\,\lambda\in\Lambda\), let \(u\in\STD(\mu)\), and let \(t\in\STD(\lambda)\).
Suppose that \(\D(u) = \{n-1\} \cup \D(t)\) and that \(\mu(c_{\beta,u},c_{\alpha,t}) \neq 0\)
for some \(\beta\in\mathcal{I}_{\mu}\) and \(\alpha\in\mathcal{I}_{\lambda}\). Let
\(i\) be the restriction number of \((u,t)\), and note that \(i\leqslant n-1\). If \(i<n-2\) then
\(\col_u(i+1)<\col_t(i+1)\) by Lemma~\ref{nminus1order0}, and if \(i=n-1\) then
\(\col_u(n)<\col_t(n)\) by  Lemma~\ref{noname1muequallambda}. In the remaining case
\(i=n-2\), if \(\col_u(n-1)>\col_t(n-1)\) then \(u=s_{n-1}t>t\) by
Lemma~\ref{nminus1order}. Thus it can be deduced that if
\(u\neq s_{n-1}t>t\) then \(\col_u(i+1)<\col_t(i+1)\).
\end{rema}

\begin{rema}\label{iplusonest}
Let \(\Gamma\) be a \(W_n\)-molecular graph as above.
Let \(\mu,\,\lambda\in\Lambda\), and
let \((\beta,u)\in\mathcal{I}_{\mu}\times\STD(\mu)\) and
\((\alpha,t)\in\mathcal{I}_{\lambda}\times\STD(\lambda)\) satisfy the condition
\(\mu(c_{\beta,u},c_{\alpha,t}) \neq 0\) and \(\D(t) \subsetneqq \D(u)\).
Let \(j=\min(\D(u)\setminus\D(t))\), and~\(i\) the restriction number of~\((u,t)\), and
note that \(i\leqslant j\).

Let \(K=S\setminus\{s_{j+1},\ldots,s_{n-1}\}\), and let
\(\Gamma_K = \Gamma{\downarrow_K}\), the \(W_K\)-graph obtained by restricting
\(\Gamma\) to \(W_K\). As in Remark~\ref{restrictionrem}, for each \(\lambda\in\Lambda\)
and \(\alpha\in\mathcal{I}_\lambda\) we define \(\Lambda_{K,\alpha,\lambda}\) to be the
set of all \(\kappa\in P(j+1)\) such that the molecule of \(\Gamma\) with the vertex set
\(C_{\alpha,\lambda}\) contains a \(K\)-submolecule of type~\(\kappa\),
and let \(\mathcal{I}_{K,\alpha,\lambda,\kappa}\) index these
submolecules. Let \(\Lambda_K=\bigcup_{\alpha,\lambda}\Lambda_{K,\alpha,\lambda}\), the
set of molecule types for \(\Gamma_K\), and for each \(\kappa\in\Lambda_K\) let
\(\mathcal{I}_{K,\kappa}=\bigsqcup_{\{(\alpha,\lambda)|\kappa\in\Lambda_{K,\alpha,\lambda}\}}
\mathcal{I}_{K,\alpha,\lambda,\kappa}\). For each \(\beta\in\mathcal{I}_{K,\kappa}\) we
write \(\{c'_{\beta,u}\mid u\in\STD(\kappa)\}\) for the vertex set of the corresponding
\(K\)-submolecule of~\(\Gamma\).

Let \(v=\Lessthan u{(j+1)}\) and \(x=\Lessthan t{(j+1)}\), and write \(\eta=\shape(v)\)
and \(\theta=\shape(x)\). By Remark~\ref{restrictionrem}, we can identify the vertex
\(c_{\beta,u}\) of \(\Gamma_K\) with \(c'_{\delta,v}\) for some
\(\delta\in\mathcal{I}_{K,\beta,u,\eta}\), and
the vertex \(c_{\alpha,t}\) of \(\Gamma_K\) with \(c'_{\gamma,x}\) for some
\(\gamma\in\mathcal{I}_{K,\alpha,\lambda,\theta}\). It is clear that \(\D(v)=\D(x)\cup\{j\}\),
so it follows that \(\mu(c'_{\delta,v},c'_{\gamma,x})=\underline{\mu}(c_{\beta,u},c_{\alpha,t})\neq 0\).
Moreover, since \(i\leqslant j\), the restriction number of \((v,x)\) is also \(i\). Thus
Lemma~\ref{nminus1order0}, Lemma~\ref{noname1muequallambda} and Lemma~\ref{nminus1order} are
applicable to \(\Gamma_K\) and \((v,x)\) subject to hypotheses \(i=j\), \(i<j-1\) and
\(i=j-1\), respectively. In particular, Remark~\ref{nminus1iplusonest} says that
if \(u\neq s_{i+1}t>t\) then \(\col_u(i+1)<\col_t(i+1)\).
\end{rema}
We end this section with two technical lemmas that will be used throughout the rest of the
paper. They are concerned with descent sets and the lexicographic order on standard tableaux.
The first of these lemmas is needed for future applications of the polygon rule. Recall that
if \(t\in\STD(n)\) and \(i\in[1,n-1]\) then \(s_it\in\STD(n)\) if and only if either
\(i\in\SA(t)\) or \(i\in\SD(t)\).

\begin{lemm}\label{variousaltpaths}
Let \(t\in\STD(n)\) and let \(i\in\A(t)\) and \(j\in\SD(t)\). Put \(v=s_jt\).
\begin{enumerate}[label=(\roman*),topsep=1 pt]
\item Suppose that \(i<j-1\). Then \(i\notin\D(v)\) and \(j\notin\D(v)\).

\noindent
Additionally, if \(i\in\SA(v)\) then \(i\in\D(s_{i}v)\) and \(j\notin\D(s_{i}v)\).
\item Suppose that \(i=j-1\) and \(\col_t(j+1)>\col_t(j-1)\).
Then \(j-1\notin\D(v)\) and \(j\notin\D(v)\).

\noindent
Additionally, if \(j-1\in\SA(v)\) then
\(j-1\in\D(s_{j-1}v)\) and \(j\notin\D(s_{j-1}v)\).
\item Suppose that \(i=j-1\) and \(\col_t(j+1)<\col_t(j-1)\). Then
\(j-1\in\SD(v)\). Writing \(w=s_{j-1}v\), we have
\(j-1\in\D(v)\) and \(j\notin\D(v)\), and \(j-1\notin\D(w)\) and
\(j\notin\D(w)\).

\noindent
Additionally, if \(j\in\SA(w)\), then \(j-1\in\SA(s_jw)\),
and we have  \(j\in\D(s_jw)\) and \(j-1\notin\D(s_{j}w)\), and
\(j-1\in\D(s_{j-1}s_jw)\) and \(j\notin\D(s_{j-1}s_jw)\).
\end{enumerate}
\end{lemm}

\begin{proof}
(i)\quad Since \(v=s_jt\) and \(j\in\SD(t)\), it follows that \(j\in\SA(v)\),
whence \(j\notin\D(v)\).
Since \(v\) is obtained from \(t\) by switching the positions of~\(j\) and \(j+1\),
and since \(i+1<j\), it follows that \(i\) and \(i+1\) have the same row and column
index in \(v\) as they have in \(t\). Since \(i\notin\D(t)\), this shows that
\(i\notin\D(v)\).

If \(i\in\SA(v)\) then \(s_iv\) is standard and \(i\in\D(s_iv)\).
Since \(s_iv\) is obtained from \(v\) by switching \(i\)~and~\(i+1\),
and since \(j>i+1\), it follows that \(j\) and \(j+1\) have the same row and column
index in \(s_iv\) as in \(v\). Since \(j\notin\D(v)\) it follows
that \(j\notin\D(s_iv)\).

(ii)\quad Since \(v=s_jt\) and \(j\in\SD(t)\), it follows that \(j\in\SA(v)\),
whence \(j\notin\D(v)\). Now since \(\col_v(j-1)=\col_t(j-1)\) and \(\col_v(j)=\col_t(j+1)\),
and \(\col_t(j-1)<\col_t(j+1)\) by assumption, it follows that \(\col_v(j-1)<\col_v(j)\).
That is, \(j-1\notin\D(v)\).

If \(j-1\in\SA(v)\) then \(s_{j-1}v\) is standard and \(j-1\in\D(s_{j-1}v)\).
Since \(j-1\) and \(j\) are both ascents of \(v\), we have \(\col_v(j-1)<\col_v(j)<\col_v(j+1)\),
and since \(s_{j-1}v\) is obtained from \(v\) by switching \(j-1\) and \(j\), we have
\(\col_{s_{j-1}v}(j)=\col_v(j-1)\) and \(\col_{s_{j-1}v}(j+1)=\col_v(j+1)\), and
it follows that \(\col_{s_{j-1}v}(j)<\col_{s_{j-1}v}(j+1)\). Thus \(j\notin\D(s_{j-1}v)\).

(iii)\quad As in (i) and (ii) we have \(j\notin\D(v)\). The assumption \(\col_t(j+1)<\col_t(j-1)\)
gives \(\col_v(j)<\col_v(j-1)\), and so \(j-1\in\SD(v)\). Hence \(w=s_{j-1}v\) is standard,
and \(j-1\in\SA(w)\). Since \(\col_w(j+1)=\col_v(j+1)=\col_t(j)\) and
\(\col_w(j)=\col_v(j-1)=\col_t(j-1)\), and since \(j-1\in\A(t)\) by assumption, it follows
that \(j\in \A(w)\). Thus \(j-1\in\D(v)\) and \(j\notin\D(v)\), and \(j-1\notin\D(w)\) and
\(j\notin\D(w)\), as required.

If \(j\in\SA(w)\) then \(s_jw\in\STD(\lambda)\). Since \(j-1\) and \(j\) are both strong ascents of
\(w\), we have \(\row_w(j-1)>\row_w(j)>\row_w(j+1)\), and since \(s_jw\) is obtained from \(w\)
by switching \(j\) and \(j+1\), we have \(\row_{s_jw}(j-1)=\row_w(j-1)\) and
\(\row_{s_jw}(j)=\row_w(j+1)\), and it follows that \(\row_{s_jw}(j-1)>\row_{s_jw}(j)\).
Thus \(j-1\in\SA(s_jw)\).

Now \(j-1\in\SA(s_jw)\) gives \(j-1\notin\D(s_jw)\), and gives
\(j-1\in\D(s_{j-1}s_jw)\). Similarly, \(j\in\SA(w)\) gives
\(j\in\D(s_jw)\). Finally, the assumption \(\col_t(j+1)<\col_t(j-1)\) gives
\(\col_{s_{j-1}s_jw}(j)=\col_{s_jw}(j-1)=\col_t(j+1)<\col_t(j-1)=\col_{s_jw}(j+1)
=\col_{s_{j-1}s_jw}(j+1)\), and \(j\notin\D(s_{j-1}s_jw)\).
\end{proof}

Recall from Remark~\ref{LexTableauEquiv} that if \(\lambda\in P(n)\) and
\(u,\,t\in\STD(\lambda)\) then \(t>_{\lex}u\) if and only if there exists \(l\in[1,n]\)
such that \(\col_t(l)<\col_u(l)\) and \(\morethan tl=\morethan ul\).
\begin{lemm}\label{xlexlessthant}
Let \(\lambda\in P(n)\) and \(0\leqslant i\leqslant n-1\). Let \(t,t'\in\STD(\lambda)\) satisfy
\(\morethan ti=\morethan {t'}i\). Let \(j\in \SD(t)\) and put \(v=s_jt\), and suppose that
\(i\in\A(t)\) and \(i<j\). Then \(v<_{\lex}t'\), and the following all hold.
\begin{enumerate}[label=(\roman*),topsep=1 pt]
\item If \(i\in\SA(v)\) then  \(s_iv\in\STD(\lambda)\) and \(s_iv<_{\lex}t'\).
\item If \(y\in\STD(\lambda)\) and \(y<v\) then \(y<_{\lex}t'\).
\item Suppose that \(i=j-1\) and that \(\col_t(j+1)<\col_t(j-1)\), and
let \(w=s_{j-1}v\). Then \(w\in\STD(\lambda)\) and \(w<_{\lex}t'\).
If \(j\in\SA(w)\) then \(s_{j-1}s_jw\in\STD(\lambda)\) and \(s_{j-1}s_jw<_{\lex}t'\).
\item Suppose that \(i=j-1\) and that \(\col_t(j+1)<\col_t(j-1)\), and
let \(w=s_{j-1}v\). Let \(x\in\STD(\lambda)\) be such that \(x<w\) and \(\D(x)\)
contains exactly one of \(j-1\) or \(j\), and let \(y\) be the \((j-1)\)-neighbour
of \(x\) (see Definition~\ref{k-neighbourdefinition}). Then \(y<_{\lex}t'\).
\end{enumerate}
\end{lemm}

\begin{proof}
Since \(j\in\SD(t)\) we have \(t>s_jt=v\), and hence
\(t>_{\lex}v\) by Corollary~\ref{domimplieslextableau}. Indeed,
\(\col_t(j+1)<\col_t(j)=\col_v(j+1)\) and \(\morethan t{(j+1)}=\morethan v{(j+1)}\).
Since \(\morethan ti=\morethan {t'}i\) and \(j+1>i\) it follows that
\(\col_{t'}(j+1)<\col_v(j+1)\) and
\(\morethan {t'}{(j+1)}=\morethan v{(j+1)}\),
giving \(t'>_{\lex}v\).

(i)\quad  The assumption \(i\in\SA(v)\) gives \(s_iv\in\STD(\lambda)\),
and since \(j+1>i+1\) it follows that \(\col_{t'}(j+1)<\col_v(j+1)=\col_{s_iv}(j+1)\)
and \(\morethan{t'}{(j+1)}=\morethan{s_iv}{(j+1)}\). So \(t'>_{\lex}s_iv\).

(ii)\quad If \(y<v\) then \(y<_\lex v\), by  Corollary~\ref{domimplieslextableau}, and
since \(v<_\lex t'\) this gives \(y<_{\lex}t'\).

(iii)\quad Since \(\col_v(j)=\col_t(j+1)<\col_t(j-1)=\col_v(j-1)\), we have
\(j-1\in\SD(v)\), and since this gives \(s_{j-1}v\in\STD(\lambda)\), an argument similar
to that for (i) yields \(w<_{\lex}t'\).

If \(j\in\SA(w)\) then \(s_jw\in\STD(\lambda)\). Since \(j-1\in\SA(s_jw)\)
by Lemma~\ref{variousaltpaths} (iii), we have \(s_{j-1}s_jw\in\STD(\lambda)\).
Since \(\col_t(j+1)<\col_t(j-1)=\col_{s_{j-1}s_jw}(j+1)\), and
since \(j+1>i+1\), it follows that \(\col_{t'}(j+1)<\col_{s_{j-1}s_jw}(j+1)\) and
\(\morethan{t'}{(j+1)}=\morethan{s_{j-1}s_jw}{(j+1)}\). This gives \(t'>_{\lex}s_{j-1}s_jw\).

(iv)\quad There are two cases to consider.

\begin{Case}{1.} Suppose that \(\D(x)\cap\{j-1,j\}=\{j-1\}\) and \(\D(y)\cap\{j-1,j\}=\{j\}\).
Then either \(y=s_jx>x\) or \(y=s_{j-1}x<x\).

Suppose first that \(y=s_{j}x>x\). Since \(x<w\) and \(w=s_{j-1}v<v\) by the proof of~(iii), it
follows that \(x<v\). Since \(v<s_jv=t\) and \(x<s_jx=y\), it follows by
 Lemma~\ref{characterisationofdominanceorderontab} that \(y<t\). Hence
\(y<_{\lex}t\) by Corollary~\ref{domimplieslextableau}. That is, there exists
\(l\in[1,n]\) such that \(\col_{t}(l)<\col_y(l)\) and
\(\morethan tl=\morethan yl\). Suppose, for a contradiction, that \(l\leqslant j-1\).
Then \(\morethan t{(j-1)}=\morethan y{(j-1)}\), giving
\(\morethan {(s_jt)}{(j-1)}=\morethan{(s_jy)}{(j-1)}\), that is,
\(\morethan v{(j-1)}=\morethan x{(j-1)}\). Thus \(\morethan vj=\morethan xj\), and
hence \(\morethan wj=\morethan{s_{j-1}v)}j=\morethan vj=\morethan xj\).
Now since \(\col_x(j)=\col_v(j)<\col_v(j-1)=\col_w(j)\) (using again
\(s_{j-1}v<v\)), it follows that \(w<_{\lex}x\),
contradicting the assumption that \(x<w\). Thus \(l\geqslant j\). Since
\(\morethan t{(j-1)}=\morethan {t'}{(j-1)}\), it follows that \(\col_{t'}{l}=\col_t(l)<\col_y(l)\)
and \(\morethan {t'}l=\morethan yl\). Hence \(y<_{\lex}t'\) as required.

Suppose now that \(y=s_{j-1}x<x\). Since
\(x<w\), we have \(y<w\), whence \(y<_{\lex}w\) by Corollary~\ref{domimplieslextableau}.
But \(w<_{\lex}t'\) by (iii), this yields \(y<_{\lex}t'\).
\end{Case}

\begin{Case}{2.} Suppose that \(\D(x)\cap\{j-1,j\}=\{j\}\) and \(\D(y)\cap\{j-1,j\}=\{j-1\}\).
Then either \(y=s_jx<x\) or \(y=s_{j-1}x>x\).

Suppose first that \(y=s_{j-1}x>x\). Since \(x<s_{j-1}x=y\) and \(w<s_{j-1}w=v\), the assumption
\(x<w\) gives \(y<v\) by Lemma~\ref{characterisationofdominanceorderontab}. Thus
\(y<_{\lex}t'\) by (ii).

Suppose now that \(y=s_jx<x\). Since \(x<w\), we have \(y<w\), whence
\(y<_{\lex}w\) by Corollary~\ref{domimplieslextableau}. But
\(w<_{\lex}t'\) by (iii), this yields \(y<_{\lex}t'\).\qedhere
\end{Case}
\end{proof}

\section{Ordered admissible \textit{W-}graphs in type \textit{A}}
\label{sec:8ex}

Let \(\Gamma=\Gamma(C,\mu,\tau)\) be an admissible \(W_n\)-graph, and let
\(\Lambda\subseteq P(n)\) be the set of molecule types for \(\Gamma\). As in
Remark~\ref{vertexsetofWngraph} we write
\begin{equation*}
C=\bigsqcup_{\lambda\in\Lambda}\bigsqcup_{\alpha\in\mathcal{I}_{\lambda}}C_{\alpha,\lambda},
\end{equation*}
where for each \(\lambda\in\Lambda\) the set \(\mathcal{I}_{\lambda}\) indexes
the molecules of \(\Gamma\) of type~\(\lambda\), and for each \(\lambda\in\Lambda\)
and \(\alpha\in\mathcal{I}_{\lambda}\) the set
\(C_{\alpha,\lambda}=\{c_{\alpha,t}\mid t\in\STD(\lambda)\}\) is the vertex
set of a molecule of type~\(\lambda\).
Fix \(\lambda\in\Lambda\) and let \(C'_\lambda=
C\setminus\bigl(\bigsqcup_{\alpha\in\mathcal{I}_{\lambda}}C_{\alpha,\lambda}\bigr)\),
the set of vertices of \(\Gamma\) belonging to molecules of type different
from~\(\lambda\). We define \(\Ini_{\lambda}(\Gamma)\) to be the set of
\((\alpha,t)\in\mathcal{I}_{\lambda}\times\STD(\lambda)\) such that there exists
an arc from \(c_{\alpha,t}\) to some vertex in \(C'_\lambda\). That is,
\[
\Ini_{\lambda}(\Gamma)=
\bigl\{\,(\alpha,t)\in\mathcal{I}_{\lambda}\times\STD(\lambda)\bigm|
\mu(c_{\beta,u},c_{\alpha,t})\neq 0 \text{ for some } (\beta,u)\in\!\!\!\!
\bigsqcup_{\mu\in\Lambda\setminus\{\lambda\}}\!\!\!
(\mathcal{I}_{\mu}\times\STD(\mu))\,\bigr\}.
\]
For each \(\alpha\in\mathcal{I}_\lambda\) we also define
\(\Ini_{(\alpha,\lambda)}(\Gamma)=\{\,t\in\STD(\lambda)\mid(\alpha,t)\in\Ini_{\lambda}(\Gamma)\,\}\).

Note that, by Theorem~\ref{combinatorialCharacterisation}, \(\Gamma\)
satisfies the \(W_n\)-Compatibility Rule, the \(W_n\)-Simplicity Rule,
the \(W_n\)-Bonding Rule and the \(W_n\)-Polygon Rule.

Now since \(\Gamma\) satisfies the \(W_n\)-Simplicity Rule, it follows by
Definition~\ref{simplicity} that whenever vertices \(c_{\alpha,t}\)
and \(c_{\beta,u}\) belong to different molecules and
\(\mu(c_{\beta,u},c_{\alpha,t})\neq 0\), we must have \(\D(t)\subsetneqq\D(u)\) and
\(\mu(c_{\alpha,t},c_{\beta,u})=0\).

Suppose that \(\Ini_{\lambda}(\Gamma)\neq\emptyset\). We define
\(t_{\Gamma,\lambda}\) to be the element of
\(\bigcup_{\alpha\in\mathcal{I}_\lambda}\Ini_{(\alpha,\lambda)}(\Gamma)\)
that is minimal in the  lexicographic order on \(\STD(\lambda)\). If
\(\Gamma\) is clear from the context then we will simply
write~\(t_\lambda\) for~\(t_{\Gamma,\lambda}\).

We make the following definition.
\begin{defi}\label{orderedWgraphdef}
Let \(\Gamma=\Gamma(C,\mu,\tau)\) be an admissible \(W_n\)-graph, and let
\begin{equation*}
C=\bigsqcup_{\lambda\in\Lambda}\bigsqcup_{\alpha\in\mathcal{I}_{\lambda}}C_{\alpha,\lambda},
\end{equation*}
as above. Then \(\Gamma\) is said to be \textit{ordered}
if for all vertices \(c_{\alpha,t}\) and \(c_{\beta,u}\) with
\(\mu(c_{\beta,u},c_{\alpha,t})\neq 0\), either \(u<t\) (in the extended
dominance order) or else \(\alpha=\beta\) and \(u=st>t\) for some \(s\in S_n\).
\end{defi}
Note that \(\mu(c_{\beta,u},c_{\alpha,t})\neq 0\)
implies that \(\D(u)\nsubseteq\D(t)\). In particular, since \(S_1=\emptyset\),
the condition \(\mu(c_{\beta,u},c_{\alpha,t})\neq 0\) can never be satisfied
in the case \(n=1\). Thus it is vacuously true that any \(W_1\)-graph is ordered.

Our objective in this section is to prove Theorem~\ref{inductivestep}, which states
that all admissible \(W_n\)-graphs are ordered. The proof will proceed by
induction on~\(n\).

\begin{rema}\label{orderedregular}
In particular, it will follow from Theorem~\ref{inductivestep} that the
Kazhdan--Lusztig \(W_n\)-graph corresponding to the regular representation
of~\(\mathcal{H}(W_n)\) is ordered in the sense of Definition~\ref{orderedWgraphdef}.
In this case the vertex set of \(\Gamma=(C,\mu,\tau)\) is \(C=W_n\), the set
of molecule types is \(\Lambda=P(n)\), for each \(\lambda\in P(n)\) the set of
molecules of type \(\lambda\) is indexed by~\(\mathcal{I}_\lambda=\STD(\lambda)\),
and for each \(\lambda\in\Lambda\) and \(x\in\mathcal{I}_\lambda\) the
set \(C_{x,\lambda}\) consists of those \(w\in W_n\) such that
\(Q(w)=x\), where \(Q(w)\) is the recording tableau in the Robinson--Schensted
process.

Now let \(v,\,w\in W_n\) and put \(\RS(w)=(t,x)\in\STD(\lambda)^2\) and
\(\RS(v)=(u,y)\in\STD(\nu)^2\), where \(\lambda,\,\mu\in P(n)\).
The conclusion of Theorem~\ref{inductivestep}, applied in this case, is that
if \(\mu(v,w)\ne 0\) and \(\tau(v)\nsubseteq\tau(w)\) then either \(u<t\) or
else \(\mu=\lambda\) and \((u,y)=(st,x)\) for some \(s\in S_n\).

If \(\Gamma\) is replaced by \(\Gamma\opp=(C,\mu,\tau\opp)\), then since
\(\RS(w^{-1})=(x,t)\) and \(\RS(v^{-1})=(y,u)\) by Theorem~\ref{rswinv},
the conclusion of Theorem~\ref{inductivestep} is that if \(\mu(v,w)\ne 0\) and
\(\tau\opp(v)\nsubseteq\tau\opp(w)\) then either \(y<x\) or
else \(\mu=\lambda\) and \((u,y)=(t,sx)\) for some \(s\in S_n\).

Thus, in particular, if \(\mu(v,w)\neq 0\) and \(\tau(v)\nsubseteq\tau(w)\) or
\(\tau\opp(v)\nsubseteq\tau\opp(w)\) then \(\mu\leq\lambda\).

It follows from the definition of the preorder~\(\preceq\leftright\) (in
Section~\ref{sec:3} above) that if \(v,\,w\in W_n\) and \(v\preceq\leftright w\)
then there is a sequence of elements
 \(z_0=v,\, z_1,\, \ldots,\, z_{m-1},\, z_m=w\) such that \(\mu(z_{i-1},z_i)\ne 0\)
and \(\bar\tau(z_{i-1})\nsubseteq\bar\tau(z_i)\) for
each \(i\in[1,m]\). Since \(\bar\tau(z_{i-1})\nsubseteq\bar \tau(z_i)\) is equivalent
to \(\tau(z_{i-1})\nsubseteq\tau(z_i)\) or \(\tau\opp(z_{i-1})\nsubseteq\tau\opp(z_i)\),
it follows that \(\mu\leqslant\lambda\).
\end{rema}

We now commence the proof of Theorem~\ref{inductivestep}. We assume that \(n\) is a
positive integer and that all admissible \(W_m\)-graphs are ordered for \(1\leqslant m<n\).
We let \(\Gamma=\Gamma(C,\mu,\tau)\) be an admissible \(W_n\)-graph, and use the
notation introduced in the preamble to this section: \(\Lambda\) is the set of molecule
types of \(\Gamma\), and for each \(\lambda\in\Lambda\) the set \(\mathcal{I}_\lambda\)
indexes the molecules of type~\(\lambda\). We fix \(K=S_n\setminus\{s_{n-1}\}\) and
\(L=S_n\setminus\{s_1\}\), and we let \(\Gamma_K = \Gamma{\downarrow_K}\)
and \(\Gamma_L = \Gamma{\downarrow_L}\), the \(W_K\)-graph and \(W_L\)-graph obtained by restricting
\(\Gamma\) to \(W_K\) and \(W_L\). Since \(|K|=|L|=n-1\), the
inductive hypothesis tells us that \(\Gamma_K\) and \(\Gamma_L\) are ordered.

By Remark~\ref{restrictionrem}, the set of molecule types for
\(\Gamma_K\) is
\(\Lambda_K=\bigcup_{\alpha,\lambda}\Lambda_{K,\alpha,\lambda}\),
where \(\Lambda_{K,\alpha,\lambda}\) is the set of all \(\kappa\in P(n-1)\)
such that the molecule with the vertex set \(C_{\alpha,\lambda}\) contains a
\(K\)-submolecule of type \(\kappa\), and for each \(\kappa\in\Lambda_K\),
the indexing set for those molecules of type \(\kappa\) is
\(\mathcal{I}_{K,\kappa}=\bigsqcup_{\{\alpha,\lambda\mid \kappa\in\Lambda_{K,\alpha,\lambda}\}}
\mathcal{I}_{K,\alpha,\lambda,\kappa}\), where \(\mathcal{I}_{K,\alpha,\lambda,\kappa}\)
indexes the \(K\)-submolecules of type \(\kappa\) in the molecule with
the vertex set \(C_{\alpha,\lambda}\). The vertex set of \(\Gamma_K\) is
\begin{equation*}
C = \bigsqcup_{\kappa\in\Lambda_K}\{c'_{\gamma,x}\mid (\gamma,x)\in\mathcal{I}_{K,\kappa}\times\STD(\kappa)\}.
\end{equation*}

By Remark~\ref{restrictionrem}, the set of molecule types for
\(\Gamma_L\) is
\(\Lambda_L=\bigcup_{\alpha,\lambda}\Lambda_{L,\alpha,\lambda}\),
where \(\Lambda_{L,\alpha,\lambda}\) is the set of all \(\theta\in P(n-1)\)
such that the molecule with the vertex set \(C_{\alpha,\lambda}\) contains an
\(L\)-submolecule of type \(\theta\), and for each \(\theta\in\Lambda_L\),
the indexing set for those molecules of type \(\theta\) is
\(\mathcal{I}_{L,\theta}=\bigsqcup_{\{\alpha,\lambda\mid \theta\in\Lambda_{L,\alpha,\lambda}\}}
\mathcal{I}_{L,\alpha,\lambda,\theta}\), where \(\mathcal{I}_{L,\alpha,\lambda,\theta}\)
indexes the \(L\)-submolecules of type \(\theta\) in the molecule with
the vertex set \(C_{\alpha,\lambda}\). The vertex set of \(\Gamma_L\) is
\begin{equation*}
C = \bigsqcup_{\theta\in\Lambda_L}\{c''_{\epsilon,y}\mid (\epsilon,y)\in\mathcal{I}_{L,\theta}\times\STD(\kappa)\}.
\end{equation*}

\begin{lemm}\label{kldominancesamenolecule0}
Let \(\mu,\,\lambda\in\Lambda\) with \(\mu\leqslant\lambda\), and
let \((\beta,u)\in\mathcal{I}_{\mu}\times\STD(\mu)\) and
\((\alpha,t)\in\mathcal{I}_{\lambda}\times\STD(\lambda)\) satisfy the condition
\(\mu(c_{\beta,u},c_{\alpha,t}) \neq 0\) and \(\D(t) \subsetneqq \D(u)\).
Let \(j=\min(\D(u)\setminus\D(t))\) and assume that \(j<n-1\). Then \(u < t\) unless
\(\alpha = \beta\) and \(u=s_j t>t\).
\end{lemm}

\begin{proof}
Since \(j\) is at least \(1\), the requirement that \(n-1>j\) implies that \(n\geqslant 3\).
Let \(v=\Lessthan u{(n-1)}\) and \(x=\Lessthan t{(n-1)}\), and write
\(\eta=\shape(v)\) and \(\theta=\shape(x)\). We shall need the restriction of \(\Gamma\)
to \(W_K\) constructed earlier.

By Remark~\ref{restrictionrem}, we can identify the vertex \(c_{\beta,u}\) of
\(\Gamma_K\) with \(c'_{\delta,v}\) for some \(\delta\in\mathcal{I}_{K,\beta,u,\eta}\), and
the vertex \(c_{\alpha,t}\) of \(\Gamma_K\) with \(c'_{\gamma,x}\) for some
\(\gamma\in\mathcal{I}_{K,\alpha,\lambda,\theta}\). Now since \(j\in\D(u)\setminus\D(t)\) and
\(j<n-1\), we have \(j\in(\D(u)\cap[1,n-2])\setminus(\D(t)\cap[1,n-2])=\D(v)\setminus\D(x)\),
and it follows that
\(\mu(c'_{\delta,v},c'_{\gamma,x})=\underline{\mu}(c_{\beta,u},c_{\alpha,t})\neq 0\).
Since \(\Gamma_K\) is ordered, we have either \(v<x\) or \(\gamma=\delta\) and
\(v=s_ix>x\) for some \(i\in[1,n-2]\). In the former case, since
\(\shape(u)=\mu\leqslant\lambda=\shape(t)\) by hypothesis and
since \(\Lessthan u{(n-1)}=v<x=\Lessthan t{(n-1)}\),
we have \(u<t\) by the remark following Definition~\ref{ExtDominance Order Tableaux}
In the latter case, we have \(\alpha=\beta\), and since it is clear that
\(i=j\), it follows that \(u=s_jt>t\).
\end{proof}

\begin{prop}\label{kldominancesamenolecule}
Let \(\mu,\,\lambda\in\Lambda\) with \(\mu\leqslant\lambda\), and
suppose that \((\beta,u)\in\mathcal{I}_{\mu}\times\STD(\mu)\) and
\((\alpha,t)\in\mathcal{I}_{\lambda}\times\STD(\lambda)\) satisfy
\(\mu(c_{\beta,u},c_{\alpha,t}) \neq 0\). Then \(u < t\) unless
\(\alpha = \beta\) and \(u=s_i t>t\) for some \(i\in[1,n-1]\).
\end{prop}

\begin{proof}
Since \(\mu(c_{\beta,u},c_{\alpha,t}) \neq 0\), it follows that
\(\D(u)\nsubseteq\D(t)\).
If \(\D(t)\nsubseteq\D(u)\) then the \(W_n\)-Simplicity Rule shows that
\(\{c_{\beta,s},c_{\alpha,t}\}\) is a simple edge, thus \(\alpha=\beta\) and
\(u=s_i t\) for some \(i\in[1,n-1]\). Thus we may assume that \(\D(t)\subsetneqq\D(u)\).
If \(\min(\D(s)\setminus\D(t))<n-1\) then
the result is given by Lemma~\ref{kldominancesamenolecule0}.

It remains to consider the case \(\D(u)=\D(t)\cup\{n-1\}\). Let \(i\) be the restriction
number of the pair \((u,t)\) and note that \(i<n\) by Remark~\ref{k-restrricted-rem}.
If \(i=n-1\) or \(i=n-2\) then the results are given
by Lemma~\ref{nminus1order0} and Lemma~\ref{nminus1order}, respectively. We may assume that
\(i<n-2\). It follows by Lemma~\ref{noname1muequallambda} that
\((u,t)\approx_i (v,x)\) for some \((v,x)\in\STD(\mu)\times\STD(\lambda)\) satisfying
the condition \(\mu(c_{\beta,u},c_{\alpha,t}) = \mu(c_{\beta,v},c_{\alpha,x}) \neq 0\)
and \(\D(x)\subsetneqq\D(v)=\D(x)\cup\{i,n-1\}\).
Since it is clear that \(v\neq s_k x > x\) for all \(k\in[1,n-1]\),
Lemma~\ref{kldominancesamenolecule0}
shows that \(v<x\), equivalently, \(u<t\) by Proposition~\ref{bruhatorderpreserved}.
\end{proof}

The following definitions are motivated by the structure of \(t_{\Gamma,\lambda}\).

\begin{defi}\label{k-minimalappr}
Let \(\mu,\lambda\in P(n)\). Let \((u,t)\in\STD(\mu)\times\STD(\lambda)\), and
let \(k\) be the restriction number of \((u,t)\).
The pair \((u,t)\) is said to be \(k\)-\textit{minimal}, and
\(t\) is said to be \(k\)-\textit{minimal}
with respect to \(u\), if \(\D(t)\subsetneqq\D(u)\) and \(\Morethan tk\)
is \(k\)-critical, and \(\lessthan tk\) is the minimal tableau of its shape.
\end{defi}
For example, \((\ttab(125,3,4),\ttab(123,45))\) is
\(2\)-restricted but not \(2\)-minimal, \((\ttab(123,4,5),\ttab(123,45))\) is \(4\)-minimal,
and \((\ttab(125,3,4),\ttab(125,34))\) is \(3\)-minimal.

Let \(\mu,\,\lambda\in P(n)\), and let \((u,t)\in\STD(\mu)\times\STD(\lambda)\).
Let \(k\) be the restriction number of~\((u,t)\), and assume that
\(k\in[1,n-1]\) (or, equivalently, \(u\neq t\)).
Recall that
\begin{equation*}
F(u,t) = \{(v,x)\in C_k(u,t) \mid v^{-1}(k)=x^{-1}(k) \text{ lies between } u^{-1}(k+1) \text{ and } t^{-1}(k+1)\}.
\end{equation*}

\begin{defi}\label{k-minimalapproximate}
Let \(\mu,\lambda\in P(n)\) and \((u,t)\in\STD(\mu)\times\STD(\lambda)\)
with \(u\neq t\), and let \(k\) be the restriction number of~\((u,t)\).
We define \(A(u,t) = \{(v,x)\in F(u,t) \mid \col_x(k)=\col_t(k+1)-1\}\) and call any
element of \(A(u,t)\) an \textit{approximate} of~\((u,t)\).
\end{defi}
Note that \(A(u,t)\neq\emptyset\) if and only if \(\col_u(k+1)<\col_t(k+1)\).

\begin{rema}\label{k-1subclassofk}
Let \(u,\,t\) as above and assume that \(A(u,t)\ne\emptyset\). It is clear
from Definition~\ref{k-minimalapproximate} that every approximate
\((v,x)\) of~\((u,t)\) is \(k\)-restricted and satisfies \((v,x)\approx_k(u,t)\),
and that
\(A(u,t) = \{(v,x)\in C_k(u,t) \mid \col_{v}(k)=\col_{x}(k)=\col_t(k+1)-1\}\),
which is a (non-empty) \((k-1)\)-subclass of \(C_k(u,t)\). It follows
that if \(\kappa = \shape(\Lessthan {x}{(k-1)})=\shape(\Lessthan {v}{(k-1)})\), where
\((v,x)\in A(u,t)\), then
the bijection from \(\STD(\kappa)\) to \(A(u,t)\) given by \(w\mapsto (v,x)\) such that
\(\Lessthan {v}{(k-1)}=\Lessthan {x}{(k-1)}=w\) transfers the partial order
\(\leqslant\) from \(\STD(\kappa)\) to \(A(u,t)\).
The minimal element of
\(A(u,t)\), called the \textit{minimal approximate} of~\((u,t)\),
is the pair \((v,x)\) given by \(w=\tau_{\kappa}\), and
the maximal element of \(A(u,t)\), called the \textit{maximal approximate} of~\((u,t)\),
is the pair \((v,x)\) given by \(w=\tau^{\kappa}\).
\end{rema}

\begin{rema}\label{nonemptyAkst}
Let \(\mu,\,\lambda\in\Lambda\), and
let \((\beta,u)\in\mathcal{I}_{\mu}\times\STD(\mu)\) and
\((\alpha,t)\in\mathcal{I}_{\lambda}\times\STD(\lambda)\) satisfy the condition
\(\mu(c_{\beta,u},c_{\alpha,t}) \neq 0\) and \(\D(t) \subsetneqq \D(u)\).
Let \(k\in[1,n-1]\) be the restriction number of the pair
\((u,t)\). Let \(l=\min(\D(u)\setminus\D(t))\), and let
\(L=\{s_1,\ldots,s_l\}\). Remark~\ref{iplusonest}
applied to \(\Gamma\downarrow_L\), the \(W_L\)-graph
obtained by restricting \(\Gamma\) to \(W_L\), shows that
if \(u\neq s_{k+1}t>t\) then \(\col_u(k+1)<\col_t(k+1)\). Thus if \(u\neq s_{k+1}t>t\)
then the set \(A(u,t)\neq\emptyset\).
In particular,
if \(\alpha\neq\beta\) then,
since \(\mu(c_{\beta,u},c_{\alpha,t})\neq 0\) implies that \(\D(t) \subsetneqq \D(u)\), and
since \(\mu=\lambda\) and \(\mu(c_{\beta,u},c_{\alpha,t})\neq 0\) imply that \(u<t\)
by Proposition\ref{kldominancesamenolecule},
the condition \(\mu(c_{\beta,u},c_{\alpha,t})\neq 0\) is sufficient for
the set \(A(u,t)\neq\emptyset\).
\end{rema}

\begin{lemm}\label{k-minimallemma1}
Let \(\mu,\lambda\in P(n)\) and \((u,t)\in\STD(\mu)\times\STD(\lambda)\)
with \(u\neq t\), and let \(k\) be the restriction number of~\((u,t)\).
Assume that \(A(u,t)\neq\emptyset\), and let
\((v,x)\in A(u,t)\). Then \((v,x)\) is \(k\)-restricted and satisfies
\((v,x)\approx_k(u,t)\).
If, moreover, \(D(t)\subsetneqq\D(u)\), then we have
\(\D(x)\subsetneqq\D(v)\) with \(k=\min(\D(v)\setminus\D(x))\).
\end{lemm}

\begin{proof}
It follows from Remark~\ref{k-1subclassofk} that \((v,x)\) is \(k\)-restricted and
satisfies \((v,x)\approx_k(u,t)\).
It remains to show that \(\D(x)\subsetneqq\D(v)\) with \(k=\min(\D(v)\setminus\D(x))\)
if \(\D(t)\subsetneqq\D(u)\). So suppose further that
\(\D(t)\subsetneqq\D(u)\)

Since \(\col_u(k+1)<\col_t(k+1)\)
(since \(A(u,t)\neq\emptyset\)), Lemma~\ref{knuthmoveweakorderrelationLem}
and Lemma~\ref{knuthmoveweakorderrelationLemb} show that
\(\D(x)\setminus\D(v)=\D(t)\setminus\D(u)\) and
\(\D(v)\setminus\D(x)\supseteq\D(u)\setminus\D(t)\). Since
\(\D(t)\subsetneqq\D(u)\), this yields \(\D(x)\subsetneqq\D(v)\).
Now since \((v,x)\) is favourable, we have \(k=\min(\D(v)\oplus\D(x)\)
by Remark~\ref{k-reducedrem}, and it follows that \(k=\min(\D(v)\setminus\D(x))\).
\end{proof}

\begin{lemm}\label{k-minimallemma}
Let \(\mu,\lambda\in\Lambda\) with \(\mu\neq\lambda\), and let
\((\beta,u)\in\mathcal{I}_{\mu}\times\STD(\mu)\) and
\((\alpha,t)\in\mathcal{I}_{\lambda}\times\STD(\lambda)\) satisfy
\(\mu(c_{\beta,u},c_{\alpha,t}) \neq 0\). Let \(k\) be the
restriction number of \((u,t)\). Then \(A(u,t)\neq\emptyset\), and for
all \((v,x)\in A(u,t)\) the following three conditions hold:
\begin{enumerate}[label=(\roman*),topsep=1 pt]
\item
\((v,x)\approx (u,t)\),
\item
\(\D(x)\subsetneqq\D(v)\) and \(k=\min(\D(v)\setminus\D(x))\),
\item
\(\mu(c_{\beta,v},c_{\alpha,x})=\mu(c_{\beta,u},c_{\alpha,t})\).
\end{enumerate}
\end{lemm}

\begin{proof}
We have \(A(u,t)\neq\emptyset\) by Remark~\ref{nonemptyAkst}. Let
\((v,x)\in A(u,t)\), then by Lemma~\ref{k-minimallemma1}, we have
\((u,t)\approx_k (v,x)\), whence \((u,t)\approx (v,x)\).
Moreover, since \(\mu\neq\lambda\), and since \(\mu(c_{\beta,u},c_{\alpha,t}) \neq 0\),
we have \(\D(t)\subsetneqq\D(u)\), and it follows by Lemma~\ref{k-minimallemma1}
that \(\D(x)\subsetneqq\D(v)\) with \(k=\min(\D(v)\setminus\D(x))\).
It remains to show that
\(\mu(c_{\beta,v},c_{\alpha,x})=\mu(c_{\beta,u},c_{\alpha,t})\).
Let \(l=\min(\D(u)\setminus\D(t))\). Since \((u,t)\) is \(k\)-restricted, we have
\(k\leqslant l\).

Suppose first that \(k<l\). Since \((u,t)\approx_k (v,x)\) and
\(k<l\in\D(u)\setminus\D(t)\), the result follows from Lemma~\ref{noname0}.

Suppose now that \(k=l=\min(\D(u)\setminus\D(t))\), in particular,
this shows that \(k\in\D(u)\setminus\D(t)\). Let \(w=\Lessthan tk=\Lessthan uk\in\STD(\xi)\),
where \(\xi=\shape(w)\), and let \((h,q)=t^{-1}(k+1)\) and \((g,p)=t^{-1}(k)\),
the boxes of \(t\) that contain \(k+1\) and \(k\) respectively. Since
\(k\notin\D(t)\), it follows that \(g\geqslant h\) and \(p<q\). If \(p=q-1\) then
we have \((u,t)\in A(u,t)\). Since \((u,t)\approx_{k-1} (v,x)\) by
Remark~\ref{k-1subclassofk} and since
\(k\in\D(u)\setminus\D(t)\), we have
\(\mu(c_{\beta,v},c_{\alpha,x})=\mu(c_{\beta,u},c_{\alpha,t})\)
by Lemma~\ref{noname0}. Thus, we can assume that \(p<q-1\).

Let \((d,m)=(\xi_{q-1},q-1)\), and note that the assumption implies that
\(g>d\geq h>\xi_{q}\). It is clear that \((g,p)\) and \((d,m)\) are \(\xi\)-removable,
and \((g,p)\neq (d,m)\). Let \(\zeta\in P(k-2)\) such that
\([\zeta]=[\xi]\setminus \{(g,p),(d,m)\}\), and let \((i,j)\) be a
\(\zeta\)-removable box that lies between \((g,p)\) and \((d,m)\)
(in the sense that \(g>i\geqslant d\) and \(p\leqslant j<m\)).
We can choose \(w'\in\STD(\xi)\) with \(w'(i,j)=k-2\), \(w'(d,m)=k-1\) and
\(w'(g,p)=k\), and define \((u_1,t_1)\) by \(\Lessthan {u_1}k = w'\) and
\(\morethan {u_1}k = \morethan {u}k\), and \(\Lessthan {t_1}k = w'\) and
\(\morethan {t_1}k = \morethan tk\). Since
\[(u_1,t_1)\in \{(v,x)\in C_k(u,t) \mid \col_v(k)=\col_x(k)=g\},\]
the
\((k-1)\)-subclass of \(C_k(u,t)\), it follows that
\((u,t)\approx_{k-1}(u_1,t_1)\).
and it follows by  Lemma~\ref{noname0} that
\(\mu(c_{\beta,u_1},c_{\alpha,t_1})=\mu(c_{\beta,u},c_{\alpha,t})\) and
\(k\in\D(u_1)\setminus\D(t_1)\).

Since \(p<m\), we have \(k-1\in\SD(w')\), and so \(k-1\in\SD(u_1)\) and
\(k-1\in\SD(t_1)\). It follows that we can define
\((u_2,t_2)\in\STD(\mu)\times\STD(\lambda)\) by \(u_2=s_{k-1}u_1\) and \(t_2=s_{k-1}t_1\),
and we note that \(w''=\Lessthan {u_2}k = \Lessthan {t_2}k = s_{k-1}w'\), and
\(\morethan {u_2}k = \morethan {u_1}k\) and \(\morethan {t_2}k = \morethan {t_1}k\).
Since
\(w'=s_{k-1}w''>w''\), and \(\D(w'')\cap \{k-2,k-1\}=\{k-2\}\)
and \(\D(w')\cap \{k-2,k-1\} = \{k-1\}\), it follows that there is a dual
Knuth move (of the first kind) of index~\(k-1\) taking \(w''\) to~\(w'\). As the
same dual Knuth move takes \((u_2,t_2)\) to \((u_1,t_1)\), we have
\((u_1,t_1)\) and \((u_2,t_2)\) are related by a paired
\(\leqslant k\)-dual Knuth relation indexed by~\((k-1)\). Moreover, it can be
verified easily that
\begin{align*}
\qquad&&\D(t_1)\cap \{k-2,k-1,k\} &= \{k-1\},&\D(u_1)\cap \{k-2,k-1,k\} &= \{k-1,k\},&&\qquad\\
\qquad&&\D(t_2)\cap \{k-2,k-1,k\} &= \{k-2\},&\D(u_2)\cap \{k-2,k-1,k\} &= \{k-2,k\}.&&\qquad,
\end{align*}
and it follows by Proposition~\ref{arctransport} that
\(\mu(c_{\beta,u_2},c_{\alpha,t_2})=\mu(c_{\beta,u_1},c_{\alpha,t_1})\).

Finally, since it is clear that \((u_2,t_2)\in A(u,t)\), we have
\((u_2,t_2)\approx_{k-1} (v,x)\) by Remark~\ref{k-1subclassofk},
and since \(k\in\D(u_2)\setminus\D(t_2)\), it follows by Lemma~\ref{noname0} that
\(\mu(c_{\beta,v},c_{\alpha,x})=\mu(c_{\beta,u_2},c_{\alpha,t_2})\). Thus,
 \(\mu(c_{\beta,v},c_{\alpha,x})=\mu(c_{\beta,u},c_{\alpha,t})\). as required.
\end{proof}

\begin{prop}\label{cellorder}
Let \(\lambda\in\Lambda\) satisfy the condition that
\(\Ini_{\lambda}(\Gamma)\neq\emptyset\), and let \(t'=t_{\Gamma,\lambda}\).
Let \((\alpha,t')\in\mathcal{I}_{\lambda}\times\STD(\lambda)\) and \((\beta,u')\in\mathcal{I}_{\mu}\times\STD(\mu)\),
where \(\mu\in\Lambda\setminus \{\lambda\}\), satisfy the condition that
\(\mu(c_{\beta,u'},c_{\alpha,t'})\neq 0\).
Let \(k\) be the restriction number of~\((u',t')\), and let
\((u,t)\in A(u',t')\).
Then \(\Morethan tk\) is \(k\)-critical. Thus if \((u,t)\) is
the minimal approximate of \((u',t')\) then \(t\) is \(k\)-minimal
with respect to~\(u\).
\end{prop}

\begin{proof}
Lemma~\ref{k-minimallemma} tells us that \((u,t)\approx (u',t')\), that
\(\D(t)\subsetneqq\D(u)\) and \(k=\min(\D(u)\setminus\D(t))\),
and that \(\mu(c_{\beta,u},c_{\alpha,t})=\mu(c_{\beta,u'},c_{\alpha,t'})\neq 0\).
Note that \(\col_t(k+1)=\col_t(k)+1\), since \((u,t)\) is an approximate of
\((u',t')\) (see Definition~\ref{k-minimalapproximate}). Thus, by Remark~\ref{critical-tableau}, to show that \(\Morethan tk\)
is \(k\)-critical it will suffice to show that every \(j\in\D(t)\) with
\(j>k+1\) is in \(\WD(t)\), and that either \(\col_t(k+2)=\col_t(k)\)
or \(k+1\notin\SD(t)\). We do both parts of this by contradiction.

For the first part, suppose that \(j>k+1\) and \(j\in\SD(t)\).
Since \(j\in\D(t)\) and \(\D(t)\subsetneqq\D(u)\), we have
\(j\in\D(t)\cap\D(u)\), and since \(k\in\D(u)\setminus\D(t)\),
it follows that \(j\in\D(t)\) and \(k\notin\D(t)\),
and \(k,j\in\D(u)\). Let \(v=s_jt\), which is standard since \(j\in\SD(t)\).
It follows by Lemma~\ref{variousaltpaths} (i) that \(k,j\notin\D(v)\).
Moreover, since \(\mu(c_{\alpha,t},c_{\alpha,v})=1\) by Corollary~\ref{arweightone},
and since \(\mu(c_{\beta,u},c_{\alpha,t})\neq 0\), it follows that
\((c_{\alpha,v},c_{\alpha,t},c_{\beta,u})\) is an alternating directed path of
type \((j,k)\).

Recall that if \(\mu(c_{\beta,u},c_{\alpha,t})\neq 0\) then
\(\mu(c_{\beta,u},c_{\alpha,t})> 0\), because \(\Gamma\) is admissible.
Thus \(N^{2}_{j,k}(\Gamma;v,u)>0\), whence
\(N^{2}_{k,j}(\Gamma;v,u)>0\), as~\(\Gamma\) satisfies the \(W_n\)-Polygon
Rule. So there exists at least one \(\nu\in\Lambda\) and
\((\gamma,y)\in\mathcal{I}_{\nu}\times\STD(\nu)\)
such that \((c_{\alpha,v},c_{\gamma,y},c_{\beta,u})\)
is an alternating directed path of type \((k,j)\).
Since \(\morethan {t'}k=\morethan tk\) and~\(k<j-1\),
we have \(v<_{\lex}t'\) by Lemma~\ref{xlexlessthant}. Thus, if
\(\nu\neq\lambda\) then we have \((\alpha,v)\in\Ini_\lambda(\Gamma)\)
and \(v\in\bigcup_{\alpha\in\mathcal{I}_\lambda}\Ini_{(\alpha,\lambda)}(\Gamma)\), and
so this contradicts the assumption that \(t'=t_{\Gamma,\lambda}\). It follows that
\(\nu=\lambda\) and \(y\in\STD(\lambda)\).

By Proposition~\ref{kldominancesamenolecule}, we must
have either \(\gamma=\alpha\) and \(y=s_kv>v\)
or \(y<v\). Recall that \(s_kv\in\STD(\lambda)\) and \(s_kv>v\) if and only if
\(k\in\SA(v)\). Thus in the case \(\gamma=\alpha\) and \(y=s_kv>v\), then since
\(\morethan {t'}k=\morethan tk\) and \(k<j-1\) we have
\(y=s_kv <_{\lex}t'\) by Lemma~\ref{xlexlessthant} (i),
while in the case \(y<v\), then since \(\morethan {t'}k=\morethan tk\) and
\(k<j-1\) we have \(y<_{\lex}t'\) by Lemma~\ref{xlexlessthant} (ii).
In either case, since \((\gamma,y)\in\Ini_\lambda(\Gamma)\)
and \(y\in\bigcup_{\alpha\in\mathcal{I}_\lambda}\Ini_{(\alpha,\lambda)}(\Gamma)\),
this contradicts the assumption that \(t'=t_{\Gamma,\lambda}\).

For the second part, suppose that \(k+1\in\SD(t)\) and
\(\col_t(k+2)\neq\col_t(k)\).

\begin{Case}{1.} Suppose that \(\col_t(k)<\col_t(k+2)\).
Since
\((u,t)\in A(u',t')\),  we have \(\col_t(k)=\col_t(k+1)-1\),
and it follows that \(\col_t(k+1)\leqslant\col_t(k+2)\).
This contradicts the assumption that \(k+1\in\SD(t)\).
\end{Case}

\begin{Case}{2.} Suppose that \(\col_t(k+2)<\col_t(k)\).
Since \(k+1\in\SD(t)\subseteq\D(t)\) and \(\D(t)\subsetneqq\D(u)\),
it follows that \(k+1\in\D(t)\cap\D(u)\),
and since \(k\in\D(u)\setminus\D(t)\), it follows that
\(k+1\in\D(t)\) and \(k\notin\D(t)\),
and \(k,k+1\in\D(u)\).
Let \(v=s_{k+1}t\). Since \(k+1\in\SD(t)\)), we have \(v\in\STD(\lambda)\).
Let \(w=s_kv\). Since \(k\in\SD(v)\) by Lemma~\ref{variousaltpaths} (iii),
we have \(w\in\STD(\lambda)\). Now, it follows by
Lemma~\ref{variousaltpaths} (iii) that \(k\in\D(v)\) and \(k+1\notin\D(v)\),
and \(k\notin\D(w)\) and \(k+1\notin\D(w)\).

Moreover, since \(\mu(c_{\alpha,v},c_{\alpha,w})=\mu(c_{\alpha,t},c_{\alpha,v})=1\)
by Corollary~\ref{arweightone}, and since it is also true that
\(\mu(c_{\beta,u},c_{\alpha,t})\neq 0\), it follows that
\((c_{\alpha,w},c_{\alpha,v},c_{\alpha,t},c_{\beta,u})\) is
an alternating directed path of type \((k,k+1)\).

As recalled above, if \(\mu(c_{\beta,u},c_{\alpha,t})\neq 0\) then
\(\mu(c_{\beta,u},c_{\alpha,t})> 0\), because \(\Gamma\) is admissible.
Thus \(N^{3}_{k,k+1}(\Gamma;w,u)>0\), whence
\(N^{3}_{k+1,k}(\Gamma;w,u)>0\), as~\(\Gamma\) satisfies the \(W_n\)-Polygon
Rule.

So there exist \(\xi\in\Lambda\) and
\((\delta,x)\in\mathcal{I}_{\xi}\times\STD(\xi)\), and
\(\nu\in\Lambda\) and \((\gamma,y)\in\mathcal{I}_{\nu}\times\STD(\nu)\)
such that \((c_{\alpha,w}, c_{\delta,x}, c_{\gamma,y}, c_{\beta,u})\)
is an alternating directed path of type \((k+1,k)\).

Since \(\morethan {t'}k=\morethan tk\), we have \(w<_{\lex}t'\) by
Lemma~\ref{xlexlessthant} (iii). Thus, if
\(\xi\neq\lambda\) then we have \((\alpha,w)\in\Ini_\lambda(\Gamma)\)
and \(w\in\bigcup_{\alpha\in\mathcal{I}_\lambda}\Ini_{(\alpha,\lambda)}(\Gamma)\), and
so this contradicts the assumption that \(t'=t_{\Gamma,\lambda}\). It follows that
\(\xi=\lambda\) and \(x\in\STD(\lambda)\).

Since \(\mu(c_{\gamma,y},c_{\delta,x})\neq 0\), and since
\(\D(x)\cap\{k,k+1\}=\{k+1\}\) and \(\D(y)\cap\{k,k+1\}=\{k\}\), we
have \(\{c_{\delta,x},c_{\gamma,y}\}\) is a simple edge by the \(W_n\)-Simplicity Rule.
Thus \(\nu=\lambda\) and \(\gamma=\delta\), and
\(y\) and \(x\) are related by a dual Knuth move. We have
either \(\delta=\alpha\) and \(x=s_{k+1}w>w\) or \(x<w\) by
Proposition~\ref{kldominancesamenolecule}, and
\(y\) is the unique \(k\)-neighbour of \(x\). If \(x=s_{k+1}w>w\),
then since \(x\in\STD(\lambda)\), this is equivalent to
\(k+1\in\SA(w)\). It follows by Lemma~\ref{variousaltpaths} (iii) that
\(y=s_kx>x\) is the unique \(k\)-neighbour of \(x\).
In this case, since \(\morethan {t'}k=\morethan tk\), we have
\(y<_{\lex}t'\) by Lemma~\ref{xlexlessthant} (iii). If
\(x<w\) and \(y\) is the unique \(k\)-neighbour of \(x\), then since
\(\morethan {t'}k=\morethan tk\), we have
\(y<_{\lex}t'\) by Lemma~\ref{xlexlessthant} (iv).
In either case, since \((\gamma,y)\in\Ini_\lambda(\Gamma)\)
and \(y\in\bigcup_{\alpha\in\mathcal{I}_\lambda}\Ini_{(\alpha,\lambda)}(\Gamma)\),
this contradicts the assumption that \(t'=t_{\Gamma,\lambda}\).
\end{Case}

If \((u,t)\) is the minimal approximate of
\((u',t')\), then it is clear that \(t\) is \(k\)-minimal
with respect to \(u\) in accordance with Definition~\ref{k-minimalappr}.
\end{proof}

\begin{coro}\label{talphanystructure}
Let \(\lambda\in\Lambda\) satisfy the condition that
\(\Ini_{\lambda}(\Gamma)\neq\emptyset\), and let \(t'=t_{\lambda}\).
Let \((\alpha,t')\in\mathcal{I}_{\lambda}\times\STD(\lambda)\) and \((\beta,u')\in\mathcal{I}_{\mu}\times\STD(\mu)\),
where \(\mu\in\Lambda\setminus \{\lambda\}\), satisfy the condition that
\(\mu(c_{\beta,u'},c_{\alpha,t'})\neq 0\). Let \(k\) be the restriction number of~\((u',t')\).
Then \(\morethan {t_{\lambda}}{(k+1)}\) is minimal and if \(k+1\in\SD(t_{\lambda})\) then
\(\col_{t_{\lambda}}(k+1)=\col_{t_{\lambda}}(k+2)+1\).
\end{coro}

\begin{proof}
Let \((u,t)\in A(u',t')\).
Then \(\Morethan tk\) is \(k\)-critical, by Proposition~\ref{cellorder}. Now
since \(\morethan tk=\morethan {t_{\lambda}}k\), this shows that
\(\morethan {t_{\lambda}}{(k+1)}\) is minimal and if \(k+1\in\SD(t_{\lambda})\) then
\(\col_{t_{\lambda}}(k+1)=\col_{t_{\lambda}}(k+2)+1\).
\end{proof}

\begin{lemm}\label{1-minimalcomp}
Let \(n\geq 2\), and let \(\mu,\,\lambda\in P(n)\).
Let \(t\in\STD(\lambda)\) and \(u\in\STD(\mu)\) and suppose that \(t\) is \(1\)-minimal with
respect to \(u\). Then \(\mu<\lambda\).
\end{lemm}

\begin{proof}
Since \((u,t)\) is \(1\)-minimal, we have
\(t(1,1)=u(1,1)=1\) and \(t(1,2)=u(2,1)=2\). So if \(n=2\), we have
\(\mu=(2)<(1,1)=\lambda\). We proceed inductively on \(n\geq 3\).
If \(t(1,3)=3\) then since \(t\) is \(1\)-minimal, we have \(t=\tau_{\lambda}\),
where \(\lambda=(1,\ldots,1)\). Since \(\lambda=\max((P(n),\leqslant))\), and
since \(\mu_1>1=\lambda_1\), we deduce that \(\mu<\lambda\). We may just assume that
\(t(2,1)=3,\ldots,t(\lambda_1,1)=\lambda_1+1\), and it follows that
\(2,\ldots,\lambda_1\in\D(t)\). Now since \(\D(t)\subsetneqq\D(u)\), we have
\(2,\ldots,\lambda_1\in\D(u)\), and so, \(u(3,1)=3,\ldots,u(\lambda_1+1,1)=\lambda_1+1\). In
particular, this shows \(\mu_1>\lambda_1\).

Let \(\eta=\shape(\Lessthan u{(n-1)})\) and let \(\theta=\shape(\Lessthan t{(n-1)})\).
It is clear that  \(\Lessthan t{(n-1)}\) is \(1\)-minimal with respect to
\(\Lessthan u{(n-1)}\), whence
\(\eta<\theta\) by the inductive hypothesis. We shall show that
\(\col_u(n)\leqslant\col_t(n)\). Suppose to the contrary that
\(\col_t(n)<\col_u(n)\).

Suppose first that \(\col_t(n-1)=1\) so that \(1,3,\ldots,n-1\) fill column~\(1\)
of~\(t\). Since \(\D(t)\subsetneqq\D(u)\), we have \(1,2,3,\ldots,n-1\) fill column~\(1\) of~\(u\).
Since \(\col_u(n)>1\), we have \(u(1,2)=n\), and it follows that
\(n-1\in\A(u)\). Now since
\(\D(t)\subsetneqq\D(u)\), we have \(n-1\in\A(t)\), consequently \(\col_t(n)>\col_t(n-1)=1\).
It follows that \(\col_t(n)\geqslant 2=\col_u(n)\), contradicting our assumption.

Suppose now that \(\col_t(n-1)>1\). Let \(1<q=\col_t(n-1)\leqslant\col_t(n)\).
Since \(\eta\leqslant\theta\), we have \(n-1=\sum_{m=1}^q \theta_m \leqslant\sum_{m=1}^q \eta_m\),
and so, if \(i<n\) then \(\col_u(i)\leqslant\col_t(n-1)\). It follows that
\(\col_u(n)\leqslant q+1=\col_t(n-1)+1<\col_t(n)+1\), whence
\(\col_u(n)\leqslant\col_t(n)\), if \(n-1\in\WA(u)\), and
\(\col_u(n)\leqslant\col_u(n-1)\leqslant\col_t(n-1)\leqslant\col_t(n)\), if \(n-1\in\D(u)\).
Either case contradicts our assumption.

Since \(\eta<\theta\) and \(\col_u(n)\leqslant\col_t(n)\), we have \(\mu\leqslant\lambda\)
by Lemma~\ref{nsbeforent}, and since \(\mu_1>\lambda_1\), we obtain \(\mu<\lambda\).
\end{proof}

\begin{lemm}\label{k-minimalonerow}
Let \(\lambda\in\Lambda\) satisfy the condition that
\(\Ini_{\lambda}(\Gamma)\neq\emptyset\). Let
\((\alpha,t')\in\Ini_{\lambda}(\Gamma)\) with \(t'=t_{\lambda}\), and
let \(\mu\in\Lambda\setminus\{\lambda\}\) and
\((\beta,u')\in\mathcal{I}_{\mu}\times\STD(\mu)\) such that
\(\mu(c_{\beta,u'},c_{\alpha,t'})\neq 0\).
Let \((u,t)\in\STD(\mu)\times\STD(\lambda)\), and let \(k\geqslant 3\)
be the restriction number of~\((u,t)\). Suppose that \((u,t)\) satisfies
\[
\Lessthan tk=\Lessthan uk\ =\ \vcenter{\offinterlineskip\small
\hbox{\vrule height 0.4 pt depth 0 pt width 108.4pt}
\hbox{\vrule height 9 pt depth 4 pt\hbox to 13pt{\hfil\(1\)\hfil}\vrule
\hbox to 13pt{\hfil\(2\)\hfil}\vrule\hbox to 30pt{\hfil\(\cdots\)\hfil}\vrule
\hbox to 25pt{\hfil\(k-2\)\hfil}\vrule
\hbox to 25pt{\hfil\(k-1\)\hfil}\vrule}
\hrule
\hbox{\vrule height 9 pt depth 4 pt\hbox to 13pt{\hfil\(k\)\hfil}\vrule}
\hrule width 14.2 pt}\ ,
\]
and \(t\) satisfies further properties that
\(\col_t(n)=k-1\), and
\(u\) satisfies further properties that \(u(1,k)=n\) and
\((2,k-1)\notin [\mu]\).
Then \((u,t)\notin A(u',t')\).
\end{lemm}

\begin{proof}
Assume to the contrary that \((u,t)\in A(u',t')\). By Remark~\ref{k-1subclassofk},
both \((u',t')\) and \((u,t)\) have the same restriction number, and \(A(u',t')\) consists of
\((v,x)\in\STD(\mu)\times\STD(\lambda)\) such that
\(\Lessthan v{(k-1)}=\Lessthan x{(k-1)}\in\STD((1^{k-1}))\), and
\(\Morethan vk=\Morethan uk\) and \(\Morethan xk=\Morethan tk\). Thus it follows that
\((u',t')\) is \(k\)-restricted, and \(A(u',t')=\{(u,t)\}\).

It follows by Lemma~\ref{k-minimallemma} that
\((u,t)\approx (u',t')\), \(\D(t)\subsetneqq\D(u)\) with \(k=\min(\D(u)\setminus\D(t))\), and \(\mu(c_{\beta,u},c_{\alpha,t})=\mu(c_{\beta,u'},c_{\alpha,t'})\neq 0\), and it follows by
Lemma~\ref{cellorder} that \(\Morethan tk\) is \(k\)-critical.
Since \(\col_t(k+1)=\col_t(k)+1\), we have \(t(2,2)=k+1\), and since
\(k\in\D(u)\setminus\D(t)\), we have \(\col_u(k+1)\leqslant\col_u(k)\), hence \(\col_u(k+1)=1\),
and it follows that \(u(3,1)=k+1\). Since \(u(1,k)=n\), it follows further that
\(k+1<n\).

\Case{1.} Suppose that \((u,t)=(u',t')\).

Since \(k\geqslant 3\), we have
\(\col_u(k)=1<k-1=\col_u(k-1)\), and so
\(k-1\in\SD(u)\subseteq\D(u)\). Let \(v=s_{k-1}u\). Since
\(k-1\in\SD(u)\), it follows that
\(v\in\STD(\mu)\) and \(k-1\notin\D(v)\).
Since \(\col_u(k-2)=k-2<k-1=\col_u(k-1)\), it follows that
\(k-2\notin\D(u)\). Moreover,
since \(v\) is obtained from \(u\) by switching the positions of \(k-1\) and \(k\),
and since \(k\geqslant 3\), we have
\(\col_v(k-1)=\col_u(k)=1\leqslant k-2=\col_u(k-2)=\col_v(k-2)\),
and so \(k-2\in\D(v)\).Thus there is a dual
Knuth move (of the first kind) of index~\(k-1\) taking \(v\) to \(u\),
which shows that \(\{c_{\beta,u},c_{\beta,v}\}\) is a simple edge in \(\Gamma\).

Since \(k\geqslant 3\), we have \(1\leqslant k-2=\col_t(k-2)<k-1=\col_t(k-1)\),
and it follows that \(k-2\notin\D(t)\). Since \(k\in\D(u)\setminus\D(t)\), we
also have \(k\notin\D(t)\). Similarly, since \(\Lessthan uk = \Lessthan tk\),
we have \(k-2\notin\D(u)\), but since
\(k\in\D(u)\setminus\D(t)\), we have \(k\in\D(u)\).
We have shown that \(k-2\in\D(v)\). Now since
\(\col_v(k+1)=\col_u(k+1)=1<k-1=\col_u(k-1)=\col_v(k)\), as
\(v\) is obtained from \(u\) by switching the positions of \(k-1\) and \(k\),
we also have \(k\in\D(v)\).
Moreover, since \(\mu(c_{\beta,u},c_{\alpha,t})\neq 0\) and
\(\mu(c_{\beta,v},c_{\beta,u})=1\) (as
\(\{c_{\beta,u},c_{\beta,v}\}\) is a simple edge),
it follows that \((c_{\alpha,t},c_{\beta,u},c_{\beta,v})\) is
an alternating directed path of type~\((k,k-2)\).

Since \(N^{2}_{k-2,k}(\Gamma;t,v)=N^{2}_{k,k-2}(\Gamma;t,v)\), as
\(\Gamma\) satisfies the \(W_n\)-Polygon Rule,
and since
\(N^{2}_{k,k-2}(\Gamma;t,v)\geqslant \mu(c_{\beta,u},c_{\alpha,t})\),
it follows that \(N^{2}_{k-2,k}(\Gamma;t,v)>0\),
whence there are \(\xi\in\Lambda\) and \((\gamma,x)\in\mathcal{I}_{\xi}\times\STD(\xi)\)
such that \((c_{\alpha,t},c_{\gamma,x},c_{\beta,u})\)
is an alternating directed path of type \((k-2,k)\).

Since \(\row_t(k-2)=\row_t(k-1)\), we have \(k-2\in\WA(t)\),
and so \(k-2\notin\D(t)\), and since \(\Lessthan tk = \Lessthan uk\), we have
\(k-1\in\D(t)\), since \(k-1\in\D(u)\). Thus if \(k-1\notin\D(x)\) then
\(\{c_{\alpha,t},c_{\gamma,x}\}\) is a simple edge. That is, if \(k-1\notin\D(x)\) then
\(\alpha=\gamma\) and \(t\) and \(x\) are related by a dual Knuth move. Moreover, since
\(s_{k-2}t\notin\STD(\lambda)\), this shows that
\(x=s_{k-1}t\). But then since \(k\geqslant 3\) and since \(x\) is obtained from \(t\)
by switching the positions of \(k-1\) and \(k\), we have
\(\col_x(k+1)=\col_t(k+1)=2\leq k-1=\col_t(k-1)=\col_x(k)\),
and it follows that \(k\in\D(x)\), contradicting the requirement that \(k\notin\D(x)\).
Thus \(k-1\in\D(x)\).

Since \(\mu(c_{\beta,v},c_{\gamma,x})\neq 0\), and since
\(\D(x)\cap\{k-1,k\}=\{k-1\}\) and \(\D(v)\cap\{k-1,k\}=\{k\}\), the
\(W_n\)-Simplicity Rule shows that \(\{c_{\beta,v},c_{\gamma,x}\}\) is a
simple edge. Equivalently, \(\gamma=\beta\), and \(x\) and \(v\)
are related by a dual Knuth move. Indeed, \(x\) is, in this case,
the \((k-1)\)-neighbour of~\(v\). But the \((k-1)\)-neighbour of~\(v\) is
\(s_kv\), since \(\col_v(k+1)=\col_v(k-1)<\col_v(k)\). Therefore, \(x=s_kv\).

It can be seen that \((\alpha,t)=(\alpha,t')\) and \((\beta,s_kv)\) satisfy the
conditions of Corollary~\ref{talphanystructure}.
Since it is clear that \((s_kv,t')=(s_kv,t)\) is \((k-2)\)-restricted, and since
\(\col_{t}(k)<\col_{t}(k-1)\), we have by Corollary~\ref{talphanystructure} that
\(k-1=\col_t(k-1)=\col_t(k)+1=2\). Thus, \(k=3\), and so \(\col_t(n)=2\). Since
\(\morethan t{(k-1)}\) is the minimal tableau of its shape by Corollary~\ref{talphanystructure},
we have \(\col_t(3)\leqslant\col_t(4)\leqslant \cdots \leqslant\col_t(n)\), and so
\(\col_t(3)=1\) and \(\col_t(4)=\cdots = \col_t(n)=2\). Thus \(\lambda_1=2\) and \(\lambda_2=n-2\),
and  since \(\lambda_1\geqslant \lambda_2\), it follows that \(n\leqslant 4\).
This contradicts the fact that \(n>k+1\) (as shown earlier).

\Case{2. }Suppose that \((u,t)\neq(u',t')\).

Since \((u,t)\approx_k(u',t')\) by Lemma~\ref{k-minimallemma1}, there exists
\(z\in W_k\setminus \{1\}\) with \(u'=zu\) and \(t'=zt\). Hence there is an \(i\in[1,k-2]\)
such that \(u'\) and \(t'\) satisfy
\[
w'=\Lessthan {t'}{k}=\Lessthan {u'}k\ =\ \vcenter{\offinterlineskip\small
\hbox{\vrule height 0.4 pt depth0 pt width 152.4pt}
\hbox{\vrule height 9 pt depth 4 pt\hbox to 25pt{\hfil\(1\)\hfil}\vrule
\hbox to 30pt{\hfil\(\cdots\)\hfil}\vrule\hbox to 13pt{\hfil\(i\)\hfil}\vrule
\hbox to 25pt{\hfil\(i+2\)\hfil}\vrule\hbox to 30pt{\hfil\(\cdots\)\hfil}\vrule\
\hbox to 25pt{\hfil\(k\)\hfil}\vrule}
\hrule
\hbox{\vrule height 9 pt depth 4 pt\hbox to 25pt{\hfil\(i+1\)\hfil}\vrule}
\hrule width 25 pt}\ ,
\]
and, furthermore, \(\morethan {t'}k=\morethan tk\) and \(\morethan {u'}k=\morethan uk\).

If \(k=3\) then since \((2,2)\notin [\mu]\) and \(u(1,3)=n\), we have
\(\mu_2=1\) and \(\mu_3=1\). Therefore \(\mu=(n-2,1,1)\) and
the first row of \(u\) is
\(\vcenter{\offinterlineskip
\hrule
\hbox{\vrule height 8 pt depth 2 pt\hbox to 15pt{\hfil\(1\)\hfil}\vrule
\hbox to 15pt{\hfil\(2\)\hfil}\vrule
\hbox to 15pt{\hfil\(n\)\hfil}\vrule}
\hrule}\,\),
while \(u(2,1)=3\) and \(u(i,1)=i+1\) for \(i\in[3, n-2]\). Since
\(\col_u(n-1)=1<3=\col_u(n)\), we have \(n-1\notin\D(u)\), and
since \(\D(t)\subsetneqq\D(u)\) it follows that \(n-1\notin\D(t)\). That is,
\(\col_t(n-1)<\col_t(n)\). Now \(\morethan t{(k+1)}\) is the minimal tableau
of it shape, by Corollary~\ref{talphanystructure}, and so
\[
\col_t(5)\leqslant\col_t(6)\leqslant\cdots\leqslant\col_t(n-1)<\col_t(n).
\]
Thus \(\col_t(k+2)=\col_t(5)=\cdots=\col_t(n-1)=1\)
and \(\col_t(n)=2\). Hence \(\lambda=(n-3,3)\), the first three rows of \(t\) are
\(\vcenter{\offinterlineskip
\hrule
\hbox{\vrule height 8 pt depth 2 pt\hbox to 15pt{\hfil\(1\)\hfil}\vrule
\hbox to 15pt{\hfil\(2\)\hfil}\vrule}
\hrule}\,\),
\(\vcenter{\offinterlineskip
\hrule
\hbox{\vrule height 8 pt depth 2 pt\hbox to 15pt{\hfil\(3\)\hfil}\vrule
\hbox to 15pt{\hfil\(4\)\hfil}\vrule}
\hrule}\,\)
and
\(\vcenter{\offinterlineskip
\hrule
\hbox{\vrule height 8 pt depth 2 pt\hbox to 15pt{\hfil\(5\)\hfil}\vrule
\hbox to 15pt{\hfil\(n\)\hfil}\vrule}
\hrule}\,\)
and \(t(i,1)=i+2\) for \(i\in [4, n-2]\).
Now since \(C_k(u,t)=\{(u,t),(s_2u,s_2t)\}\),
we have \((u',t')=(s_2u,s_2t)\). It can be verified easily that
\(\D(s_2u)=\D(s_2t)=\{1,3,\ldots, n-2\}\), and it follows that
\(\mu(c_{\beta,s_2u},c_{\alpha,s_2t})=0\). This is in contradiction to
the assumption  that \(\mu(c_{\beta,u'},c_{\alpha,t'})\neq 0\).
Henceforth, we may assume that \(k\geqslant 4\).

Since \(\Lessthan {t'}{k}=\Lessthan {u'}k\), we have \(\D(t')\cap[1,k-1]=\D(u')\cap[1,k-1]\).
Since \(k\geqslant 4\), one the one hand, we have
\(\col_{t'}(k+1)=\col_t(k+1)=2<k-1=\col_t(k-1)=\col_{t'}(k)\),
and on the other hand, we have  \(\col_{u'}(k+1)=\col_u(k+1)=1<k-1=\col_u(k-1)=\col_{u'}(k)\).
Thus, it follows that \(k\in\D(t')\cap\D(u')\), and so
\(\D(t')\cap[1,k]=\D(u')\cap[1,k]\). Let \(l\in\D(u')\setminus\D(t')\),
which is not an empty set since \(\D(t')\subsetneqq\D(u')\).
This shows that \(l>k\).

We claim that \(i=1\).
Suppose to the contrary that \(i>1\).
Now since \(i>1\), it follows that
\(\col_{w'}(i+1)=1\leqslant i-1=\col_{w'}(i-1)<\col_{w'}(i-1)+1=\col_{w'}(i)\),
and so \(i\in\SD(w')\subseteq\D(w')\) and \(i-1\notin\D(w')\). Since
\(i\in\SD(w')\), we have \(s_iw'\) is standard and \(i\notin\D(s_iw')\). Moreover,
since \(\col_{s_iw'}(i)=\col_{w'}(i+1)<\col_{w'}(i)=\col_{s_iw'}(i-1)\), it follows
that \(i-1\notin\D(s_iw')\).
Thus \(s_iw'\to^{*1} w'\) with the index \(i\),
and since the same dual Knuth move takes \((s_iu',s_it')\) to \((u',t')\), we have
\((s_iu',s_it')\approx_k(s',t')\). It follows by Lemma~\ref{noname0} that
\(\mu(c_{\beta,s_iu'},c_{\alpha,s_it'})=\mu(c_{\beta,u'},c_{\alpha,t'})\neq 0\).
Since \(s_it'<t'\), it follows from
Corollary~\ref{domimplieslextableau} that
\(s_it'<_{lex}t'\). But \((\alpha,s_it')\in\Ini_{\lambda}(\Gamma)\) and
\(s_it'\in\bigcup_{\alpha\in\mathcal{I}_\lambda}\Ini_{(\alpha,\lambda)}(\Gamma)\), this
contradicts the assumption that \(t'=t_{\lambda}\).
Hence, \(i=1\), as claimed.

Let \(v=j(\morethan{u'}{1})\) and \(x=j(\morethan{t'}{1})\), and write
\(\zeta=\shape(v)\) and \(\xi=\shape(x)\). We shall need the restriction of
\(\Gamma\) to \(W_L\) constructed earlier.

By Remark~\ref{restrictionrem}, we can identify the vertex \(c_{\beta,u'}\) of
\(\Gamma_L\) with \(c''_{\delta,v}\) for some \(\delta\in\mathcal{I}_{L,\beta,\mu,\zeta}\), and
the vertex \(c_{\alpha,t'}\) of \(\Gamma_L\) with \(c''_{\gamma,x}\) for some
\(\gamma\in\mathcal{I}_{L,\alpha,\lambda,\xi}\). Note that since
\(\mu\neq\lambda\), and so \(\beta\neq\alpha\), we have
\(\mathcal{I}_{L,\beta,\mu,\xi}\cap\mathcal{I}_{L,\alpha,\lambda,\zeta}=\emptyset\),
and it follows that \(\delta\neq\gamma\).
Now since \(l>1\), we have \(l\in\D(v)\setminus\D(x)\), whence \(\D(v)\nsubseteq\D(x)\),
and it follows that \(\underline{\mu}(c_{\delta,v},c_{\gamma,x})=\mu(c_{\beta,u'},c_{\alpha,t'})\neq 0\).
Since \(\Gamma_L\) is ordered, and since
\(\delta\neq\gamma\), we obtain \(v<x\); in particular, \(\zeta\leqslant\xi\).

Let \((g,p)\) and \((h,q)\) be boxes vacated in \(j(\morethan{u'}{1})\) and
\(j(\morethan{t'}{1})\) respectively. Since \(\morethan {t'}k=\morethan tk\), we
have \(t'(2,2)=k+1\) and \(\col_{t'}(n)=k-1\). Moreover, Corollary~\ref{talphanystructure}
shows that \(\morethan{t'}{(k+1)}\) is the minimal tableau of its shape, equivalently
\(\col_{t'}(k+2)\leqslant\cdots\leqslant\col_{t'}(n)\). It is therefore clear that
\(\col_{t'}(i)\leqslant\col_{t'}(n)=k-1\) for all \(i\in[1,n]\), in particular,
this shows that \(q\leqslant k-1\). Since \(\morethan{u'}{k}=\morethan{u}{k}\),
we have \(u'(1,k)=n\), and since \(u'(1,k-1)=k\) while \((2,k-1)\notin [\mu]\),
it follows that \(\col_{u'}(i)\leqslant k-2\) for all \(i\in [1,n]\setminus \{k\}\cup\{n\}\).
Note, moreover, that the box \((2,1)\) is in the slide path of \(j((1,1),\morethan{u'}{1})\)),
and so we have \(g\geq 2\), and it follows that \(p\leqslant k-2<k-1\).
Hence, we obtain
\(\sum_{m=1}^{k-1} \xi_m = \sum_{m=1}^{k-1} \lambda_m - 1 = \sum_{m=1}^{k-1} \mu_m + 1 - 1 = \sum_{m=1}^{k-1} \zeta_m +1 + 1- 1\).
Thus \(\zeta\nleqslant\xi\), a desired contradiction.
\end{proof}

\begin{lemm}\label{kminimimalcompared}
Let \(\lambda\in\Lambda\), let \(\mu\in\Lambda\setminus\{\lambda\}\), and suppose that
\((\alpha,t')\in\mathcal{I}_{\lambda}\times\STD(\lambda)\) and \((\beta,u')\in\mathcal{I}_{\mu}\times\STD(\mu)\)
satisfy the condition that \(\mu(c_{\beta,s'},c_{\alpha,t'})\neq 0\), where we write
\(t'\) for \(t_{\lambda}\). Then \(\mu<\lambda\).
\end{lemm}

\begin{proof}
It is clear that \(n\) is at least \(2\). Recall that \(\mu(c_{\beta,u'},c_{\alpha,t'})\neq 0\)
implies that \(\D(t')\subsetneqq\D(u')\) and \(\mu(c_{\alpha,t'},c_{\beta,u'})=0\),
since vertices \(c_{\beta,u'}\) and \(c_{\alpha,t'}\) belong to different molecules.
Let \(k\) be the restriction number of the pair \((u',t')\) and note that \(1\leqslant k\leqslant n-1\).
By Lemma~\ref{k-minimallemma}, we have \(A(u',t')\neq\emptyset\).
Let \((u,t)\) be an approximate of \((u',t')\).
By Lemma~\ref{k-minimallemma1}, we have
\((u,t)\) is \(k\)-restricted. By Lemma~\ref{k-minimallemma}, we have
\((u,t)\approx (u',t')\), \(\D(t)\subsetneqq\D(u)\) with \(k=\min(\D(u)\setminus\D(t))\), and \(\mu(c_{\beta,u},c_{\alpha,t})=\mu(c_{\beta,u'},c_{\alpha,t'})\neq 0\).
By Proposition~\ref{cellorder}, \(\Morethan{t}{k}\) is \(k\)-critical.
For later reference, let \(\nu=\shape(\Lessthan uk) = \shape(\Lessthan tk)\).

If \(k=1\) then since \(t\) is \(1\)-minimal with respect to \(u\), it follows
by Lemma~\ref{1-minimalcomp} that \(\mu<\lambda\), and if \(k=n-1\) then since
\(\D(u)=\{n-1\}\cup\D(t)\),  it follows by Lemma~\ref{nminus1order0} that \(\mu<\lambda\) .
We may therefore assume that \(1<k<n-1\).

Let \(w=j(\morethan u1)\) and let \(y=j(\morethan t1)\),
and let \(v=\Lessthan u{(n-1)}\) and let \(x=\Lessthan t{(n-1)}\).
Let \(\zeta=\shape(w)\) and \(\xi=\shape (y)\),
and let \(\eta=\shape(v)\) and \(\theta=\shape(x)\).
We shall need the restriction of \(\Gamma\) to \(W_K\) and \(W_L\) established earlier.

By Remark~\ref{restrictionrem}, the vertex
\(c_{\beta,u}\) of \(\Gamma_K\) coincides with the vertex
\(c'_{\delta,v}\) for some \(\delta\in\mathcal{I}_{K,\eta}\)
and the vertex \(c_{\alpha,t}\) of \(\Gamma_K\) coincides with
the vertex \(c'_{\gamma,x}\) for some \(\gamma\in\mathcal{I}_{K,\theta}\).
Since \(k\in\D(u)\setminus\D(t)\) and \(k<n-1\), we have \(k\in\D(v)\setminus\D(x)\), and so,
\(\mu(c'_{\delta,v},c_{\gamma,x})=\underline{\mu}(c_{\beta,u},c_{\alpha,t})\neq 0\).

By Remark~\ref{restrictionrem}, the vertex
\(c_{\beta,u}\) of \(\Gamma_L\) coincides with the vertex
\(c''_{\pi,w}\) for some \(\pi\in\mathcal{I}_{L,\zeta}\)
and the vertex \(c_{\alpha,t}\) of \(\Gamma_L\) coincides with
the vertex \(c''_{\epsilon,y}\) for some \(\epsilon\in\mathcal{I}_{L,\xi}\).
Since \(k\in\D(u)\setminus\D(t)\) and \(k>1\), we have \(k\in\D(w)\setminus\D(y)\), and so,
\(\mu(c''_{\pi,w},c_{\epsilon,y})=\underline{\mu}(c_{\beta,u},c_{\alpha,t})\neq 0\).

Since \(\alpha\neq\beta\) (since \(\mu\neq\lambda\)), we have
\(\gamma\neq\delta\) and \(\epsilon\neq\pi\). Since
\(\Gamma_K\) and \(\Gamma_L\) are ordered, it follows that
\(v<x\) and \(w<y\). In particular, this gives \(\eta\leqslant\theta\) and \(\zeta\leqslant\xi\).

Since \(\Morethan tk\) is~\(k\)-critical, it follows from the minimality
of \(\col_{\Morethan tk}(k)\) that \(\col_t(n)\geqslant\col_t(k)\).  We shall show
that if \(\col_t(n)>\col_t(k)\) then \(\col_t(n)\geqslant\col_u(n)\).
Suppose to the contrary that \(\col_t(k)<\col_t(n)<\col_u(n)\).
We aim to show that \((u,t)\) satisfies the hypothesis of Lemma~\ref{k-minimalonerow}.

Let \(l=\col_t(n)\). Since \(\eta\leqslant\theta\), it follows that
\begin{equation}\label{etalethetaeq}
\sum_{m=1}^l\theta_m\leqslant\sum_{m=1}^l\eta_m.
\end{equation}
Moreover,
since \(\col_t(k+2)\leqslant\cdots\leqslant\col_t(n-1)\leqslant\col_t(n)\),
since \(\morethan t{(k+1)}\) is the minimal tableau of its shape, and
since  \(\col_t(k+1)\leqslant\col_t(n)\), since \(\col_t(k)+1=\col_t(k+1)\)
and \(\col_t(k)+1\leqslant\col_t(n)\) by assumption, we have
\begin{equation}\label{coltileql}
\col_t(i)\leqslant l \quad \text{if} \quad k<i<n,
\end{equation}
and so Eq.~\eqref{etalethetaeq} can be expressed in the form
\begin{equation*}
\sum_{m=1}^l\nu_m + n-1-k \leqslant\sum_{m=1}^l\nu_m + \sum_{m=1}^l (\eta_m-\nu_m).
\end{equation*}
Thus  \(\sum_{m=1}^l (\eta_m-\nu_m)\geqslant n-1-k\). But for each \(m\in[1,l]\),
\(\eta_m-\nu_m\) counts certain positive integers between \(k\) and \(n\),
and so, we have \(\sum_{m=1}^l (\eta_m-\nu_m)\leqslant n-1-k\).
It follows that \(\sum_{m=1}^l (\eta_m-\nu_m)= n-1-k\). Equivalently, we have
\begin{equation}\label{colsileql}
\col_u(i)\leqslant l \quad \text{if} \quad k<i<n.
\end{equation}
In particular,  Eq.(\ref{colsileql}) shows that \(\col_u(n-1)\leqslant l=\col_t(n)\). Since
\(\col_t(n)<\col_u(n)\) by our assumption, this implies that \(\col_u(n-1)<\col_u(n)\).
Thus \(n-1\in\A(u)\). Since \(\D(t)\subsetneqq\D(u)\), it follows that
\(n-1\in\A(t)\).
Since \(\col_u(i)\leqslant l\) whenever \(k<i<n\) by Eq.(\ref{colsileql}) and
\(\col_t(i)\leqslant l\) whenever \(k<i<n\) by Eq.(\ref{coltileql}),
and since \(\col_u(n)>\col_t(n)=l\) by our assumption, we have
\begin{equation}\label{smuonemorethansnu}
\sum_{m=1}^l\mu_m = \sum_{m=1}^l\nu_m + n-1-k = \sum_{m=1}^l\nu_m + (n-k) - 1 =
\sum_{m=1}^l\lambda_m-1.
\end{equation}

Let \((g,p)\) and \((h,q)\) be the boxes vacated by \(j((1,1),\morethan u1)\)
and \(j((1,1),\morethan t1)\), respectively. We claim that
\begin{equation}\label{qleqcoltnlp}
q\leqslant l<p.
\end{equation}
If \(l<q\) then \(\sum_{m=1}^l \xi_m = \sum_{m=1}^l \lambda_m\),
and since \(\sum_{m=1}^l \mu_m \geqslant \sum_{m=1}^l \zeta_m\), it follows by
Eq.(\ref{smuonemorethansnu}) that \(\sum_{m=1}^l \xi_m > \sum_{m=1}^l \zeta_m\).
If \(p\leqslant l=\col_t(n)\) then
\(\sum_{m=1}^l \zeta_m =\sum_{m=1}^l \mu_m -1< \sum_{m=1}^l \mu_m\).
Moreover, since
\(\sum_{m=1}^l \lambda_m - 1\leqslant\sum_{m=1}^l\xi_m\), it follows by Eq.(\ref{smuonemorethansnu})
that
\(\sum_{m=1}^l \zeta_m <\sum_{m=1}^l\xi_m\).
Since \(\zeta\leqslant\xi\), either case results in a contradiction, whence
\(q\leqslant l<p\), as claimed.

Let \(u(g,p)=b\). We claim that \(b=n\).\\
If \(k+1\leqslant b<n\) then \(p=\col_u(b)\leqslant l\) by Eq.~(\ref{colsileql}),
contradicting Eq.~(\ref{qleqcoltnlp}). Thus \(b\leqslant k\) or \(b=n\). But
\(\col_u(k)=\col_t(k)<\col_t(n)=l\) by our assumption, and so the case \(b=k\)
is excluded by Eq.~(\ref{qleqcoltnlp}). Suppose that
\(b<k\). Since the box \((g,p)\) in the diagram of \(\shape(\morethan u1)\)
is vacated by \(j((1,1),\morethan u1)\), and since
\(\Lessthan uk = \Lessthan tk\), the box \((g,p)\) in the diagram of \(\shape(\morethan t1)\)
is in the slide path of
\(j((1,1),\morethan t1)\), and it follows that \(h\geqslant g\) and \(q\geqslant p\)
by Lemma~\ref{slidepath}. The latter inequality contradicts \(q<p\)
given by Eq.(~\ref{qleqcoltnlp}). Hence \(b=n\), as claimed.

We claim that
\begin{equation}\label{colsiltpminus1}
\col_u(i)<p-1 \quad \text{if} \quad k<i<n.
\end{equation}
By Eq.~(\ref{coltileql}), we have \(\col_t(k+1)\leq\col_t(n)\).

Suppose first that \(\col_t(k+1)=\col_t(n)\). Since
\(n-1\in\A(t)\), as shown above, we have \(\col_t(n-1)<\col_t(n)\), and since
\(\col_t(k+1)=\col_t(k)+1\), the assumption \(\col_t(k+1)=\col_t(n)\) implies that
\(\col_t(n-1)\leqslant\col_t(k)\). But \(\col_t(n-1)\geqslant\col_t(k)\)
by the minimality of \(\col_{\Morethan tk}(k)\), we therefore have \(\col_t(n-1)=\col_t(k)\).
Since \(\morethan t{(k+1)}\) is the minimal tableau of its shape, this shows that
\(\col_t(k+2)\leqslant\cdots\leqslant\col_t(n-1)=\col_t(k)\), and so
it follows by the minimality of \(\col_{\Morethan tk}(k)\) that
\(\col_t(k)=\col_t(k+2)=\cdots=\col_t(n-1)\), whence
\(k+1,k+2,\ldots,n-2\in\D(t)\). On the one hand,
since \(k\in\D(u)\), and since
\(\D(t)\subsetneqq\D(u)\), we have \(k,k+1,\ldots,n-2\in\D(u)\),
and it follows that
\(\col_u(n-1)\leqslant\col_u(n-2)\leqslant\cdots\leqslant\col_u(k+1)\leqslant\col_u(k)\).
On the other hand, since \(\col_u(k)=\col_t(k)=\col_t(k+1)-1=\col_t(n)-1=l-1\), and it
follows by~Eq.(\ref{qleqcoltnlp}) that \(\col_u(k)<p-1\). Hence,
if \(k<i<n\) then \(\col_u(i)<p-1\).

Suppose now that \(\col_t(k+1)<\col_t(n)=l\). Since \(\morethan t{(k+1)}\)
is the minimal tableau of its shape, and since \(n-1\in\A(t)\), as shown above, so that \(\col_t(n-1)<\col_t(n)\), we have
\(\col_t(k+2)\leqslant\cdots\leqslant\col_t(n-1)<\col_t(n)\). Hence
if \(k<i<n\), we have \(\col_t(i)<l\). Since \(\eta\leqslant\theta\), we have
\(\sum_{m=1}^{l-1}\theta_m\leqslant\sum_{m=1}^{l-1}\eta_m\). This gives
\begin{equation*}
\sum_{m=1}^{l-1}\nu_m + n-1-k \leqslant\sum_{m=1}^{l-1}\nu_m + \sum_{m=1}^{l-1} (\eta_m-\nu_m)
\end{equation*}
that is, \(n-1-k\leqslant\sum_{m=1}^{l-1} (\eta_m-\nu_m)\). But since \(\eta_m-\nu_m\)
counts, for each \(m\in[1,l-1]\), certain positive integers between \(k\) and \(n\),
it follows that \( \sum_{m=1}^{l-1} (\eta_m-\nu_m)\leqslant n-1-k\). Therefore, we
conclude that \(\sum_{m=1}^{l-1} (\eta_m-\nu_m)=n-1-k\),
that is, \(\col_u(i)\leqslant l-1\) if \(k<i<n\). Since \(l<p\) by Eq.~(\ref{qleqcoltnlp}), we have
\(\col_u(i)<p-1\) if \(k<i<n\). This completes the proof of our claim.

Obviously \(n\) slides from the box \((g,p)\) of the diagram of
\(\shape(\morethan u1)\) into either the box \((g-1,p)\) or the box \((g,p-1)\).
Note that Eq.~(\ref{colsiltpminus1}) gives \(u(g,p-1)\leqslant k\) and
\(u(g-1,p)\leqslant k\), and so \(t(g,p-1)=u(g,p-1)\) and \(t(g-1,p)=u(g-1,p)\).
Now if \(n\) slides into the box \((g-1,p)\), so that the box \((g-1,p)\) is
in the slide path of \(j((1,1),\morethan t1)\), then Lemma~\ref{slidepath}
gives \(p\leqslant q\), contradicting Eq.~(\ref{qleqcoltnlp}).
Thus \(n\) slides into the box \((g,p-1)\), so that the box \((g,p-1)\) is in the
slide path of \(j((1,1),\morethan t1)\), and Lemma~\ref{slidepath} gives \(p-1\leqslant q\).
But since \(q\leqslant l<p\) by Eq.~(\ref{qleqcoltnlp}),
this shows that \(\col_t(n)=l=q=p-1\).

Let
\(\kappa=(\kappa_1^{m_1},\ldots,\kappa_s^{m_s})=\shape(\Lessthan u{k-1})=\shape(\Lessthan t{k-1})\).

We claim that \(\col_u(k)=\col_t(k)=1\).\\
Suppose to the contrary that \(\col_t(k)>1\).
Choose \((u,t)\) to be the minimal approximate of \((u',t')\).
By Lemma~\ref{jofminandmaxt}, we have \((\kappa_1,m_1)\) is vacated by
\(j((1,1),\morethan {\tau_{\kappa}}{1})\).
Since \((g,p-1)\) is vacated by \(j((1,1),\morethan {\tau_{\kappa}}{1})\), we have
\(\kappa_1=g\) and \(m_1=p-1\). Since \(\col_t(n)=p-1\), and since
\(t^{-1}(k)\) is a \(\kappa\)-addable box and \(t^{-1}(k)\neq (\kappa_1,1)\),
we have \(p-1=m_1<\col_t(k)\),
the latter inequality shows that \(\col_t(n)<\col_t(k)\), contradicting our assumption
that \(\col_t(k)<\col_t(n)\). Thus \(\col_u(k)=\col_t(k)=1\), as claimed.

Since \(\kappa_1=g\), as shown above, we have
\(\row_u(k)=\row_t(k)=\kappa_1+1 = g+1\). We claim that~\(g=1\). \\
Suppose to the contrary that \(g>1\).
We have
\((g,p)\notin[\kappa]\) since \(u(g,p)=n\) but \((g-1,p)\in[\kappa]\) because
\(g>1\) and because of Eq.~(\ref{colsiltpminus1}). It follows that
\((g,p)=u^{-1}(n)\) is a \(\kappa\)-addable box, whence \(s>1\).
Choose \((u,t)\) to be the maximal approximate of \((u',t')\). Let
\(\kappa^*=(\kappa_1^{*n_1},\ldots,\kappa_r^{*n_r})\). It follows from
 Lemma~\ref{jofminandmaxt} that \((n_1,\kappa_1^{*})\) is
vacated by \(j((1,1),\morethan {\tau^{\kappa}}{1})\).
Since \(n_1=m_s\) and \(\kappa_1^{*}=m_1+\cdots + m_s\), it follows that
 \((\kappa_1,m_1)\neq (n_1,\kappa_1^{*})\), a clear contradiction.
Thus \(g=1\), as claimed.

Since \(t(2,1)=u(2,1)=k\), we deduce that \(\kappa\) consists of \((k-1)\) parts of length \(1\),
that is, \(m_1=k-1\) and \(\kappa_1=1\). Thus \(\col_t(n)=k-1\) and
\(u(1,k)=n\), and since
\(\col_t(n)>\col_t(k)\) by our assumption, we have
\(\col_t(n)\geqslant 2\), and it follows that \(k\geqslant 3\). Moreover,
since \(\col_u(i)<k-1\) for \(k\leqslant i\leqslant n-1\), we have
\((2,k-1)\notin [\mu]\). It is clear that \((u,t)\)
satisfies the hypothesis of Lemma~\ref{k-minimalonerow}. But
Lemma~\ref{k-minimalonerow} shows that \((u,t)\notin A(u',t')\),
which completes our argument by contradiction.

We have shown that \(\col_t(n)=\col_t(k)\) or if
\(\col_t(n)>\col_t(k)\) then \(\col_t(n)\geqslant\col_u(n)\).

Suppose first that  \(\col_t(n)=\col_t(k)\). Since
\(\morethan t{(k+1)}\) is the minimal tableau of its shape,
we have
\(\col_t(n)\geqslant\col_t(n-1)\geqslant\cdots\geqslant\col_t(k+3)\geqslant\col_t(k+2)\).
Since \(\col_t(k)=\col_t(n)\), it follows from the minimality of
\(\col_{\Morethan tk}(k)\) that \(\col_t(k)=\col_t(k+2)=\cdots=\col_t(n)\).
Moreover, since \(\col_t(k+1)=\col_t(k)+1>\col_t(k)\), this shows that
\(\col_t(k+1)>\col_t(k+2)\). Thus, it follows that
\(k+1,k+2,\ldots,n-1\in\D(t)\).
Now since \(\D(t)\subsetneqq\D(u)\) and
\(k=\min(\D(u)\setminus\D(t))\), in particular, \(k\in\D(u)\setminus\D(t)\),
it follows that \(k,k+1,k+2\ldots,n-1\in\D(u)\).
This shows that \(\col_u(n)\leqslant\col_u(n-1)\leqslant\cdots\leqslant\col_u(k)=\col_t(k)=\col_t(n)\), whence
\(\mu<\lambda\) by Lemma~\ref{nsbeforent}.

Finally, suppose that \(\col_u(n)\leqslant l=\col_t(n)\). Since \(\eta\leqslant \theta\),
we have \(\mu\leqslant\lambda\) by Lemma~\ref{nsbeforent}, and since \(\mu\neq\lambda\), we have
\(\mu<\lambda\).
\end{proof}

\begin{lemm}\label{samelambda}
Suppose further that \(\Gamma\) is a cell.
Then \(\Lambda=\{\lambda\}\) for some \(\lambda\in P(n)\).
\end{lemm}

\begin{proof}
Assume to the contrary that \(\Lambda\) consists of more than one partitions of \(n\).
Let \(\lambda\in\Lambda\). Since \(\Gamma\) is strongly connected, the set
\(\Ini_{\lambda}(\Gamma)\neq\emptyset\). Let \((\alpha,t_{\lambda})\in\Ini_{\lambda}(\Gamma)\).
Let \(\mu\in\Lambda\setminus\{\lambda\}\) be such that \(\mu(c_{\beta,u},c_{\alpha,t_{\lambda}})\neq 0\),
for some \((\beta,u)\in\mathcal{I}_{\mu}\times\STD(\mu)\).
Then \(\mu<\lambda\) by Lemma~\ref{kminimimalcompared}.
Repeating the argument with \(\mu\) in place of \(\lambda\). Since \(\Lambda\) is a
finite set and \(\Gamma\) is strongly connected, a finite chain
\(\lambda>\mu>\cdots>\gamma>\cdots>\nu>\gamma\) is eventually reached,
a clear contradiction.
\end{proof}
Lemma~\ref{samelambda} says that the set of molecule types for an admissible \(W_n\)- cell is
a singleton set \(\{\lambda\}\), where \(\lambda\) is a partition of \(n\).

\begin{lemm}\label{muandnucells}
Suppose that \(n\geqslant 2\). Let \(D\) and \(D'\) be cells of \(\Gamma\), and let
\(\{\mu\}\) and \(\{\lambda\}\) be the sets of molecule types for \(D\) and \(D'\), respectively.
Then \(D\leqslant_{\Gamma} D'\) implies \(\mu\leqslant\lambda\). In particular, if
\(c_{\alpha,t}\in D'\) and \(c_{\beta,u}\in D\) satisfy the
condition that \(\mu(c_{\beta,u},c_{\alpha,t}) \neq 0\),
then \(\mu\leqslant\lambda\).
\end{lemm}

\begin{proof}
If \(\mu=\lambda\) then the result holds trivially. So we can assume that
\(\mu\neq\lambda\). Let \((\mathcal{C},\leqslant_{\Gamma})\)
be the poset of cells of \(\Gamma\) induced by the preorder \(\leqslant_{\Gamma}\).
It follows that \(|\mathcal{C}|\geqslant 2\).

Suppose first that \(D\) and \(D'\) are the only cells of~\(\Gamma\). Since
\(D\leqslant_{\Gamma} D'\), the set \(\Ini_\lambda(\Gamma)\neq\emptyset\). Let
\((\alpha,t_{\lambda})\in\mathcal{I}_{\lambda}\times\STD(\lambda)\) and
\((\beta,u)\in\mathcal{I}_{\mu}\times\STD(\mu)\) satisfy
\(\mu(c_{\beta,u},c_{\alpha,t_{\lambda}}) \neq 0\).
It follows readily from Lemma~\ref{kminimimalcompared} that \(\mu<\lambda\).

Suppose now that \(|\mathcal{C}|>2\) and the result holds for
any admissible \(W_n\)-graph of less than \(|\mathcal{C}|\) cells.
Let \(C_0\) and \(C_1\) be a minimal and a maximal cell in \((\mathcal{C},\leqslant_{\Gamma})\).
It is clear that \(C_0\) and \(C\setminus C_1\) are closed subsets of \(C\), hence
the full subgraphs \(\Gamma(C\setminus C_0)\) and \(\Gamma(C\setminus C_1)\)
induced by \(C\setminus C_0\) and \(C\setminus C_1\)
are themselves admissible \(W_n\)-graph with edge weights and vertex colours
inherited from \(\Gamma\). It follows that if both \(D\) and \(D'\) are cells of
\(\Gamma(C\setminus C_0)\) or  \(\Gamma(C\setminus C_1)\), then the result is
given by the inductive hypothesis. Furthermore, since \(D\leqslant_{\Gamma} D'\) by
assumption, we can assume that
\(D=C_0\) and \(D'=C_1\) are the (unique) minimal and maximal cells in
\((\mathcal{C},\leqslant_{\Gamma})\).

Let \(C'\neq C_0,\,C_1\) be a cell of
\(\Gamma\). By Lemma~\ref{samelambda}, the set of molecule types for \(C'\)
is \(\{\nu\}\) for some \(\nu\in\Lambda\). Now since \(C_0\leqslant_{\Gamma} C'\) and
\(C_0\) and \(C'\) are cells of \(\Gamma(C\setminus C_1)\), we have
\(\mu\leqslant\nu\) by the inductive hypothesis. Similarly, since \(C'\leqslant_{\Gamma} C_1\) and
\(C'\) and \(C_1\) are cells of \(\Gamma(C\setminus C_0)\), we have
\(\nu\leqslant\lambda\) by the inductive hypothesis. It follows that \(\mu\leq\lambda\) as required.

Since \(\mu(c_{\beta,u},c_{\alpha,t}) \neq 0\), we have \(\D(u)\nsubseteq\D(t)\).
It follows that \(c_{\beta,u}\leqslant_{\Gamma}c_{\alpha,t}\), hence
\(D\leqslant_{\Gamma} D'\) by the definition of the preorder \(\leqslant_{\Gamma}\). It
follows from the result above that \(\mu\leqslant\lambda\).
\end{proof}

\begin{theo}\label{inductivestep}
\(\Gamma\) is ordered.
\end{theo}

\begin{proof}
Suppose that \((\alpha,t)\in\mathcal{I}_{\lambda}\times\STD(\lambda)\) and
\((\beta,u)\in\mathcal{I}_{\mu}\) satisfy
\(\mu(c_{\beta,u},c_{\alpha,t}) \neq 0\). It follows from
Lemma~\ref{muandnucells} that \(\mu\leqslant\lambda\). Now
Proposition~\ref{kldominancesamenolecule} says that
\(u < t\) unless \(\alpha=\beta\) and \(u=s_i t>t\) for
some \(i\in[1,n-1]\). That is, \(\Gamma\) is ordered.
\end{proof}

\begin{rema}\label{orderedtwosided}
Let \(y,w \in W_n\), and let \(RS(y)=(u,v)\) and \(RS(w)=(t,x)\). It follows
from Remark~\ref{orderedregular} that if\(y\preceq\leftright w\) then
\(\mu\leqslant\lambda\), where \(\mu=\shape(x)=\shape(u)\) and
\(\lambda=\shape(y)=\shape(v)\). This gives an alternative approach to
the necessary part of the following well-known result. (See, for example,
\cite[Theorem 5.1]{geck:klMurphy}.)
\begin{theo}\label{Wn-two-sidedcell}
Let \(y,w\in W_n\) and \(\mu,\lambda\in P(n)\), and suppose that \(RS(y)\in\STD(\mu)\times\STD(\mu)\) and
\(RS(w)\in\STD(\lambda)\times\STD(\lambda)\). Then
\(y\preceq\leftright w\) if and only if \(\mu\leqslant\lambda\). In particular,
the sets \(D(\lambda):=\{w\in W_n \mid RS(w)\in\STD(\lambda)\times\STD(\lambda)\}\), where
\(\lambda\in P(n)\), are precisely the Kazhdan--Lusztig two-sided cells.
\end{theo}
Let \(\lambda\in P(n)\). For each \(t\in\STD(\lambda)\), since
\(C(t)=\{w\in W_n\mid Q(w)=t\}\) gives rise to the left cell isomorphic to
\(\Gamma_{\lambda}\), we have \(D(\lambda)=\bigsqcup_{t\in\STD(\lambda)}C(t)\) gives
rise to the union of \(\lvert \STD(\lambda) \rvert\) left cells whose molecule types
are \(\lambda\).
\end{rema}

\section{\textit{W-}graphs for admissible cells in type \textit{A}}
\label{sec:9}

\begin{defi}\label{probablepair}
Let \(\lambda\in P(n)\). A pair of standard \(\lambda\)-tableaux \((u,t)\) is said to
be a \textit{probable pair} if \(u<t\) and \(\D(t)\subsetneqq\D(u)\).
\end{defi}
It can be seen that there is no probable pair unless \(n\geqslant 5\).

\begin{lemm}\label{rnltmaxsd}
Let \(\lambda\in P(n)\), and let \(u,t\in\STD(\lambda)\).
Let \(i\) be the restriction number of~\((u,t)\) and \(j=\max(\SD(t))\).
If \((u,t)\) is favourable and satisfies
\(\D(t)\subsetneqq\D(u)\) then \(i<j\).
\end{lemm}

\begin{proof}
Suppose to the contrary that \(i\geqslant j\). Since \((u,t)\) is favourable, and
since \(\D(t)\subsetneqq\D(u)\), we have \(i\in\D(u)\oplus\D(t)=\D(u)\setminus\D(t)\).
Since \(i\notin\D(t)\), we have \(i\neq j\), and so \(i>j\).
Let \(w=\Lessthan ti=\Lessthan ui\in\STD(\mu)\), where \(\mu=\shape(w)\).
Since \(j=\max(\SD(t))\) and since \(i>j\),
we have \(\D(\morethan ti)\cap [i+1,n-1]=\WD(\morethan ti)\cap [i+1,n-1]\),
and it follows by Remark~\ref{minimal-tableau} that \(\morethan ti\) is minimal,
that is, \(\morethan ti=\tau_{\lambda/\mu}\).
Moreover, since \(i\notin\D(t)\), this shows that for all \(k>i\), we have
\(\col_t(k)\geqslant\col_t(i+1)>\col_t(i)\), from which we have
\(\lambda_m=\mu_m\) for all \(m\leqslant \col_t(i)\). Hence if
\(k>i\) then \(\col_u(k)>\col_u(i)\), in particular,
\(\col_u(i+1)>\col_u(i)\), contradicting \(i\in\D(u)\).
\end{proof}

\begin{lemm}\label{existencef}
Let \(\lambda\in P(n)\), and let \(u,t\in\STD(\lambda)\). Let \(i\) be the
restriction number of~\((u,t)\).
Suppose that \((u,t)\) is favourable and satisfies \(\D(t)\subsetneqq\D(u)\).
If, moreover, \(i+1=\max\SD(t)\), then \(\col_{t}(i+2)\neq\col_t(i)\).
\end{lemm}

\begin{proof}
Suppose to the contrary that \(\col_{t}(i+2)=\col_t(i)\). Since
\((u,t)\) is \(i\)-restricted, we have \(\Lessthan ui=\Lessthan ti\).
Let \(\mu=\shape(\Lessthan ti)=\shape(\Lessthan ui)\), and let
\(u(g,p)=t(g,p)=i\). Now since \(i+1=\max\SD(t)\), we have
\(\morethan t{(i+1)}\) is minimal, hence
\(\col_t(i)=\col_t(i+2)\leqslant\col_t(k)\) for \(k>i+2\).
Furthermore, since \((u,t)\) is favourable, we have \(i\in\D(u)\setminus\D(t)\),
and so we have \(\col_t(i)<\col_t(i+1)\). It follows that for \(k>i\),
we have \(\col_t(k)\geqslant\col_t(i)\). Therefore, for
each \(j\in[1,p-1]\), we have \(\lambda_j=\mu_j\). This shows that for
\(k>i\), we have \(\col_u(k)\geqslant\col_u(i)\), in particular, we have
\(\col_u(i+1)\geqslant\col_u(i)\). Since \(i\in\D(u)\setminus\D(t)\),
as \((u,t)\) is favourable, we have \(\col_u(i+1)\leqslant\col_u(i)\).
Thus \(\col_u(i+1)=\col_u(i)\), and
we have \(u(g+1,p)=i+1\). An easy induction on \(l\in[1,\lambda_p-g]\) shows that
if \(t(g+l,p)=i+l+1\) then \(u(g+1,p)=i+l\). However, this contradicts
\(\D(t)\subsetneqq\D(u)\), as desired.
\end{proof}
Let \(\Gamma = \Gamma(C,\mu,\tau)\) be an admissible \(W_n\)-graph. Suppose that
\(\Lambda=\{\lambda\}\), where \(\lambda\in P(n)\), is the set of molecule types
for \(\Gamma\), and let \(\mathcal{I}=\mathcal{I}_{\lambda}\) index the
molecules of~\(\Gamma\). By Remark~\ref{vertexsetofWngraph},
the vertex set of \(\Gamma\) is given by \(C=\bigsqcup_{\alpha\in\mathcal{I}}C_{\alpha,\lambda}\),
where for each \(\alpha\in\mathcal{I}\),
\(C_{\alpha,\lambda}=\{c_{\alpha,t}\mid t\in\STD(\lambda)\}\), the simple
edges of \(\Gamma\) are the pairs \(\{c_{\beta,u},c_{\alpha,t}\}\) such that
\(\alpha=\beta\) and \(u\) and \(t\) are related by
a dual Knuth move, and \(\tau(c_{\alpha,t})=\D(t)\).

\begin{lemm}\label{linktominimalCor}
Let \(u,t\in\STD(\lambda)\), and suppose that the pair \((u,t)\) is probable.
Then for all \((v,x)\in F(u,t)\), the pair \((v,x)\) is probable,
\(\max(\SD(v))=\max(\SD(t))\), and
\(\mu(c_{\beta,v},c_{\alpha,x})=\mu(c_{\beta,u},c_{\alpha,t})\).
\end{lemm}

\begin{proof}
We may assume that \((u,t)\) is not favourable.
Let \((v,x)\in F(u,t)\).
Let \(i\) be the restriction number of~\((v,x)\) (which is also
the restriction number of~\((u,t)\)), and let \(j=\max(\SD(x))\).
Since \((v,x)\in F(u,t)\), we have \((v,x)\approx_i (u,t)\), and so
\((v,x)\approx (u,t)\) by Remark~\ref{mimpliesn}, and since
\((u,t)\) is probable, we have \(u<t\), and
it follows by Lemma~\ref{bruhatorderpreserved} that \(v<x\). Furthermore,
\(u<t\) implies that  \(\Lessthan u{(i+1)}\leqslant\Lessthan t{(i+1)}\),
but since \((u,t)\) is \(i\)-restricted. we have  \(\Lessthan u{(i+1)}\neq\Lessthan t{(i+1)}\),
therefore \(\Lessthan u{(i+1)} < \Lessthan t{(i+1)}\). It follows by
Remark~\ref{resdomrem} that \(\col_u(i+1)<\col_t(i+1)\).
Moreover, as \((u,t)\) is not favourable and \(\D(t)\subsetneqq\D(u)\),
so that \(\D(u)\oplus\D(t)=\D(u)\setminus\D(t)\),
we have \(i<\min(\D(u)\setminus\D(t))\) by Remark~\ref{k-reducedrem}.
Thus \((u,t)\) satisfies the hypothesis of Lemma~\ref{noname2}.
Since \(\col_u(i+1)<\col_t(i+1)\) as shown above, it
follows by Lemma~\ref{noname2} that \((v,x)\) satisfies
\(\D(v)\setminus\D(x)=\{i\}\cup(\D(u)\setminus\D(t))\) and \(\D(x)\setminus\D(v)=\emptyset\).
Since \(\D(x)\setminus\D(v)=\emptyset\) while \(\D(v)\setminus\D(x)\neq\emptyset\),
we have  \(\D(x)\subsetneqq\D(v)\). Hence \((v,x)\) is probable.

Next, since \(j=\max(\SD(x))\) and \(j>i\) by Lemma~\ref{rnltmaxsd},
we have \(j=\max(SD(\morethan xi))\), and since \(\morethan ti = \morethan xi\),
we have \(j=\max(\SD(\morethan ti))\), and it follows that
\(j=\max(\SD(t))\), as required.

Finally, since \(i<\min(\D(u)\setminus\D(t))\) as shown above,
we have \(\mu(c_{\beta,v},c_{\alpha,x})=\mu(c_{\beta,u},c_{\alpha,t})\)
by Lemma~\ref{noname0}.
\end{proof}

\begin{prop}\label{monomolecularadmcellsareKL}
Monomolecular admissible cells of type \(A_{n-1}\) are Kazhdan--Lusztig.
\end{prop}

\begin{proof}
Suppose that \(\Gamma=\Gamma(C,\mu,\tau)\) is a monomolecular admissible \(W_n\,\)-cell.
Then there is a partition \(\lambda\) of \(n\) such that
\(C=\{c_t\mid t\in\STD(\lambda)\}\), and
\(\{c_u,c_t\}\) is a simple edge of \(\Gamma\)
if and only if \(u,\,t\in\STD(\lambda)\) are related by
a dual Knuth move.
In view of Corollary~\ref{leftcelllambda}, our task is to show that
\(\Gamma\cong \Gamma_\lambda=\Gamma(\STD(\lambda),\mu^{(\lambda)},\tau^{(\lambda)})\).
Recall from Remark~\ref{Gamma-lambdaIsAdmissible}
that \(\Gamma_\lambda\) is  an admissible \(W_n\)-graph consisting
of a single molecule of type~\(\lambda\).
Since it follows from Remark~\ref{equalsubsetsS} that
\(\tau(c_t) = \D(t) = \tau^{(\lambda)}(t)\) for all \(t\in\STD(\lambda)\), it remains
to show that \(\mu(c_u,c_t)=\mu^{(\lambda)}(u,t)\) for all \(u,\,t\in\STD(\lambda)\).
Note that, by Theorem~\ref{combinatorialCharacterisation}, both \(\Gamma\) and
\(\Gamma_\lambda\) satisfy the \(W_n\)-Compatibility Rule, the \(W_n\)-Simplicity Rule,
the \(W_n\)-Bonding Rule and the \(W_n\)-Polygon Rule.

We have shown in Theorem~\ref{inductivestep} that \(\Gamma\) and
\(\Gamma_\lambda\) are both ordered. Thus if \(u,t\in\STD(\lambda)\)
then \(\mu(c_u,c_t)=\mu^{(\lambda)}(u,t)=0\)
unless \(u<t\) or \(u=s_it>t\) for some \(i\in[1,n-1]\).
If \(u=s_it>t\) for some \(i\in[1,n-1]\) then we have
\(\mu(c_u,c_t)=\mu^{(\lambda)}(u,t)=1\) by Corollary~\ref{arweightone}.
Now suppose that \(u<t\) and \(\D(t)\nsubseteq\D(u)\). If one
or other of \(\mu(c_u,c_t)\) and \(\mu^{(\lambda)}(u,t)\) is nonzero then,
by the Simplicity Rule, one or other of
\(\{c_u,c_t\}\) and \(\{u,t\}\) is a simple edge, whence \(u\) and \(t\) are
related by a dual Knuth move (by Remark~\ref{vertexsetofWngraph}), and
both \(\{c_u,c_t\}\) and \(\{u,t\}\) are simple edges. So
\(\mu(c_u,c_t)=\mu^{(\lambda)}(u,t)=1\) in this case. Obviously there
is nothing to show if \(\mu(c_u,c_t)\) are \(\mu^{(\lambda)}(u,t)\) both zero,
and so all that remains is to show that \(\mu(c_u,c_t)=\mu^{(\lambda)}(u,t)\)
whenever \(u<t\) and \(\D(t)\subsetneqq\D(u)\). That is, it remains to show
that \(\mu(c_u,c_t)=\mu^{(\lambda)}(u,t)\) for all probable pairs of
standard \(\lambda\)-tableaux.

Let \((u',t')\) be a probable pair. If \(t'=\tau_{\lambda}\) then there is
nothing to prove. Proceeding inductively on the lexicographic order,
let \(\tau_{\lambda}\neq t'\in\STD(\lambda)\), and assume that the result holds
for all \(x\in\STD(\lambda)\) such that \(x<_{\lex} t'\).

Let \(i\) be the restriction number of \((u',t')\), and
let \(j=\max(\SD(t'))\). Let \((u,t)\in F(u',t')\), and note that \((u,t)\) is
\(i\)-restricted and favourable, and satisfies \(\morethan ti=\morethan {t'}i\)
and \(\morethan ui=\morethan {u'}i\). Moreover, Lemma~\ref{linktominimalCor} shows that
\((u,t)\) is probable, \(\max(\SD(t))=\max(\SD(t'))=j\),
and \(\mu(c_{u},c_{t})=\mu(c_{u'},c_{t'})\) and
\(\mu^{(\lambda)}(u,t)=\mu^{(\lambda)}(u',t')\).
By the last result, it suffices to show that \(\mu(c_{u},c_{t})=\mu^{(\lambda)}(u,t)\).

Since \(j\in\SD(t)\), we have \(s_{j}t\in\STD(\lambda)\) and \(s_{j}t<t\).
Let \(v=s_jt\), and note that \(v<_{\lex}t'\) by Lemma~\ref{xlexlessthant}.
Since \(i<j\) by Lemma~\ref{rnltmaxsd}, we have either \(j-i>1\) or \(j-i=1\).

\begin{Case}{1.} Suppose that \(j-i>1\), so that \(m(i,j)=2\).
Since \((u,t)\) is favourable, and
since \(\D(t)\subsetneqq\D(u)\) (since \((u,t)\) is probable),
we have \(i\in\D(u)\setminus\D(t)\) and \(j\in D(u)\cap\D(t)\), that is,
\(i\notin\D(t)\) and \(j\in\D(t)\), and \(i,j\in\D(u)\).
We also have \(i,j\notin\D(v)\), by Lemma~\ref{variousaltpaths} (i).

If \((c_v,c_{y_1},c_{u})\) is any alternating directed path of type \((j,i)\), then, since
\(\Gamma\) is ordered, it follows that either \(y_1=s_jv=t>v\) or \(y_1<v\). Similarly, if
\((c_v,c_{x_1},c_{u})\) is any alternating directed path of type \((i,j)\), then
it follows that either \(x_1=s_iv>v\) or \(x_1<v\). Note that
if \(x_1=s_iv>v\), then since \(x_1\in\STD(\lambda)\), it follows that
\(i\in\SA(v)\). Thus, if \(x_1=s_iv>v\), then
\(i\in\D(s_iv)\) and \(j\notin\D(s_iv)\) by Lemma~\ref{variousaltpaths} (i).
Now since \(\Gamma\) satisfies the \(W_n\)-Polygon Rule, we have
\(N^{2}_{j,i}(\Gamma;c_{v},c_{u})=N^{2}_{i,j}(\Gamma;c_{v},c_{u})\), and it follows that
\begin{multline}\label{sumijorder2}
\mu(c_{t},c_v)\mu(c_{u},c_{t})\,\, +
\!\!\sum_{y_1<v}\!\!\mu(c_{y_{1}},c_v)\mu(c_{u},c_{y_{1}})\\[-5pt]
=\mu(c_{s_iv},c_v)\mu(c_{u},c_{s_iv})\,\, +
\!\!\sum_{x_1<v}\!\!\mu(c_{x_1},c_v)\mu(c_{u},c_{x_1}),
\end{multline}
where the term \(\mu(c_{s_iv},c_v)\mu(c_{u},c_{s_iv})\) on the right hand side of
Eq.~\eqref{sumijorder2} should be omitted if \(i\notin\SA(v)\). Note that if \(i\in\SA(v)\)
then \((c_v,c_{s_iv},c_{u})\) is not necessarily a directed path, since there need not
be an arc from \(s_iv\) to~\(u\), but in this case
\(\mu(c_{s_iv},c_v)\mu(c_{u},c_{s_iv})=0\) since \(\mu(c_{u},c_{s_iv})=0\). Similarly,
\((c_v,c_{t},c_{u})\) is not necessarily a directed path, since there need not be an arc
from \(t\) to~\(u\), but \(\mu(c_{t},c_v)\mu(c_{u},c_{t})=0\) in this case.
So Eq.~\eqref{sumijorder2} still holds in these cases.

Since Corollary~\ref{arweightone} gives \(\mu(c_{t},c_v)= 1\), and \(\mu(c_{s_iv},c_v)= 1\)
if \(i\in\SA(v)\),
Eq.~\eqref{sumijorder2} yields the following formula for \(\mu(c_{u},c_{t})\):
\begin{equation*}
\mu(c_{u},c_{t}) =\mu(c_{u},c_{s_iv})\,\, +
\!\!\sum_{x_1<v}\!\!\mu(c_{x_1},c_v)\mu(c_{u},c_{x_1})\,\,-
\!\!\sum_{y_1<v}\!\!\mu(c_{y_1},c_v)\mu(c_{u},c_{y_1}),
\end{equation*}
where \(\mu(c_{u},c_{s_iv})\) should be interpreted as 0 if \(s_iv\notin\STD(\lambda)\).

Working similarly on \(\Gamma_{\lambda}\) yields the following formula for
\(\mu^{(\lambda)}(u,t)\):
\begin{equation*}
\mu^{(\lambda)}(u,t) =\mu^{(\lambda)}(u,s_iv)\,\,
+\!\!\sum_{x_1<v}\mu^{(\lambda)}(x_1,v)\mu^{(\lambda)}(u,x_1)
\,\,-\!\!\sum_{y_1<v}\mu^{(\lambda)}(y_1,v)\mu^{(\lambda)}(u,y_1).
\end{equation*}

Since \(v<_{\lex}t\) by Lemma~\ref{xlexlessthant}, \(s_iv<_{\lex}t\) (if \(i\in\SA(t)\))
by Lemma~\ref{xlexlessthant} (i),  and \(x_1<_{\lex}t\)
and \(y_1 <_{\lex} t\) by Lemma~\ref{xlexlessthant} (ii), it follows by the
inductive hypothesis that the corresponding edge weights that appear in the two formulae
above are the same. Thus \(\mu(c_{u},c_{t})=\mu^{(\lambda)}(u,t)\), as desired.
\end{Case}

\begin{Case}{2.} Suppose that \(i=j-1\), so that \(m(i,j)=3\).
By Lemma~\ref{existencef}, \(\col_{t}(j-1)\neq\col_{t}(j+1)\), and it follows that either
one of the following situations occurs:
\(\col_{t}(j-1)<\col_{t}(j+1)\) or \(\col_{t}(j-1)>\col_{t}(j+1)\).

If \(\col_{t}(j-1)<\col_{t}(j+1)\), then the result follows by the same argument as above,
with \(j-1\) replacing \(i\) and Lemma~\ref{variousaltpaths}(ii) replacing
Lemma~\ref{variousaltpaths}(i).

Suppose that \(\col_{t}(j-1)>\col_{t}(j+1)\). Since
\(j-1\in\SD(v)\) by Lemma~\ref{variousaltpaths} (iii), we have
\(s_{j-1}v\in\STD(\lambda)\) and \(s_{j-1}v<v\). Let
\(w=s_{j-1}v\). It follows by Lemma~\ref{variousaltpaths} (iii) that
\(j-1,\,j\notin\D(w)\), but
\(j-1,\,j\in\D(u)\), since \((u,t)\) is favourable and probable.

We consider length three alternating directed paths
of type \((j-1,j)\) and \((j,j-1)\) from \(c_w\) to~\(c_{u}\).
We have \(j\in\D(t)\) and \(j-1\notin\D(t)\) (since \((u,t)\) is favourable),
while \(j-1\in\D(v)\) and \(j\notin\D(v)\) by Lemma~\ref{variousaltpaths}~(iii).

If \((c_w,c_{x_1},c_{x_2},c_{u})\) is any alternating directed path of
type~\((j-1,j)\), then, since \(\Gamma\) is ordered, it follows that
either \(x_1=s_{j-1}w=v>w\), or else \(x_1<w\). Moreover, since \(\Gamma\) satisfies the
\(W_n\)-Simply Laced Bonding Rule, the fact that
\(j-1\in\D(x_1)\) and \(j\notin\D(x_1)\)
shows that \(c_{x_{2}}\) is the unique vertex adjacent to \(c_{x_{1}}\) satisfying \(j-1\notin\D(x_2)\)
and \(j\in\D(x_2)\). That is, \(x_2\) is the \((j-1)\)-neighbour of \(x_1\). Thus it follows that
either \(x_1=v\) and \(x_2=s_{j}v=t\), or else \(x_1<w\) and
either \(x_2=s_jx_1>x_1\) or \(x_2=s_{j-1}x_1<x_1\).

Similarly, if \((c_{w},c_{y_1},c_{y_2},c_{u})\) is any alternating directed path
of type~\((j,j-1)\), then it follows that either
\(y_1=s_jw>w\) or \(y_1<w\), and \(y_2\) is the \((j-1)\)-neighbour of~\(y_1\).
Note that if \(y_1=s_jw>w\), then since \(y_1\in\STD(\lambda)\),
it follows that \(j\in\SA(w)\). Thus, if \(y_1=s_jw>w\) then
\(y_2=s_{j-1}y_1=s_{j-1}s_jw>s_jw=y_1\), and
\(j\in\D(s_{jw})\) and \(j-1\notin\D(s_jw)\), and
\(j-1\in\D(s_{j-1}s_jw)\) and \(j\notin\D(s_{j-1}s_jw)\)
by Lemma~\ref{variousaltpaths} (iii),
while if \(y_1<w\) then either \(y_2=s_{j-1}y_1>y_1\) or \(y_2=s_jy_1<y_1\).

Now since \(\Gamma\) satisfies the \(W_n\)-Polygon Rule, we have
\(N^{3}_{j-1,j}(\Gamma;c_{w},c_{u})=N^{3}_{j,j-1}(\Gamma;c_{w},c_{u})\), and it follows that
\begin{multline}\label{sumijorder3}
\mu(c_v,c_w)\mu(c_{t},c_v)\mu(c_{u},c_{t}) +\mskip -25mu\sum_{\substack{x_1<w\\x_2=(j-1)\neb(x_1)}}\mskip -25mu
\mu(c_{x_{1}},c_w)\mu(c_{x_{2}},c_{x_{1}})\mu(c_{u},c_{x_{2}})\\
=\mu(c_{s_jw},c_w)\mu(c_{s_{j-1}s_jw},c_{s_jw})\mu(c_{u},c_{s_{j-1}s_jw}) +\\
\sum_{\substack{y_1<w\\y_2=(j-1)\neb(y_1)}}\mskip -25mu
\mu(c_{y_{1}},c_{w})\mu(c_{y_{2}},c_{y_{1}})\mu(c_{u},c_{y_{2}}),
\end{multline}
where the term \(\mu(c_{s_jw},c_w)\mu(c_{s_{j-1}s_jw},c_{s_jw})\mu(c_{u},c_{s_{j-1}s_jw})\)
on the right hand side of Eq.~(\ref{sumijorder3} should be omitted if \(j\notin\SA(w)\).
Note that if \(j\in\SA(w)\) then \((c_w,c_{s_jw},c_{s_{j-1}s_jw},c_u)\) is not necessarily a
directed path, since there need not to be an arc from \(c_{s_{j-1}s_jw}\) to \(c_u\), but
in this case \(\mu(c_{s_jw},c_w)\mu(c_{s_{j-1}s_jw},c_{s_jw})\mu(c_{u},c_{s_{j-1}s_jw})=0\)
since \(\mu(c_u,c_{s_{j-1}s_jw})=0\). Similarly, \((c_w,c_v,c_t,c_u)\) is not necessarily a
directed path, since there need not be an arc from \(c_t\) to \(c_u\), but
\(\mu(c_v,c_w)\mu(c_t,c_v)\mu(c_u,c_t)=0\) in this case.
So Eq.~(\ref{sumijorder3}) still holds in these cases.
%if \((c_w,c_{v},c_{t},c_{u})\) is not the directed path, in which case we shall assign the value of \(0\) to
%\(\mu(c_v,c_w)\mu(c_{t},c_v)\mu(c_{u},c_{t})\).

Since \(\mu(c_v,c_w) = \mu(c_{s_jw},c_w)= 1\) and
\(\mu(c_{t},c_v)= \mu(c_{s_{j-1}s_jw},c_{s_jw}) = 1\), by Corollary~\ref{arweightone},
and since \(\mu(c_{x_{2}},c_{x_{1}})=\mu(c_{y_{2}},c_{y_{1}})=1\), since
\(\{c_{x_{1}},c_{x_{2}}\}\) and \(\{c_{y_{1}},c_{y_{2}}\}\) are simple edges,
Eq.~(\ref{sumijorder3}) yields the following formula for \(\mu(c_{u},c_{t})\):
\begin{multline*}
\mu(c_{u},c_{t})=\mu(c_{u},c_{s_{j-1}s_{j}w})+\mskip -25mu
\sum_{\substack{y_1<w\\y_2=(j-1)\neb(y_1)}}\mskip -25mu \mu(c_{y_{1}},c_{w})\mu(c_{u},c_{y_{2}})\\[-5pt]
-\mskip -25mu\sum_{\substack{x_1<w\\x_2=(j-1)\neb(x_1)}}\mskip -25mu\mu(c_{x_{1}},c_w)\mu(c_{u},c_{x_{2}}),
\end{multline*}
where \(\mu(c_{u},c_{s_{j-1}s_{j}w})\) should be interpreted as \(0\) if \(s_jw\notin\STD(\lambda)\).

Working similarly on \(\Gamma_{\lambda}\) yields the following formula for \(\mu^{(\lambda)}(u,t)\):
\begin{multline*}
\mu^{(\lambda)}(u,t)=\mu^{(\lambda)}(u,s_{j-1}s_{j}w)+\mskip -25mu
\sum_{\substack{y_1<w\\y_2=(j-1)\neb(y_1)}}\mskip -25mu \mu^{(\lambda)}(y_{1},w)\mu^{(\lambda)}(u,y_{2})\\[-5pt]
-\mskip -25mu\sum_{\substack{x_1<w\\x_2=(j-1)\neb(x_1)}}\mskip -25mu\mu^{(\lambda)}(x_{1},w)\mu^{(\lambda)}(u,x_{2}).
\end{multline*}
Since \(w,\,s_{j-1}s_jw<_{\lex}t'\) (if \(j\in\SA(w)\)) by Lemma~\ref{xlexlessthant} (iii),
and \(x_2,y_2<_{\lex}t'\) by Lemma~\ref{xlexlessthant} (iv),
it follows by the inductive hypothesis that the corresponding edge weights that appear in the two formulae
above are the same. Thus \(\mu(c_{u'},c_{t'})=\mu^{(\lambda)}(u',t')\), as desired.
\qedhere
\end{Case}
\end{proof}

\begin{prop}\label{admcellaremonomol}
Let \(\Gamma=\Gamma(C,\mu,\tau)\) be an admissible \(W_n\)-graph. Suppose that
\(\Lambda=\{\lambda\}\), where \(\lambda\in P(n)\), is
the set of molecule types for \(\Gamma\), and let \(\mathcal{I}=\mathcal{I}_{\lambda}\) index
the molecules of \(\Gamma\). For each \(\alpha\in\mathcal{I}\), let
\(C_{\alpha}=C_{\alpha,\lambda} = \{c_{\alpha,t}\mid t\in\STD(\lambda)\}\) be the vertex set of
a molecule of \(\Gamma\).
Then \(\Gamma = \bigsqcup_{\alpha\in\mathcal{I}}\Gamma(C_{\alpha})\),
and \(\Gamma(C_{\alpha})\) is isomorphic to \(\Gamma_{\lambda}\) for
each \(\alpha\in\mathcal{I}\).
\end{prop}

\begin{proof}
If \(\lvert \mathcal{I}\rvert=1\) then \(\Gamma=\Gamma(C_{\alpha})\). Since \(\Gamma(C_{\alpha})\) is a
monomolecular admissible \(W_n\)-cell of type \(\lambda\), Proposition~\ref{monomolecularadmcellsareKL}
says that \(\Gamma(C_{\alpha})\) is isomorphic to \(\Gamma_{\lambda}\). So we assume that
\(\lvert \mathcal{I}\rvert >1\).

If \(\Gamma = \bigsqcup_{\alpha\in\mathcal{I}}\Gamma(C_{\alpha})\) then
for each \(\alpha\in\mathcal{I}\) the set \(C_{\alpha}=\{c_{\alpha,t}\mid t\in\STD(\lambda)\}\)
is the vertex set of a monomolecular admissible \(W_n\)-cell of type~\(\lambda\). The result
then follows immediately from Proposition~\ref{monomolecularadmcellsareKL}, which says that each
\(\Gamma(C_{\alpha})\) is isomorphic to \(\Gamma_{\lambda}\).
Thus it suffices to show that \(\Gamma = \bigsqcup_{\alpha\in\mathcal{I}}\Gamma(C_{\alpha})\).

Suppose otherwise. Then there exists \(\alpha\in\mathcal{I}\) such that
\begin{equation*}
\Ini_{\alpha}(\Gamma)=\{\,t\in\STD(\lambda)\mid\mu(c_{\beta,u},c_{\alpha,t})\neq 0
\text{ for some }(\beta,u)\in(\mathcal{I}\setminus\{\alpha\})\times\STD(\lambda)\}\ne\emptyset,
\end{equation*}
and we let \(t'\) be the element of \(\Ini_{\alpha}(\Gamma)\)
that is minimal in the  lexicographic order on \(\STD(\lambda)\). Choose
\((\beta,u')\in (\mathcal{I}\setminus \{\alpha\})\times\STD(\lambda)\) with
\(\mu(c_{\beta,u'},c_{\alpha,t'})\neq 0\). Since
\(\Gamma\) satisfies the \(W_n\)-Simplicity Rule (by
Theorem~\ref{combinatorialCharacterisation}), the assumption that
\(\alpha\neq\beta\) and  \(\mu(c_{\beta,u'},c_{\alpha,t'})\neq 0\) implies that
\(\D(t')\subsetneqq\D(u')\). Moreover, since \(\Gamma\) is ordered (by Theorem~\ref{inductivestep}),
\(\alpha\neq\beta\) implies  that \(u'<t'\). Hence \((u',t')\) is a probable pair.

Let \(i\) be the restriction number of \((u',t')\) and
\(j=\max(\SD(t'))\). Let \((u,t)\in F(u',t')\). It is clear that \((u,t)\) is
\(i\)-restricted and favourable. Thus Lemma~\ref{linktominimalCor} shows that
\((u,t)\) is probable, \(\max(\SD(t))=\max(\SD(t'))=j\),
and \(\mu(c_{\beta,u},c_{\alpha,t})=\mu(c_{\beta,u'},c_{\alpha,t'})\neq 0\).
Furthermore, since \(i\in\D(u)\setminus\D(t)\) (since \((u,t')\) is favourable), and
since \(j\in\D(t)\) and
\(\D(t)\subsetneqq\D(u)\) (since \((u,t)\) is probable), it follows
that \(j\in\D(t)\) and \(i\notin\D(t)\), and
\(i,j\in\D(u)\). Let \(v=s_jt\),
and note that \(v\in\STD(\lambda)\) and \(v<t\).
Since  \(i<j\) by Lemma~\ref{rnltmaxsd}, either \(i<j-1\) or \(i=j-1\).

Suppose first that \(i<j-1\). It follows by
Lemma~\ref{variousaltpaths} (i) that \(i,j\notin\D(v)\).
Moreover, since \(\mu(c_{\alpha,t},c_{\alpha,v})=1\)
by Corollary~\ref{arweightone}, and since \(\mu(c_{\beta,u},c_{\alpha,t})\neq 0\),
it follows that
\((c_{\alpha,v},c_{\alpha,t},c_{\beta,u})\) is
an alternating directed path of type \((j,i)\).

Since \(\Gamma\) is admissible, if \(\mu(c_{\beta,u},c_{\alpha,t})\neq 0\) then
\(\mu(c_{\beta,u},c_{\alpha,t})> 0\). So it follows that
\(N^{2}_{j,i}(\Gamma;v,u)>0\), and so \(N^{2}_{i,j}(\Gamma;v,u)>0\), since
\(\Gamma\) satisfies the \(W_n\)-Polygon Rule.
Thus there exists at least one \((\delta,x_1)\in\mathcal{I}\times\STD(\lambda)\)
such that \((c_{\alpha,v},c_{\delta,x_1},c_{\beta,u})\)
is an alternating directed path of type \((i,j)\). If \(\delta\neq\alpha\) then
\(v\in\Ini_{\alpha}(\Gamma)\). Now since \(\morethan {t'}i=\morethan ti\), we have
\(v<_{\lex}t'\) by Lemma~\ref{xlexlessthant}. This, however, contradicts the definition
of~\(t'\). Hence \(\delta=\alpha\), and \(x_1\in\Ini_{\alpha}(\Gamma)\). Now
Theorem~\ref{inductivestep} shows that either \(x_1=s_iv\) and \(i\in\SA(v)\), or else
\(x_1<v\). But \(x_1<_{\lex}t'\) by Lemma~\ref{xlexlessthant} (i)
in the former case, and \(x_1<_{\lex}t'\) by Lemma~\ref{xlexlessthant} (ii) in
the latter case. Both alternatives contradict the definition of~\(t'\),
thus showing that \(i<j-1\) is impossible.

Suppose now that \(i=j-1\). By Lemma~\ref{existencef}, we have
\(\col_{t}(j-1)\neq\col_{t}(j+1)\), and it follows that either
\(\col_{t}(j-1)<\col_{t}(j+1)\) or \(\col_{t}(j-1)>\col_{t}(j+1)\).

\begin{Case}{1.} Suppose that \(\col_{t}(j-1)<\col_{t}(j+1)\).
The result follows by the same argument as above,
with \(j-1\) replacing \(i\) and Lemma~\ref{variousaltpaths} (ii)
replacing Lemma~\ref{variousaltpaths} (i).
\end{Case}

\begin{Case}{2.} Suppose that \(\col_{t}(j-1)>\col_{t}(j+1)\). Then
\(j-1\in\SD(v)\), and \(j-1\in\D(v)\), and \(j\notin\D(v)\) by
Lemma~\ref{variousaltpaths} (iii). Note that since \(j-1\in\SD(v)\),
we have \(s_{j-1}v\in\STD(\lambda)\) and \(s_{j-1}v<v\). Let \(w=s_{j-1}v\). It
follows by Lemma~\ref{variousaltpaths} (iii) that
\(j-1,j\notin\D(w)\).

Next, since \(\mu(c_{\alpha,c},x_{\alpha,w})=\mu(c_{\alpha,t},c_{\alpha,v})=1\)
by Corollary~\ref{arweightone} and since \(\mu(c_{\beta,u},c_{\alpha,t})\neq 0\),
it follows that
\((c_{\alpha,w},c_{\alpha,v},c_{\alpha,t},c_{\beta,u})\) is
an alternating directed path of type \((j-1,j)\).

Since \(\Gamma\) is admissible, if \(\mu(c_{\beta,u},c_{\alpha,t})\neq 0\) then
\(\mu(c_{\beta,u},c_{\alpha,t})> 0\). So it follows that
\(N^{3}_{j-1,j}(\Gamma;w,u)>0\), and so \(N^{3}_{j,j-1}(\Gamma;w,u)>0\), since
\(\Gamma\) satisfies the \(W_n\)-Polygon Rule.
Thus there exists at least one \((\delta,x_1)\in\mathcal{I}\times\STD(\lambda)\) and one
\((\gamma,x_2)\in\mathcal{I}\times\STD(\lambda)\)
such that \((c_{\alpha,w},c_{\delta,x_1},c_{\gamma,x_2},c_{\beta,u})\)
is an alternating directed path of type \((j,j-1)\).
If \(\delta\neq\alpha\) then \(w\in\Ini_{\alpha}(\Gamma)\).
Now since \(\morethan t{(j-1)}=\morethan {t'}{(j-1)}\), we have \(w<_{\lex}t'\) by
Lemma~\ref{xlexlessthant} (iii). This, however,
contradicts the definition of~\(t'\).
Therefore, \(\delta=\alpha\).

Since
\(\D(x_1)\cap\{j-1,j\}=\{j\}\) and \(\D(x_2)\cap\{j-1,j\}=\{j-1\}\),
and \(\mu(c_{\gamma,x_2},c_{\delta,x_1})\neq 0\), it follows from
the \(W_n\)-Simplicity Rule that \(\{c_{\delta,x_1},c_{\gamma,x_2}\}\) is a simple edge.
Thus \(\gamma=\delta\), and \(x_1\) and \(x_2\) are related by a dual Knuth move. Thus
\(x_2\) is the \((j-1)\)-neighbour of \(x_1\). We see that \(x_2\in\Ini_{\alpha}(\Gamma)\),
and it will suffice to show that \(x_2<_{\lex}t'\), contradicting the definition of~\(t'\).

By Theorem~\ref{inductivestep} either \(x_1=s_jw>w\) or \(x_1<w\). If
\(x_1<w\) then since \(\morethan t{(j-1)}=\morethan {t'}{(j-1)}\), the conclusion
\(x_2<_{\lex}t'\) follows from Lemma~\ref{xlexlessthant}~(iv).
We are left with the case \(x_1=s_jw>w\). This gives \(j\in\SA(w)\), and we see
that the conditions of Lemma~\ref{variousaltpaths}~(iii) are satisfied: we
have \(v=s_jt\) with \(j\in\SD(t)\) and \(\col_t(j+1)<\col_t(j-1)\), and \(w=s_{j-1}v\).
Since \(j\in\SA(w)\) it follows that \(j-1\in\SA(x_1)\), and \(s_{j-1}x_1\)
is the \((j-1)\)-neighbour of~\(x_1\). Thus \(x_2=s_{j-1}x_1=s_{j-1}s_jw\), and
since \(\morethan t{(j-1)}=\morethan {t'}{(j-1)}\), we have
\(x_2<_{\lex}t'\) by Lemma~\ref{xlexlessthant}~(iii).
\qedhere
\end{Case}
\end{proof}

\begin{rema}
Since \(\alpha\neq\beta\) and \(\mu(c_{\beta,u'},c_{\alpha,t'})\neq 0\),
it follows by Remark~\ref{nonemptyAkst}
that \(A(u',t')\neq\emptyset\). Let \((u,t)\in A(u',t')\), noting that \((u,t)\in F(u',t')\).
Then \(\Morethan {t'}i\) is \(i\)-critical. (The proof is very much the same as that
for Proposition~\ref{cellorder}.)
Since \(i<j\) (as shown above) and \(j<i+2\), since \(\morethan t{(i+1)}\) is minimal,
we have \(j=i+1\). Definition~\ref{m-critical} says that \(\col_{t}(i+2)=\col_{t}(i)\).
But \(\col_{t}(i+2)\neq\col_{t}(i)\) by Lemma~\ref{existencef}. This contradiction provides
an alternative proof for Proposition~\ref{admcellaremonomol}.
\end{rema}

We are now in a position to state and prove the main result of the paper.
\begin{theo}\label{mainresult}
Admissible cells of type \(A_{n-1}\) are Kazhdan--Lusztig.
\end{theo}

\begin{proof}
Let \(\Gamma = \Gamma(C,\mu,\tau)\) be an admissible \(W_n\)-cell, and let
\(\Lambda\) be the set of molecule types for \(\Gamma\). By Lemma~\ref{samelambda},
\(\Lambda=\{\lambda\}\) for some \(\lambda\in P(n)\).
Let \(\mathcal{I}=\mathcal{I}_{\lambda}\) be the indexing set for
the molecules of \(\Gamma\), and let, for each \(\gamma\in\mathcal{I}\),
\(C_{\gamma}=C_{\gamma,\lambda} = \{c_{\gamma,w}\mid w\in\STD(\lambda)\}\)
be the vertex set of a molecule of \(\Gamma\). By Proposition~\ref{admcellaremonomol},
\(\Gamma = \bigsqcup_{\gamma\in\mathcal{I}}\Gamma(C_{\gamma})\),
where for each \(\gamma\in\mathcal{I}\), \(\Gamma(C_{\gamma})\) is
isomorphic to \(\Gamma_{\lambda}\). Since \(\Gamma\) is an admissible \(W_n\)-cell by
hypothesis, it follows that \(\mathcal{I}=\{\gamma\}\), whence
\(\Gamma=\Gamma(C_{\gamma})\) and \(\Gamma\) is isomorphic to \(\Gamma_{\lambda}\).
Since \(\Gamma_{\lambda}\) is isomorphic to \(\Gamma(C(\tau_{\lambda}))\), it follows from
Corollary~\ref{leftcelllambda} that \(\Gamma\) is isomorphic to a Kazhdan--Lusztig left cell.
\end{proof}

\begin{rema}\label{remontwosidedcells}
Let \(\lambda\in P(n)\) and let \(D(\lambda)=\bigsqcup_{t\in\STD(\lambda)}C(t)\),
the Kazhdan--Lusztig two-sided cell corresponding to \(\lambda\).
By Remark~\ref{orderedtwosided}, the singleton set \(\{\lambda\}\) is the set
of molecule types of the admissible \(W_n\)-graph \(\Gamma(D(\lambda))\). It follows
from Proposition~\ref{admcellaremonomol} that \(\Gamma(D(\lambda))\) is a disjoint union of
the Kazhdan--Lusztig left cells \(\Gamma(C(t))\). This implies the following well known
result (see, for example, \cite[Theorem 5.3]{gecpfei:charHecke}).
\begin{theo}\label{xyinsameleftcell}
Let \(\lambda\in P(n)\) and \(y,w\in D(\lambda)\). If \(y\preceq\lside w\)
then \(y,w\in C(t)\) for some \(t\in\STD(\lambda)\).
\end{theo}
\end{rema}

\section*{Acknowledgement}

I am grateful to A/Prof Robert B. Howlett for providing many improvements to the
exposition of this paper.

\end{document}